\documentclass[12pt]{amsart}
\usepackage{amsfonts}
\usepackage{mathrsfs}
\usepackage{bbm}%
\usepackage{bm}
\usepackage{bbding}

\usepackage[mathcal]{euscript}
\usepackage{graphicx}
\usepackage{amsthm,amscd}
\usepackage{calc}
\usepackage{mathtools}
\usepackage{mathrsfs,dsfont}
\usepackage{CJK,fancyhdr,amscd}
\usepackage[dvipsnames]{xcolor}
\usepackage{anysize}
\usepackage{amsmath,amssymb,amsfonts}
\usepackage{epsfig,enumerate}
\usepackage{latexsym}
\usepackage{indentfirst, latexsym}
\usepackage{graphics}
\usepackage{tikz-cd}
\usepackage[right]{lineno}
\usepackage[all,poly,knot]{xy}
\usepackage{xpatch}
\usepackage[pagebackref]{hyperref}

\hypersetup{
 	colorlinks=true,
 	linkcolor=Blue,
 	filecolor=Yellowblue,
 	citecolor = WildStrawberry,      
 	urlcolor=cyan,
}

\xpatchcmd{\step}{%
	\normalfont\scshape\centering}{%
	\normalfont\scshape}{\typeout{Success}}{\typeout{Failure}}%

\allowdisplaybreaks[2]

\providecommand{\U}[1]{\protect\rule{.1in}{.1in}}

\numberwithin{equation}{section}

\newtheorem{theorem} {Theorem} [section]
\newtheorem{proposition}[theorem]{Proposition}
\newtheorem{corollary}  [theorem]     {Corollary}
\newtheorem{lemma}  [theorem]     {Lemma}
\theoremstyle{definition}
\newtheorem{example}  [theorem]     {Example}

\newtheorem{step}{Step}
\theoremstyle{plain}

\newtheorem{definition}  [theorem]     {Definition}

\newtheorem{conjecture}  [theorem]     {Conjecture}

\newtheorem*{thm A*}{Theorem A}
\newtheorem*{thm B*}{Theorem B}
\newtheorem*{thm C*}{Theorem C}
\newtheorem*{thm D*}{Theorem D}

\theoremstyle{definition}
\newtheorem{remark}  [theorem]     {Remark}

\newcommand{\vol}{\mathrm{vol}}

\newcommand{\btheorem}{\begin{theorem}}
	\newcommand{\etheorem}{\end{theorem}}
\newcommand{\bproposition}{\begin{proposition}}
	\newcommand{\eproposition}{\end{proposition}}
\newcommand{\bdefinition}{\begin{definition}}
	\newcommand{\edefinition}{\end{definition}}
\newcommand{\bcorollary}{\begin{corollary}}
	\newcommand{\ecorollary}{\end{corollary}}
\newcommand{\bproof}{\begin{proof}}
	\newcommand{\eproof}{\end{proof}}
\newcommand{\beq}{\begin{equation}}
	\newcommand{\eeq}{\end{equation}}
\newcommand{\ee}{\end{eqnarray*}}
\newcommand{\be}{\begin{eqnarray*}}

\newcommand{\elemma}{\end{lemma}}
\newcommand{\blemma}{\begin{lemma}}

\newcommand{\bbeta}{\boldsymbol{\beta}}

\newcommand{\R}{\mathbb R}
\newcommand{\Q}{\mathbb Q}

\newcommand{\bd}{\begin{enumerate} }
	\newcommand{\ed}{\end{enumerate}}

\usepackage{adjustbox}
\usepackage{lineno}
\usepackage{float}

\begin{document}

\title{On Pseudo-Effectivity and Volumes of Adjoint Classes in K\"ahler Families with Projective Central Fiber}

\begin{abstract}
This paper is devoted to studying the deformation behavior of pseudo-effective canonical divisors and volumes of adjoint classes in K\"ahler families. Based on recent developments in the K\"ahler minimal model program, for flat families with fiberwise canonical singularities, we establish the global stability of the pseudo-effectivity of canonical divisors and uniruledness, assuming in addition that one fiber is projective, while the same conclusion for K\"ahler threefolds is also true without the projectivity assumption of the central fiber.  For a smooth K\"ahler family whose central fiber is projective with a big adjoint class, we show that its volume remains locally constant. Finally, using the (relative) minimal model program for K\"ahler threefolds, we verify the deformation invariance of volumes of adjoint classes and plurigenera for smooth K\"ahler families of threefolds, thereby partially confirming Siu's invariance of plurigenera conjecture in dimension three.
\end{abstract}
	
	\author{Christopher D. Hacon}
	\address{Department of Mathematics, University of Utah, Salt Lake City, UT 84112,
USA}
\email{hacon@math.utah.edu}
 \author{Yi Li}
    \address{Department of Mathematics, Wuhan University, Hubei, Wuhan 430074, China}
    \email{yilimath@whu.edu.cn}
    \author{Sheng Rao}
    \address{Department of Mathematics, Wuhan University, Hubei, Wuhan 430074, China}
    \email{likeanyone@whu.edu.cn}
	\thanks{The first author was partially supported by  NSF research grant   DMS-2301374 and
by a grant from the Simons Foundation SFI-MPS-MOV-00006719-07. The second and third authors were partially supported by NSFC (Grant No. 12271412, W2441003) and Hubei Provincial Innovation Research Group Project (Grant No. 2025AFA044).}
    \keywords{K\"ahler minimal model program, volume, pseudo-effectivity}
	
	\maketitle
	\setcounter{tocdepth}{1}
	\tableofcontents

\section{Introduction}\label{Introduction}

One of the central problems in complex algebraic geometry is the classification of compact complex varieties. Plurigenera, volumes, and positivity properties are fundamental invariants associated to a variety and play a crucial role in classification problems. The aim of this paper is to study the deformation behavior of pseudo-effective canonical divisors, uniruledness, plurigenera and volumes of adjoint classes. 

The first part of this paper is devoted to studying the deformation behavior of pseudo-effective canonical divisors and uniruledness. Given a smooth projective family, H. Tsuji \cite{Tsuji02} showed that if the central fiber has a pseudo-effective canonical divisor, then the nearby fibers also have pseudo-effective canonical divisors. It is natural to ask whether similar results hold for K\"ahler families. Recently, W. Ou \cite{Ou25} proved the BDPP conjecture \cite[Conjecture 1.2]{BDPP}, which characterizes  uniruledness of K\"ahler varieties.
\begin{theorem}[{\cite[Theorem 1.1]{Ou25}}]\label{Theorem Ou25}
Let $X$ be a compact K\"ahler manifold. Then $X$ is uniruled if and only if the canonical line bundle $K_X$ is not pseudo-effective.
\end{theorem}
Combining this with Fujiki--Levine's deformation stability of uniruledness \cite{Fujiki,Levine}, one obtains the deformation stability of pseudo-effectivity of canonical divisors for smooth K\"ahler families (cf. \cite[Corollary 5.5]{CP25}). 

Combining recent developments in the transcendental MMP \cite{DH24MMP}, Das--Hacon's transcendental base point free theorem for projective varieties \cite{DH24MMP}, and Koll\'ar's extension of MMP techniques in \cite{DefGeneralType}, we prove the local stability of pseudo-effective canonical divisors for K\"ahler families whose fibers are allowed to have canonical singularities in Theorem \ref{Deformation invariance of pseudo-effectivity}. Then based on Theorem \ref{Deformation invariance of pseudo-effectivity}, and using a relative Douady space argument in Proposition \ref{pseflocalglobal}, we can prove the following global stability result for the pseudo-effectivity of canonical divisors. 
\begin{theorem}[Theorem \ref{GlobalPseff}]\label{MainPSEF}
Let $f: X \to S$ be a family from a K\"ahler variety $X$ onto a smooth, connected, and relatively compact curve $S$. Assume that $X_0$ is projective, and $X_t$ have canonical singularities for all $t\in S$. Then $K_{X_0}$ is pseudo-effective if and only if $K_{X} $ is pseudo-effective over $S$.
\end{theorem}

With a similar approach to Theorem \ref{MainPSEF} and the characterization of uniruledness for K\"ahler varieties with canonical singularities given in Theorem~\ref{Theorem Ou25}, we reprove the global deformation stability of uniruled structures with canonical singularities.
\begin{corollary}[\cite{DefGeneralType,Kollar96}, Corollary \ref{DefUniruled}]\label{MainDefUniruled}
Let $f: X \to S$ be a family from a K\"ahler variety $X$ onto a smooth, connected, and relatively compact curve $S$. Assume that $X_0$ is projective, and $X_t$ have canonical singularities for all $t\in S$. Then $X_0$ is uniruled if and only if $X_t$ are uniruled for all $t \in S \setminus \{0\}$.
\end{corollary}

Note that if the MMP and the transcendental base point free conjecture hold on the central fiber, then the same proofs of Theorem~\ref{MainPSEF} and Corollary~\ref{MainDefUniruled} apply. Consequently, the same conclusions of Theorem~\ref{MainPSEF} and Corollary~\ref{MainDefUniruled} for K\"ahler families of threefolds remain valid without assuming that the central fiber is projective.
It is reasonable to propose the following more general conjecture:
\begin{conjecture}Let $f: X \to S$ be a family from a  K\"ahler variety $X$ onto a smooth, connected, and relatively compact curve $S$. Assume that $X_t$ have canonical singularities for all $t\in S$. Then $K_{X_0}$ is pseudo-effective if and only if $K_{X_t} $ is pseudo-effective for all $t \in S \setminus \{0\}$, and $X_0$ is uniruled if and only if $X_t$ are uniruled for all $t \in S \setminus \{0\}$.
\end{conjecture}

Next, we turn our attention to the deformation invariance problem for plurigenera and volumes of adjoint classes. There is a long history of results related to the deformation invariance of plurigenera for projective or Moishezon morphisms, see \cite{Siu98,Kawamata99,Kawa99,Siu02,Tsuji02,Naka04,Taka06,Paun07,BP12,HMX13,FS20,RT1,RT2,LW26} etc. In \cite{Siu98,Siu02}, Siu proved the deformation invariance of plurigenera for smooth projective families of manifolds using Ohsawa--Takegoshi extension techniques. 
Later, Kawamata \cite{Kawa99} and Nakayama \cite{Naka04} gave algebraic proofs of a singular version of Siu's theorem for projective families of varieties of general type with canonical singularities and terminal singularities, respectively. 
Takayama \cite{Taka06} proved the invariance of plurigenera for Moishezon morphisms whose fibers have canonical singularities by modifying the Ohsawa--Takegoshi extension theorem. Rao--Tsai \cite{RT1,RT2} developed a bimeromorphic embedding theorem, and using this, they proved that smooth fiberwise Moishezon families and flat Moishezon families with uncountably many fibers of general type with canonical singularities both satisfy the deformation invariance of plurigenera. A (possibly incomplete) introduction on the progress of the deformation invariance of plurigenera can be found in \cite[Introduction]{RT2}.

Nevertheless, the problem is widely open for smooth K\"ahler families, as conjectured by Siu. See the recent progress in \cite{CP23Pluri}. 
\begin{conjecture}[{\cite[Conjecture 0.4]{Siu02}, \cite[Conjecture 1.4]{CP23Pluri}}]\label{SiuConj}
Let $\pi: X \rightarrow \Delta$ be a smooth family of compact K\"ahler manifolds over the unit disk in $\mathbb{C}$ with fiber $X_t$. Then for any positive integer $m\ge 1$, the $m$-genus $
P_m\left(X_t\right):=\operatorname{dim} H^0\left(X_t, m K_{X_t}\right)
$ is independent of $t\in \Delta$.
\end{conjecture}

The cohomological methods used in the projective setting do not directly apply to transcendental classes on a K\"ahler variety. However, thanks to the recent developments in \cite{DH24MMP,DHY23}, we are able to treat transcendental classes in a way similar to boundary divisors. Using these techniques, we establish the deformation invariance of the volumes of adjoint classes under the assumption that the central fiber is projective with a big adjoint class.
\begin{theorem}[{Theorem \ref{InvarVol}}]\label{0InvarVol}
Let $f: X \to S$ be a smooth family from a K\"ahler manifold $X$ onto a smooth, connected, and relatively compact curve $S$. For a point $0\in S$, assume that $(X,B + \bbeta)$ is a generalized K\"ahler pair such that the boundary divisor $B$ does not contain the fiber $X_0$ and let $\omega$ be a K\"ahler class on $X$ with $\omega_{X_t}:=\omega|_{X_t}$ its restriction to the fiber $X_t$. 

Assume that there exists $\delta>0$ such that $K_{X_0} + B_0 +  \boldsymbol{\beta}_{X_0}+\delta {\omega}_{X_0}$ is big and $(X_0, B_0+  { \boldsymbol{\beta}}_{X_0}+ \delta\boldsymbol{\omega}_{X_0})$ is projective, with canonical singularities, $\lfloor{ B_0\rfloor} = 0$ and $N(K_{X_0 }+B_0 + \bbeta_{X_0})\wedge B_0  = 0$. Then the volume  function 
$$t \longmapsto \mathrm{vol}(K_{X_t}+B_t +  \boldsymbol{\beta}_{X_t} +\delta{\omega}_{X_t})$$ 
is constant on a Euclidean open neighborhood $U \subset S$ of $0$.
\end{theorem}

We also obtain the following variant under the assumption that $(X,B)$ is log smooth over $S$, without assuming that $N(K_{X_0} + B_0 + \bbeta_{X_0})\wedge B_0  = 0$. 
\begin{theorem}[Theorem \ref{InvarVol2}]\label{0InvarVol2}
Let $f:X\to S$ be a smooth family from a K\"ahler manifold $X$ onto a smooth, connected, relatively compact curve $S$. Assume that $(X,B+\bbeta)$ is a generalized klt pair such that $(X,B)$ is log smooth over $S$ and $\bbeta =\overline \beta $ where $\beta$ is nef over $S$, and $X_0$ is projective. If $K_{X_0}+B_0+\bbeta _{X_0}$ is big, then the volume function
\[t\longmapsto {\rm vol}(K_{X_t}+B_t+\bbeta _{X_t})\]
is constant on some Euclidean neighborhood $U\subset S$ of $0$.
\end{theorem}

Since the (relative) minimal model program and the abundance conjecture are completely established for K\"ahler threefolds by the series of works \cite{HP16,CHP16,CHP23Erratum,DHY23,DH24,DHP24,DO24,DO24II,DH20}, etc, we prove the deformation invariance of volumes (as well as plurigenera) for smooth families of K\"ahler threefolds without any additional assumption on $X_0$. In particular, this partially resolves Conjecture \ref{SiuConj} in the case of threefolds.
\begin{theorem}[Theorem \ref{DefVol3fold}]\label{0DefVol3fold}
Let $f: X \to S$ be a smooth K\"ahler family of   threefolds over a smooth, connected, and relatively compact curve $S$. Then for any integer $m\ge 1$, the $m$-genus $P_m(X_t)$ is constant for any $t\in S$. Furthermore, assume that there exists a K\"ahler class $\omega_X$ on $X$ such that $\omega_{X_t} :  =  \omega_X |_{X_t}$. Then for every $\delta\ge 0$, the function $$t \longmapsto \operatorname{vol}(K_{X_t}+\delta\,\omega_{X_t})$$is constant for any $t\in S$.
\end{theorem}
Let us highlight some key points of our proofs. To study the deformation stability of pseudo-effectivity of canonical divisors in Theorem \ref{MainPSEF}, we will make use of some MMP arguments. 
A key complication is that, without the projectivity assumption of the family, it is not a priori clear whether the classical MMP framework can be employed here. However, because we assume that the central fiber is projective, thanks to \cite{DH24MMP}, it is possible to run the MMP on the central fiber. We then use the extension technique from \cite[Theorem 2]{DefGeneralType} as in Proposition \ref{ExtendMMP} to extend the MMP to nearby fibers. Note that the transcendental base point free Theorem \ref{DHTransbpf} plays an important role in guaranteeing the termination of the MMP for nearby fibers, provided termination holds for the central fiber.

In order to prove the deformation invariance of the volumes in Theorem \ref{0InvarVol}, we make use of a MMP type argument via \cite{DH24MMP} as in the proof of Theorem \ref{MainPSEF}. Since we assume that the central fiber has a big adjoint class, the MMP with scaling on the central fiber terminates. By base point freeness, the MMP also terminates on nearby fibers. Hence, without loss of generality, we may assume that all fibers have nef adjoint class, so that the volume can be computed via top intersection numbers. In the case of divisors, deformation invariance of the volume follows from the constancy of the Hilbert polynomial in flat families; however, we do not have a good substitute for this in the case of transcendental classes. Fortunately, after running the MMP, the singularities are still concentrated on a codimension 2 subset, and this allows us to apply a singular version of Stokes' formula to deduce the constancy of the volume for adjoint classes. Using a similar argument, we can prove the deformation invariance of the volume of adjoint classes for smooth families of K\"ahler threefolds as in Theorem \ref{0DefVol3fold}.

In this paper, a \emph{family} $\pi: X\to S$ will refer to a proper surjective morphism from a normal complex analytic space
to a smooth connected curve whose fibers $X_t:=\pi^{-1}(t)$ $(t\in S)$ are connected (possibly reducible) complex analytic spaces, as in \cite{CRC25}. In particular, we call $\pi$ a \emph{smooth family} if it is a flat \emph{submersion}  (i.e., the relative cotangent sheaf $\Omega^1_{X|S}$ is locally
free of rank $\dim X -\dim S$). Note that in a smooth family, both the total space and each fiber are smooth. We also assume that the general fiber of a family is irreducible. The “general fiber” of a family refers to the one over a nonempty analytic Zariski open subset of the base. 

\textbf{Organization of the paper.} 
The paper is organized as follows. Section \ref{Preliminaries} establishes the foundations by reviewing generalized pairs, positivity notions such as nef, big, and pseudo-effective classes, the volume of cohomology classes, and basic results on bigness and pseudo-effectivity in families. Section \ref{Section: SingandPositivty} introduces singularities and positivity of line bundles in families, including deformation of singularities, relative ampleness, relative pseudo-effectivity, and criteria for global generation and semi-ampleness. Section \ref{Section: ExtMMP} develops the main tools needed for the main results, namely the extension of the analytic minimal model program to nearby fibers and the deformation openness of nef and big adjoint classes. Section \ref{Section: Deformation pseff} proves the global stability of  pseudo-effectivity of canonical divisors and uniruledness in K\"ahler families with one projective fiber. Section \ref{Section: DeformationVol} establishes invariance of the volumes of adjoint classes, first when the projective central fiber has a big adjoint class and then, in the case of K\"ahler threefolds without projectivity assumption. We call a family $\pi: X\to S$ a \emph{K\"ahler family} if the total space $X$ is a K\"ahler variety.\\
\textbf{Acknowledgment.} 
The authors would like to express their gratitude to Professor Mihai P\u{a}un for informing us of the singular version of Stokes’ formula, Professors J\'anos Koll\'ar, Wenhao Ou and Omprokash Das for answering our questions and pointing out several errors in our preliminary version, and Professor Song Sun for bringing Junsheng Zhang’s preprint \cite{JunSheng} to our attention. The last two authors also thank Professors Junyan Cao and Xiaojun Wu for a useful discussion on pseudo-effectivity and uniruledness during the conference ``Families of K\"ahler spaces" at CIRM, France in April 2025. Finally, the authors sincerely acknowledge the reviewer for the thoughtful comments and suggestions, which have greatly improved the quality of this paper.
\section{Preliminaries: generalized K\"ahler pairs and positivities of classes}\label{Preliminaries}
We first introduce some preparatory notions and results to be used in the proofs of our main theorems in Sections \ref{Section: Deformation pseff} and \ref{Section: DeformationVol}.
\subsection{Generalized pairs}
Generalized pairs were first introduced in \cite{BZ16} in the study of the effective Iitaka fibration conjecture and later generalized to the K\"ahler setting in \cite{DHY23}. With this additional flexibility, we can run the transcendental minimal model program for generalized K\"ahler pairs. Let us first recall some basic definitions. 

Let $X$ be a \emph{complex analytic variety}, i.e., an irreducible reduced complex space. Let $\mathcal{H}_X$ be the sheaf of real parts of holomorphic functions multiplied with $\sqrt{-1}$. A $(1,1)$-form (resp. $(1,1)$-current) with local potentials on $X$ is a global section of the quotient sheaf $\mathcal{A}_X^0 / \mathcal{H}_X$ (resp. $\mathcal{D}_X / \mathcal{H}_X$), where $\mathcal{A}_X^0$ is the sheaf of smooth real valued functions and $\mathcal{D}_X$ is the sheaf of  distributions. We define the \emph{Bott--Chern cohomology} to be $$H^{1,1}_{\rm{BC}}(X,\mathbb{R}) : = H^1(X,\mathcal{H}_X).$$

Given a closed $(1,1)$-form $\alpha$ with local potentials (resp.\ a closed $(1,1)$-current $T$ with local potentials), we denote by $[\alpha]$ (resp.\ $[T]$) the associated Bott--Chern class. Without causing confusion, given an $\mathbb{R}$-Cartier divisor $D$, we will also denote by $[D]$ the associated cohomology class in $H^{1,1}_{\mathrm{BC}}(X,\mathbb{R})$. When there is no danger of ambiguity, we will sometimes write a class $\alpha \in H^{1,1}_{\mathrm{BC}}(X,\mathbb{R})$ without explicitly specifying its representative.

A complex analytic variety $X$ is a \emph{K\"ahler variety} if there exists a K\"ahler form $\omega$ on $X$, i.e., a positive closed real $(1,1)$-form $\omega \in \mathcal{A}_{\mathbb{R}}^{1,1}(X)$ such that for every singular point $x \in X$ there exists an open neighborhood $U \subseteq X$ of $x$ and a closed embedding $\iota_U: U \hookrightarrow V$ into an open subset $V \subseteq \mathbb{C}^N$, and a strictly plurisubharmonic $C^{\infty}$ function $f: V \rightarrow \mathbb{R}$ with $\left.\omega\right|_{U \cap X_{\mathrm{sm}}}=\left.(i \partial \bar{\partial} f)\right|_{U \cap X_{\mathrm{sm}}}$.

Let $f: X \rightarrow S$ be a proper morphism of normal K\"ahler varieties with relatively compact $S$, and $\beta$  a closed $(1,1)$-current with local potentials. We say that the class $[\beta] \in H_{\mathrm{BC}}^{1,1}(X,\mathbb{R})$ is \emph{relatively K\"ahler} (also referred to as \emph{$f$-K\"ahler} or \emph{K\"ahler over $S$}) if $\left[\beta+f^* \omega_S\right] \in H_{\mathrm{BC}}^{1,1}(X,\mathbb{R})$ is a K\"ahler class for some K\"ahler form $\omega_S$ on $S$.

Let $X$, $Y$ and $S$ be complex analytic varieties. A mapping $\varphi$ of $X$ into the power set of $Y$ over $S$ is called a \emph{meromorphic $S$-mapping} of $X$ into $Y$ (we shall write $\varphi: X \dashrightarrow Y/S$), if $\varphi$ satisfies the following conditions:\\
(1) The graph $G_{\varphi}=\{(x, y) \in X \times_S Y \mid y \in \varphi(x)\}$ of $\varphi$ is an irreducible analytic subset in ${X} \times_S {Y}$;\\
(2) The projection map ${p}_{{X}}: {G}_{\varphi} \to {X}$ is a proper modification. \\
A meromorphic $S$-mapping $\varphi: {X} \dashrightarrow {Y}/S$ of complex varieties is called a \emph{bimeromorphic $S$-mapping} if ${p}_{{Y}}: {G}_{\varphi} \to {Y}$ is also a proper modification. 

In what follows, without causing any confusion, we will use the terms `birational' and `bimeromorphic' interchangeably in the analytic setting.

A compact complex analytic variety is said to be in the \emph{Fujiki class $\mathcal{C}$} if it is bimeromorphic to a compact K\"ahler manifold. A compact complex analytic variety is said to be a \emph{Moishezon variety} if it is bimeromorphic to a projective manifold. 

Let $X$ be a normal complex analytic variety. A sheaf $\mathcal{F}$ on $X$ is called a \emph{divisorial sheaf} if $\mathcal{F}$ is a reflexive sheaf of rank~$1$. The set of all divisorial sheaves, denoted by $W(X)$, forms a group under the reflexive tensor product (that is, for $\mathcal{F}, \mathcal{G} \in W(X)$, we define their \emph{reflexive tensor product} by
$\mathcal{F} \,\hat{\otimes}\, \mathcal{G} := (\mathcal{F} \otimes \mathcal{G})^{**} \in W(X)$). A normal complex analytic variety $X$ is said to be \emph{strongly $\mathbb{Q}$-factorial} if every divisorial sheaf is $\mathbb{Q}$-Cartier; that is, for every $\mathcal{F} \in W(X)$ there exists $m \in \mathbb{N}$ such that the reflexive power $\mathcal{F}^{[m]}$ is locally free.

As explained in \cite[II.2, p. 54]{Naka04}, for an $n$-dimensional normal complex analytic variety $X$, the \emph{canonical sheaf} $\omega_X$ is the unique reflexive sheaf whose restriction to $X_{\text {reg }}$ is isomorphic to the sheaf $\Omega_{X_{\text {reg }}}^n$ of germs of holomorphic differential $n$-forms. A non-zero meromorphic $n$-form $\eta$ on $X_{\text {reg }}$ is regarded as a meromorphic section of $\omega_X$. The associated divisor $\operatorname{div}(\eta)$ is called the \emph{canonical divisor} of $X$ and is denoted by $K_X$ even though it depends on the choice of $\eta$. Some complex analytic varieties do not admit any non-zero meromorphic section of $\omega_X$. However, we use the symbol $K_X$ as a formal divisor class with an isomorphism $\mathcal{O}_X(K_X)\cong \omega_X$ and call it the \emph{canonical divisor} of $X$. 

\begin{definition}[{Generalized pair, \cite[Definition 2.7]{DHY23}}]
Let $f: X \rightarrow S$ be a proper surjective morphism from a normal complex variety $X$ to a relatively compact complex variety $S$, $\nu: X^{\prime} \rightarrow X$ a resolution, and  $B^{\prime}$ an $\mathbb{R}$-divisor on $X^{\prime}$ with simple normal crossing support such that $B:=\nu_* B^{\prime} \ge 0$, and $ \boldsymbol{\beta}$ a closed b-$(1,1)$ current. We say that $(X, B+ { \boldsymbol{ \beta}})$ is a \emph{generalized pair} if\\
\emph{(1)} $ { \boldsymbol{\beta}}$ is a real closed b-$(1,1)$ current that descends to $X^{\prime}$,\\
\emph{(2)} $\left[ { \boldsymbol{\beta}}_{X^{\prime}}\right] \in H_{\mathrm{BC}}^{1,1}\left(X^{\prime},\mathbb{R}\right)$ is nef over $S$, and\\
\emph{(3)} $\left[K_{X^{\prime}}+B^{\prime}+ { \boldsymbol{\beta}}_{X^{\prime}}\right]=\nu^* \gamma$ for some $\gamma \in H_{\mathrm{BC}}^{1,1}(X,\mathbb{R})$.\\
If additionally $X$ is a K\"ahler variety, then we call $(X, B + \boldsymbol{\beta})$ a \emph{generalized K\"ahler pair}. Here a \emph{b-$(1,1)$ current} is a collection of closed bi-degree $(1,1)$ currents $\boldsymbol{\beta}_{X^{\prime}}$ on all proper bimeromorphic models $X^{\prime} \rightarrow X$ such that if $p: X_1 \rightarrow X_2$ is a bimeromorphic morphism of proper models of $X$, then $p_* \boldsymbol{\beta}_{X_1}=\boldsymbol{\beta}_{X_2}$. We say that $\bbeta$ \emph{descends to $X'$} or equivalently $\bbeta =\overline { \bbeta _{X'}}$ if on any higher model $\pi:X''\to X'$ we have $\bbeta_{X''}=\pi^* \bbeta_{X'}$. 
\end{definition}
Following the usual convention, we will sometimes denote by $\boldsymbol{\beta}_X$ the trace of the b-$(1,1)$ current $\boldsymbol{\beta}$ on $X$, by $\boldsymbol{\beta}_{X'}$ its trace on $X'$, and by $\boldsymbol{\beta}$ the actual b-$(1,1)$-current.

\begin{definition}[Generalized discrepancies]\label{logdiskrep}
Let $\left({X}, {B}+ { \boldsymbol{\beta}}\right)$ be a generalized pair. Let $E$ be a prime divisor on some birational model of ${X}$. We define the {generalized  discrepancy} of $E$ with respect to the above generalized pair as follows. Let $f:X'\to X$ be a log resolution, such that:\\ 
\emph{(a)} $B'$ is a $\mathbb{R}$-divisor with simple normal crossing support;\\ 
\emph{(b)} $ { \boldsymbol{\beta}}$ is the real closed b-(1,1) current that descends to $X'$ so that  $\bbeta _{X'}$ is nef and $\bbeta =\overline { \bbeta _{X'}}$;\\ 
\emph{(c)} $[K_{X'}+ B'+  { {\bbeta}}_{X'}]  =f^* \gamma$ for some $\gamma\in H^{1,1}_{\rm{BC}}(X, \mathbb{R})$. We can then write
$$
{K}_{{X}'}+{B}'+ { {\bbeta}}_{X'}=f^*\left({K}_{{X}}+{B}+ { {\bbeta}_X}\right).
$$
The \emph{generalized discrepancy} $a(X,B, { {\boldsymbol{\beta}}} , E)$ of $E$ is defined to be $-\text{coeff}_E(B')$.
\end{definition}

Using generalized discrepancies, we define various types of singularities for generalized pairs:

A generalized  pair $(X,B +  { { \boldsymbol{\beta}}} )$ is said to have \emph{canonical singularities} if 
the generalized discrepancy for any exceptional divisors satisfies $a(X,B ,  { \boldsymbol{\beta}} , E) \ge 0$. 

A generalized pair $(X,B +  { \boldsymbol{\beta}} )$ is said to have \emph{generalized klt singularities} (resp. \emph{generalized lc singularities}) or $(X,B +  { \boldsymbol{\beta}} )$ is said to be \emph{gklt} (resp. \emph{glc}) if 
the generalized discrepancy of any prime divisors $E$ over $X$ is $a(X,B ,  { \boldsymbol{\beta}} ,E)>-1$ (resp. $a(X,B, { \boldsymbol{\beta}};E)\ge -1$).

A generalized lc pair $(X,B +  { \boldsymbol{\beta}} )$ is said to have \emph{generalized dlt singularities} or $(X,B +  { \boldsymbol{\beta}} )$ is said to be \emph{gdlt} if there is an open subset $U \subset X$ such that $\left(U,\left.(B+\boldsymbol{\beta})\right|_U\right)$ is a log resolution (of itself) and $-1 \leq a(P, X, B+\boldsymbol{\beta}) \leq 0$ for any prime divisor $P$ on $U$ and $-1<a(P, X, B+\boldsymbol{\beta}) $ for any prime divisor $P$ over $X$ with center contained in $X \setminus U$.

When ${\boldsymbol{\beta}}$ descends to $X$ (i.e., $\boldsymbol{\beta}_X$ is nef and $\bbeta =\overline {\boldsymbol{\beta}_X} $), the term $ { \boldsymbol{\beta}}$ in the generalized pair $(X,B + \boldsymbol{\beta})$ does not contribute to the singularities. 
\begin{proposition}[{\cite[Lemma 2.13]{DHY23}}]
Let $(X,B+  { \boldsymbol{\beta}})$ be a generalized pair. Suppose that $ { \boldsymbol{\beta}}$ descends to $X$ and $K_X+B$ is $\mathbb{Q}$-Cartier, then $(X,B)$ is klt (resp. dlt, lc) if and only if $(X,B+ { \boldsymbol{\beta}})$ is gklt (resp. gdlt, glc).
\end{proposition}
\begin{proof}
The `if part' has been proved in \cite[Lemma 2.13]{DHY23}. The converse direction follows by 
\cite[Remark 2.6.(a)]{DHY23} combined with Definition \ref{logdiskrep}.
\end{proof}

When $ { \boldsymbol{\beta}}$ does not descend to $X$, it will contribute to the singularities (cf. \cite[Remark 4.2]{BZ16}).


\subsection{Nef, big and pseudo-effective classes} In this subsection, we will introduce nef, big and pseudo-effective Bott--Chern cohomology classes on a normal compact complex analytic variety and their basic properties. 

\begin{definition}[Metrically nef, quasi-nef and algebraically nef classes]
Let $X$ be a normal compact complex analytic variety in the Fujiki class $\mathcal{C}$, and let $\alpha = [\theta] \in H_{\mathrm{BC}}^{1,1}(X,\mathbb{R})$ be a class represented by a form $\theta$ with local potentials. Then \\
\emph{(1)} We say that the class $\alpha$ is \emph{(metrically) nef} if for some positive smooth $(1,1)$-form $\omega$ on $X$ and for every $\varepsilon>0$ there exists $f_\varepsilon \in \mathcal{A}^0(X)$ such that $$ \theta+i \partial \overline{\partial} f_\varepsilon \ge-\varepsilon \omega .$$
\emph{(2)} We say that the class $\alpha$ is \emph{quasi-nef} if there exists a K\"ahler modification $\nu \colon \widehat{X} \to X$ such that $\nu^*\alpha$ is a metrically nef class on $\widehat{X}$.\\
\emph{(3)} We say that the class $\alpha$ is \emph{algebraically nef} if for all irreducible curves $C \subset X$, we have  $\alpha\cdot C  \ge 0$.\\
\emph{(4)} Let $f:X\to S$ be a proper surjective morphism of normal relatively compact complex analytic varieties, and  $\alpha \in H^{1,1}_{\rm{BC}}(X,\mathbb{R})$ a Bott--Chern (1,1)-class on $X$. We say that $\alpha$ is \emph{$f$-metrically nef (resp. $f$-algebraically nef)} if $\alpha|_{X_s}$ is metrically nef (resp. algebraically nef) for any $s\in S$.
\end{definition}

Using the cone theorem for K\"ahler varieties \cite{CH20, HP24}, we first establish the equivalence of several notions of nefness for adjoint classes on K\"ahler varieties (cf. \cite[\S\ 2]{Naka86}).

\begin{proposition}\label{nef}
Let $X$ be a normal compact K\"ahler variety, such that $(X, B + \boldsymbol{\beta}_X)$ is a generalized klt pair, and $\alpha \in H^{1,1}_{\mathrm{BC}}(X, \mathbb{R})$. Then the class $\alpha$ is metrically nef if and only if it is quasi-nef. Moreover, if we further assume that $\alpha = K_X + B + \boldsymbol{\beta}_X$, then $\alpha$ is algebraically nef if and only if it is metrically nef.
\end{proposition}
    
\begin{proof}
By definition, $\alpha$ is quasi-nef if and only if $f^* \alpha$ is metrically nef on a resolution $f: \widehat{X} \to X$. On the other hand, by \cite[Theorem 1]{Paunnef},  $\alpha$ is metrically nef if and only if $f^* \alpha$ is metrically nef. Hence clearly $\alpha$ is metrically nef if and only if it is quasi-nef. 

Assume now that $\alpha = K_X + B + \bbeta_X$. By \cite[Theorem 0.5]{HP24} (or alternatively \cite[Theorem 1.3]{CH20}), together with Theorem~\ref{Theorem Ou25}, if $K_X + B + \bbeta_X$ is not metrically nef, then there exists a $(K_X + B + \bbeta_X)$-negative rational curve $C$ such that $
(K_X + B + \bbeta_X)\cdot C < 0$. 
This implies that $K_X + B + \bbeta_X$ is not algebraically nef. 
\end{proof}

\begin{definition}[Pseudo-effective and big classes]
On a normal compact complex analytic variety $X$, we say that:\\  
\emph{(1)} A cohomology class $\alpha \in H^{1,1}_{\rm{BC}}(X,\mathbb{R})$ is \emph{pseudo-effective} if there exists a closed positive (1,1)-current $T\ge 0$, such that $T$ represents the class $\alpha$.\\
\emph{(2)} A cohomology class $\alpha \in H^{1,1}_{\rm{BC}}(X,\mathbb{R})$ is \emph{big} if there exists a closed positive (1,1)-current $T$ to represent the class $\alpha$,  such that $T\ge \delta \omega$ for some smooth Hermitian metric $\omega$ and some number $\delta >0$. If $f:X\to Y$ is a bimeromorphic morphism and $\alpha \in H^{1,1}_{\rm{BC}}(X)$ is big, then we say that $f_*\alpha$ is \emph{modified big}.
\end{definition}

If we perturb a pseudo-effective class in the K\"ahler direction, we obtain a big class. 
\begin{proposition}\label{bigandpseff}
Let $X$ be a normal compact K\"ahler variety. A class $\alpha \in H^{1,1}_{\rm{BC}}(X,\mathbb{R})$ is pseudo-effective if and only if there exists a K\"ahler class $\omega$ such that $\alpha + \varepsilon \omega$ is big for any $\varepsilon >0$. 
\end{proposition}

\subsection{Lelong numbers and restricted non-K\"ahler locus}
Let us introduce the definitions of Lelong numbers and restricted non-K\"ahler locus on normal complex analytic varieties. 
\begin{definition}[{Lelong numbers, \cite[D\'efinition 2]{Demailly82}}]
Let $X$ be a normal complex $n$-dimensional analytic variety, and ${T} \geq 0$ a closed positive $(1,1)$-current on $X$. Let ${x} \in {X}$ be an arbitrary point. Consider an embedding
$$
({X}, {x}) \hookrightarrow\left(\mathbb{C}^{{N}}, 0\right).
$$
The coordinate functions $\left({z}_{{i}}\right)_{{i}=1,\cdots, {N}}$ on $\mathbb{C}^{{N}}$ restricted to $X$ induce a generating system $\left({g}_{{i}}\right)_{{i}=1,\cdots, {N}}$ of the maximal ideal of the ring $\mathcal{O}_{{X}, {x}}$. The function $\varphi:=\sum_{{i}}\left|{g}_{{i}}\right|$ is defined on some small open subset $U$ containing $x$. We then define the \emph{Lelong number} of $T$ at $x$ to be
$$
\nu({T},{x}):=\lim _{{r} \rightarrow 0} \frac{1}{(2 \pi {r})^{2 {n}-2}} \int_{\varphi<{r}} {~T} \wedge(\sqrt{-1} \partial \bar{\partial} \varphi)^{{n}-1}.$$
\end{definition}

\begin{definition}[{Restricted non-K\"ahler locus, \cite[Definition 4.18]{HP24}}]
Let $X$ be a normal complex analytic variety and $T$ a closed positive $(1,1)$-current on $X$ with local potentials. For each $c>0$ we define the \emph{Lelong level set} as 
$$
E_c:=\{x \in X: \nu(T, x)\ge c\},
$$
where $\nu(T, x)$ denotes the Lelong number of $T$ at $x \in X$.
This is an analytic subset of $X$. We define
$$
E_{+}(T):=\bigcup_{c \in \mathbb{R}^{+}} E_c .
$$
The \emph{restricted non-K\"ahler locus} of a class $\alpha\in H^{1,1}_{\rm{BC}}(X,\R)$ is the set
$$
E_{n K}^{a s}(\alpha):=\bigcap_{T \in \alpha} E_{+}(T),
$$
where $T$ above is assumed to be a K\"ahler current with weak analytic singularities, and $E_{+}(T) \subset X$ is the (analytic) subset of $X$ for which the Lelong numbers of $T$ are strictly positive. 
\end{definition}

\begin{lemma}\label{nonKahlerlocus}
Let $\alpha \in H^{1,1}_{\rm{BC}}(X,\mathbb{R})$ on a normal compact K\"ahler variety $X$. Then $E_{nK}^{as}(\alpha) = X$ if and only if $\alpha$ is not a big class; $E_{nK}^{as}(\alpha) = \emptyset $ if and only if $\alpha$ is a K\"ahler class.
\end{lemma}
{
\begin{proof}
When $X$ is a K\"ahler manifold, this follows from \cite[Theorem 3.17]{Boucksom04}. Let us prove that it also holds for normal compact K\"ahler varieties. If $\alpha$ is not big, then it contains no K\"ahler current, in particular no K\"ahler current with weak analytic singularities, so $E^{as}_{nK}(\alpha) = X$ by convention. If $\alpha$ is big, by \cite[Lemma 3.5]{DH24}, there exists a K\"ahler current $T^*$ with weak analytic singularities such that $E^{as}_{nK}(\alpha) = E_+(T^*)$; in particular, $E^{as}_{nK}(\alpha)$ is a proper closed analytic subset, which proves the first half of the result.

On the other hand, if $\alpha$ is K\"ahler, let $\omega \in \alpha$ be a K\"ahler form, which is in particular a K\"ahler current with weak analytic singularities. Therefore, by definition, we have $E^{as}_{nK}(\alpha) \subset E_+(\omega) = \emptyset$. Conversely, suppose $E^{as}_{nK}(\alpha) = \emptyset \ne X$. Then by what we proved above, $\alpha$ is big. By \cite[Lemma 3.5]{DH24}, there exists a K\"ahler current with weak analytic singularities such that
$$E_+(T^*) = E^{as}_{nK}(\alpha) = \emptyset,$$
and hence $E_{c}(T^*) = \emptyset$ for every $c>0$. Therefore, by \cite[Proposition 3.3]{DP04}, $\alpha = \{T^*\}$ is K\"ahler.
\end{proof}
}


\subsection{Volume of cohomology classes}
Compared with the projective setting, a complex analytic variety may admit only a few divisors but many cohomology classes in $H_{\rm{BC}}^{1,1}(X,\mathbb{R})$. While the notion of a global section does not apply to these classes, we can still define their volume. In this subsection, we briefly recall some basic facts about the volume of a cohomology class. 

\begin{definition}[Volume of a cohomology class]\label{Def:Vol}
For $\alpha \in H^{1,1}_{\rm{BC}}(X, \mathbb{R})$ on a normal compact K\"ahler variety $X$, if $\alpha$ is a pseudo-effective cohomology class, we define the \emph{volume} of $\alpha$ by pulling back via a resolution $\pi: X'\to X$ 
$$
\mathrm{vol}  (\alpha):=\sup _{T\in \pi^*(\alpha)} \int_{X'} T_{\rm{ac}}^n =  \sup _{T\in \pi^*(\alpha) }\int_{X' \setminus \operatorname{Sing}(T)}T^n,
$$
for $T$ ranging over all the closed positive (1,1)-currents in $\pi^*(\alpha)$; otherwise, we set $\mathrm{vol}(\alpha)=0$. Here $T_{\mathrm{ac}}$ denotes the absolutely continuous part of the current $T$, and $\operatorname{Sing}(T)$ denotes the singular locus of $T$.
\end{definition}

\begin{remark}
Since the volume is a birational invariant for smooth varieties (cf. \cite[Example 2.2.49]{Laza04} and \cite{Boucksom02}), the volume of a cohomology class on a normal compact K\"ahler variety is independent of the choice of resolution.
\end{remark}
\begin{remark}
Given a pseudo-effective line bundle $L$ on a compact K\"ahler manifold $X$, \cite[Theorem 1.2]{Boucksom02} showed that the usual definition of the volume of a line bundle,
$$
\vol(L) := \limsup_{k \to +\infty} \frac{n!}{k^n} h^0(X, kL),
$$
coincides with Definition~\ref{Def:Vol}.
\end{remark}
The volume function is continuous on the Bott--Chern cohomology space of a normal compact K\"ahler variety.
\begin{proposition}[{\cite[Corollary 4.11]{Boucksom02}}]\label{ContVol}
Let $X$ be a normal compact K\"ahler variety. Then the volume function $$\vol : H^{1,1}_{\rm{BC}}(X,\mathbb{R})\to \mathbb{R}$$is continuous. 
\end{proposition}

For a smooth family of K\"ahler manifolds, the volume function is upper semi-continuous.
\begin{proposition}[{\cite[Proposition 4.13]{Boucksom02}}]\label{uscvolume}
Let ${X} \rightarrow S$ be a smooth family of K\"ahler manifolds. If $\alpha \in H^{1,1}_{\rm{BC}}(X_0, \mathbb{R})$ is a pseudo-effective class on the central fiber $X_0$, then one has
$$
\mathrm{vol}(\alpha) \ge \limsup _{k \rightarrow+\infty} \mathrm{vol}\left(\alpha_k\right)
$$
for every sequence $\alpha_k \in H^{1,1}_{\rm{BC}}\left(X_{t_k}, \mathbb{R}\right)$ converging to $\alpha$ with $(t_k \rightarrow 0)$.
\end{proposition}
\begin{remark}
Recently, \cite{Jiao} proved the upper semi-continuity of volume of divisors for a projective morphism when fibers are irreducible and reduced, by use of a Fujita approximation technique. 
\end{remark}
The volume function satisfies some special properties for big classes and nef classes.

\begin{lemma}[{\cite[Theorems 4.1, 4.7]{Boucksom02}}]\label{BigVol}
For a normal compact K\"ahler variety $X$, one has: \\
(a) A class $\alpha \in H^{1,1}_{\mathrm{BC}}(X, \mathbb{R})$ is big if and only if $\mathrm{vol}(\alpha)>0$;\\
(b) For a nef class $\beta\in H^{1,1}_{\mathrm{BC}}(X,\mathbb{R})$, $\mathrm{vol}(\beta) =\beta^n.$
\end{lemma}

Given a smooth family, the top intersection number of a cohomology class is constant under deformations.
\begin{lemma}\label{constIntersection}
Let $f: X \to S$ be a smooth family of compact complex manifolds over a smooth connected curve $S$, with $n = \dim_{\mathbb{C}}(X/S)$. Suppose that $\alpha \in H^{1,1}_{\mathrm{BC}}(X, \mathbb{R})$. Then the top intersection number
$$
s \longmapsto \int_{X_s} (\alpha|_{X_s})^n
$$
is constant in the family.
\end{lemma}
\begin{proof}
Since $S$ is connected, given any two points $s_1, s_2 \in S$, there exists a smooth real curve $\gamma : [0,1] \to S$ with $\gamma(0) = s_1$, $\gamma(1) = s_2$. Since $f:X\to S$ is a smooth morphism, the constant rank theorem implies that the inverse image $\overline{M} = f^{-1}(\gamma([0,1]))$ is a smooth manifold with boundary $\partial \overline{M} = X_{s_1} \cup X_{s_2}$. Thus, Stokes' theorem gives rise to  
$$
\int_{X_{s_1}}  (\alpha|_{X_{s_1}})^n - \int_{X_{s_2}} (\alpha|_{X_{s_2}})^n = \int_{\overline{M}} d(\alpha^n) = 0.
$$
That is, $s \longmapsto \int_{X_s} (\alpha|_{X_s})^n$ is constant in the family.
\end{proof}


Given a pseudo-effective class on a compact complex manifold or more generally a normal compact complex analytic variety, we can take its Boucksom--Zariski decomposition. 
\begin{definition}[Minimal multiplicity]
Let $X$ be a compact complex manifold and $\alpha \in H^{1,1}_{\mathrm{BC}}(X, \mathbb{R})$ a pseudo-effective class. The \emph{minimal multiplicity} of $\alpha$ at a point $x \in X$ is defined by
\[
\nu(\alpha, x) := \sup_{\varepsilon > 0} \nu\!\left(T_{\min,\varepsilon}, x\right),
\]
where $T_{\min,\varepsilon}$ denotes a current with minimal singularities in the set
\[
\alpha[-\varepsilon \omega] := \{\text{closed almost positive } (1,1)\text{-currents } T \mid T \text{ represents } \alpha,\ T \ge -\varepsilon \omega \,\},
\]and $\nu(T_{\min, \varepsilon},x)$ is the Lelong number of $T_{\min,\varepsilon}$ at $x$. 
If $D$ is a prime divisor on $X$, the \emph{generic minimal multiplicity} of $\alpha$ along $D$ is defined as
\[
\nu(\alpha, D) := \inf \{\, \nu(\alpha, x) \mid x \in D\}.
\]
\end{definition}

\begin{definition}[{Boucksom--Zariski decomposition, \cite[Definition 3.7]{Boucksom04}}]
Let $X$ be a compact complex manifold and $\alpha\in H^{1,1}_{\rm{BC}}(X,\mathbb{R})$ a pseudo-effective class. Then we define the \emph{negative part} of $\alpha$ as $$N(\alpha) = \sum  \nu (\alpha , D)\cdot D,$$
where $D$ runs through all the divisors on $X$,  
and $Z(\alpha) = \alpha  - [N(\alpha)].$ 
Then we say that the decomposition 
$$\alpha = Z(\alpha)+ [N(\alpha)]$$ 
is the \emph{Boucksom--Zariski decomposition}. 
\end{definition}
\begin{remark}\label{NegZariski}
Given a normal compact complex analytic variety $X$, let $f \colon Y \to X$ be a resolution of singularities of $X$. For a pseudo-effective class $\alpha\in H^{1,1}_{\rm{BC}}(X,\mathbb{R})$, we define the \emph{negative part} of $\alpha$ by
$$
N(\alpha) := f_*\bigl(N(f^*\alpha)\bigr).
$$
By \cite[Lemma A.7]{DHY23}, this definition is independent of the choice of resolution, and hence $N(\alpha)$ on a normal complex analytic variety is well defined.
\end{remark}
We then describe the behavior of the volume of a big class under the Boucksom--Zariski decomposition:
\begin{proposition}[{\cite[Proposition 3.20]{Boucksom04}}]\label{VolZariski}
Let $X$ be a compact K\"ahler manifold. 
If $\alpha \in H^{1,1}_{\mathrm{BC}}(X, \mathbb{R})$ is big and $E \ge 0$ is a divisor on $X$ such that 
$\mathrm{supp}(E) \subset \mathrm{supp}(N(\alpha))$, then
\[
\mathrm{vol}(\alpha) = \mathrm{vol}(Z(\alpha) + [E]) = \mathrm{vol}(Z(\alpha)),
\]
where $N(\alpha)$ is the negative part of the Boucksom--Zariski decomposition of $\alpha$, 
and $Z(\alpha) := \alpha - [N(\alpha)]$ is the positive part.
\end{proposition}
Let us end this subsection with some volume inequalities.
\begin{lemma}\label{vollemma}
Let $X$ be a compact K\"ahler manifold of dimension $n$ and $\omega \in H^{1,1}_{\mathrm{BC}}(X, \mathbb{R})$ a K\"ahler class on $X$. If $\alpha \in H^{1,1}_{\mathrm{BC}}(X, \mathbb{R})$ is a pseudo-effective class, then for any $\varepsilon >0$, the volume satisfies the inequality $$\vol (\alpha + \varepsilon \omega) \ge \varepsilon^n  \int_X\omega^n.$$
\end{lemma}
\begin{proof}
Let $T \in \alpha$ be a closed positive $(1,1)$-current on $X$. 
By Demailly's regularization theorem (\cite[Proposition (18.7)]{Demailly12}), 
there exists a sequence $\{T_k\}_{k \in \mathbb{N}}$ of closed $(1,1)$-currents 
with analytic singularities in the same cohomology class $\alpha$, weakly converging to $T$, such that 
\[
T_k + \varepsilon_k \omega \ge 0, \qquad \varepsilon_k \downarrow 0.
\]

In particular, for large $k$ we have $\varepsilon_k<\varepsilon$, so $T_k + \varepsilon\omega \ge (\varepsilon-\varepsilon_k)\omega$. Set $\delta_k=\varepsilon-\varepsilon_k>0$ and then $\delta_k \to \varepsilon$. By the definition of volume, we have $$\vol(\alpha + \varepsilon \omega) = \sup_{T_{\varepsilon} \in \alpha + \varepsilon \omega}\int _X (T_{\varepsilon})_{ac}^n \ge \limsup_{k\to \infty} \int_X (T_k + \varepsilon \omega)^n_{ac} \ge  (\limsup_{k\to \infty} \delta_k^n) \int_X\omega^n = \varepsilon^n \int_X\omega^n.$$
\end{proof}

The volume function satisfies certain log concavity properties.
\begin{proposition}\label{r-logc}
Let $X$ be a compact K\"ahler manifold of complex dimension $n$. If $\alpha, \alpha ' \in H^{1,1}_{\rm{BC}}(X, \mathbb{R})$ are big classes, then the volume satisfies the log concavity property $$\vol(\alpha + \alpha ')^{1/n} \ge \vol (\alpha)^{1/n} + \vol (\alpha ')^{1/n}.$$
\end{proposition}
\begin{proof}
Note that the log concavity of the volume is well known for line bundles and $\R$-divisors (cf. \cite[Remark 2.2.50]{Laza04} and also \cite[Corollary 4.12]{LM09}) and also for nef classes by the Khovanskii--Teissier inequalities proved in \cite{Dem93}. 
The general case then follows easily from the analytic version of Fujita's approximation theorem  \cite[Theorem 1.4]{Boucksom02}. Given big classes $\alpha,\alpha' \in H^{1,1}_{\rm{BC}}(X, \mathbb{R})$ and $\varepsilon >0$, there exists a resolution $\nu : X_\varepsilon \to X$ and K\"ahler classes $\omega ,\omega ' \in H^{1,1}_{\rm{BC}}(X_{\varepsilon}, \mathbb{R})$ on $X_\varepsilon$ such that
$${\rm vol}(\omega )\leq {\rm vol}(\alpha)\leq {\rm vol}(\omega )+\varepsilon ,$$ $${\rm vol}(\omega') \leq {\rm vol}(\alpha ')\leq {\rm vol}(\omega ') +\varepsilon,$$ $${\rm vol}(\omega +\omega')\leq{\rm vol}(\alpha+\alpha ').$$
But then
$$({\rm vol}(\alpha )-\varepsilon)^{1/n}+({\rm vol}(\alpha')-\varepsilon)^{1/n}\le {\rm vol}(\omega )^{1/n} +{\rm vol}(\omega')^{1/n}\le {\rm vol}(\omega +\omega')^{1/n}\le {\rm vol}(\alpha+\alpha')^{1/n},$$
and taking the limit as $\varepsilon \to 0$ yields ${\rm vol}(\alpha)^{1/n}+{\rm vol}(\alpha')^{1/n} \le {\rm vol}(\alpha+\alpha')^{1/n}$.
\end{proof}

\subsection{Bigness and pseudo-effectivity in families} 
We start this subsection with: 
\begin{definition}[Relatively big and pseudo-effective classes]\label{pseffclass}
Let $f:X\to S$ be a proper surjective morphism from a normal complex analytic variety $X$ onto a relatively compact base $S$. Then we say that:\\
\emph{(1)} A cohomology class $\alpha\in H^{1,1}_{\rm{BC}}(X,\mathbb{R})$ is \emph{$f$-pseudo-effective over $S$} or \emph{relatively pseudo-effective over $S$} if there is a complement $V$ of a countable union of Zariski closed proper subsets of $S$ such that $\alpha_s:=\alpha|_{X_s}$ is pseudo-effective for any $s\in V$.\\
\emph{(2)} A cohomology class $\alpha\in H^{1,1}_{\rm{BC}}(X,\mathbb{R})$ is \emph{$f$-big over $S$} or \emph{relatively big over $S$} if there is a non-empty Zariski open subset $U$ of $S$ such that $\alpha_s$ is big for any $s\in U$.  
\end{definition}

The following proposition shows that a relatively pseudo-effective class is pseudo-effective on general fibers.
\begin{proposition}\label{genericpseff}
Let $f:X\to S$ be a family from a normal K\"ahler variety $X$ with a K\"ahler class $\omega$ onto a relatively compact base $S$. A class $\alpha \in H^{1,1}_{\rm{BC}}(X, \mathbb{R})$ is $f$-pseudo-effective over $S$ if and only if there exists a non-empty Zariski open subset $U \subset S$ such that $\alpha_s$ is pseudo-effective for any $s\in U$. 
\end{proposition}
\begin{proof}
After shrinking $S$ and replacing $X$ by a resolution,  we may assume that $X\to S$ is smooth. By assumption $\alpha _s$ is pseudo-effective for any very general $s\in S$.
By Lemma \ref{vollemma}, $$\vol(\alpha _s+\varepsilon \omega _s)  \ge \varepsilon ^d \cdot v \text{ for very general } s\in S,$$ where $d=\dim (X/S)$ and $v=\int _{X_s}\omega _s^d$. Note that $v= \int_{X_s} \omega_s^d>0$ is a constant independent of $s\in S$ by Lemma \ref{constIntersection}. 

Let $X_{s_0}$ be a special fiber, and define the pseudo-effective threshold
\[
\tau_{s_0} = \inf \{ t \ge 0 \mid \alpha_{s_0} + t \omega_{s_0} \text{ is pseudo-effective} \}.
\]
Then $\alpha_{s_0} + \tau_{s_0} \omega_{s_0}$ is pseudo-effective but not big; hence $
\vol(\alpha_{s_0} + \tau_{s_0} \omega_{s_0}) = 0$. 
By Boucksom's upper semicontinuity Proposition~\ref{uscvolume},
\[
0 = \vol(\alpha_{s_0} + \tau_{s_0} \omega_{s_0})
\ge \limsup_{k \to \infty} \vol(\alpha_{s_k} + \tau_{s_0} \omega_{s_k})
\ge \tau_{s_0}^d \, v.
\]
Therefore $\tau_{s_0} = 0$, and in particular $\alpha_{s_0}$ is pseudo-effective.
\end{proof}

Before proceeding, let us first recall the definition of the pullback of a closed positive $(1,1)$-current with local potentials. 

\begin{definition}[{Pull back of a closed positive $(1,1)$-current with local potential, \cite[Section 3.1]{Meo96}}]\label{Def: pullcurrent}
Let $f: X \rightarrow Y$ be a  surjective morphism between complex analytic varieties. For a closed positive $(1,1)$-current $T$ with local potentials on $Y$, we define the pull back $f^* T$ as follows:

Let $\Omega$ be an open set in $Y$ on which $T_{\Omega}=i \partial \bar{\partial} u$ with $u$ a plurisubharmonic potential function in $\Omega$. If the plurisubharmonic function $u \circ f$ is not identically $-\infty$ on each connected component of $f^{-1}(\Omega)$, we then define $$\left.\left(f^* T\right)\right|_{f^{-1}(\Omega)}=i \partial \bar{\partial}(u \circ f).$$If $\Omega^{\prime}$ is another open set in $Y$ with a similar representation $T_{\Omega^{\prime}}=i \partial \bar{\partial} u^{\prime}$, the function $u-u^{\prime}$ is pluriharmonic on $\Omega \cap \Omega^{\prime}$, and the same holds for $u \circ f-u^{\prime} \circ f$ on $f^{-1}(\Omega) \cap f^{-1}\left(\Omega^{\prime}\right)$, so the above definition is independent of the choice of $\Omega$ and the potential $u$.
\end{definition}
\begin{remark}\label{Rmk: Globalpull}
Note that the closed positive current $T$ admits a global representative 
$T = \theta + i \partial \bar{\partial} \varphi$, where $\theta$ is a smooth $(1,1)$-form and $\varphi$ is an almost plurisubharmonic function on $Y$, obtained via a partition of unity argument (cf.~\cite[Proposition 3.1.1]{Meo96}). Hence the pullback $f^*T$ can be written as 
$f^*T = f^*\theta + i \partial \bar{\partial} f^*\varphi$, and therefore $[f^*T] = f^*[T]$.
\end{remark}

Given a proper bimeromorphic morphism $f\colon Y \to X$ of normal complex spaces and a class
$\beta \in H^{1,1}_{\rm BC}(Y,\mathbb{R})$, we write $f_{*}\beta \in H^{1,1}_{\rm BC}(X,\mathbb{R})$
to mean that for some (equivalently for every) closed $(1,1)$-current $T$ with local potentials
representing $\beta$, the current $f_{*}T$ admits local potentials on $X$; in this case we set
$f_{*}\beta := [f_{*}T]$, which is independent of the choice of $T$.

We can now describe the behavior of pseudo-effective and big classes under bimeromorphic change.

\begin{proposition}\label{Prop: PullPSEf}
Let $f: Y \rightarrow X/S$ be a proper bimeromorphic morphism of normal complex analytic
varieties over a relatively compact base $S$. Let $\alpha \in H^{1,1}_{\rm BC}(X,\mathbb{R})$
and $\beta \in H^{1,1}_{\rm BC}(Y,\mathbb{R})$ be relatively pseudo-effective over $S$, and
assume that $\beta' := f_{*}\beta \in H^{1,1}_{\rm BC}(X,\mathbb{R})$. Then $\beta'$ and
$\alpha' := f^{*}\alpha$ are relatively pseudo-effective over $S$.
\end{proposition}
\begin{proof}
First, by Definition \ref{pseffclass}, to check relative pseudo-effectivity, it is sufficient to check it on the general fiber. By \cite[Lemma 2.22]{CRC25}, $f_s : Y_s \to X_s$ is bimeromorphic on the general fiber. Hence, in what follows, without loss of generality, we can reduce the problem to the absolute case and assume that $S$ is a point. 

We then show that if $\alpha \in H^{1,1}_{\rm BC}(X, \mathbb{R})$ is pseudo-effective, then so is $f^*\alpha$. Let $T_\alpha$ be a closed positive $(1,1)$-current such that $[T_\alpha] = \alpha$. By Definition \ref{Def: pullcurrent}, we can find an open cover $\{\Omega_i\}$ of $X$ such that
\[
T_\alpha|_{\Omega_i} = i \partial \bar{\partial} u_i,
\]
for some plurisubharmonic function $u_i$, and
\[
f^*T_\alpha \big|_{f^{-1}(\Omega_i)} = i \partial \bar{\partial} (u_i \circ f).
\]
By \cite[Proposition~1.1.8]{Demailly12}, the pullback of a plurisubharmonic function is again plurisubharmonic, so $f^*T_\alpha$ is a positive current. By Remark \ref{Rmk: Globalpull}, $[f^* T_{\alpha}] = f^* \alpha$. Hence $f^* \alpha$ is a pseudo-effective class.


We end our proof by showing that if $\beta \in H^{1,1}_{\rm BC}(Y,\mathbb{R})$ is pseudo-effective, then so is $f_* \beta$. First, there exists a closed positive $(1,1)$-current $T_{\beta}$ such that $[T_{\beta}] = \beta$. By definition, a closed $(1,1)$-current is positive if and only if it is positive on every smooth positive $(n-1,n-1)$-form $\varphi$. Moreover, since $\varphi$ is positive, the pullback $f^*\varphi$ is semi-positive. Hence
\[
\langle f_* T_{\beta}, \varphi \rangle = \langle T_{\beta}, f^*\varphi \rangle \ge 0,
\]
for every smooth positive $(n-1,n-1)$-form $\varphi$. By assumption $f_{*}\beta \in H^{1,1}_{\rm BC}(X,\mathbb{R})$, so $f_{*}T_{\beta}$ admits
local potentials; combined with the above, $f_{*}T_{\beta}$ is a closed positive
$(1,1)$-current representing $f_{*}\beta$. Hence $f_{*}\beta$ is pseudo-effective.
\end{proof}
\begin{proposition}\label{Prop: PullBig}
Let $f: Y \rightarrow X/S$ be a proper bimeromorphic morphism of normal K\"ahler varieties
over a relatively compact base $S$. Let $\alpha \in H^{1,1}_{\rm BC}(X,\mathbb{R})$ and
$\beta \in H^{1,1}_{\rm BC}(Y,\mathbb{R})$ be relatively big over $S$, and assume that
$\beta' := f_{*}\beta \in H^{1,1}_{\rm BC}(X,\mathbb{R})$. Then $\alpha' := f^{*}\alpha$ and
$\beta'$ are relatively big over $S$.
\end{proposition}
\begin{proof}
Similar to Proposition~\ref{Prop: PullPSEf}, we may reduce to the absolute setting and assume that $S$ is a point. In particular, both $X$ and $Y$ are compact.

We first show that if $\alpha \in H^{1,1}_{\mathrm{BC}}(X, \mathbb{R})$ is big, then $f^{*}\alpha$ is big. Since
$$
\mathrm{vol}(f^{*}\alpha) = \mathrm{vol}(\alpha) > 0,
$$
Lemma~\ref{BigVol} implies that $f^{*}\alpha$ is big.

We next prove that if $\beta \in H^{1,1}_{\mathrm{BC}}(Y, \mathbb{R})$ is big, then $\beta' = f_{*}\beta$ is big. Let $T_{\beta}$ be a K\"ahler current representing $\beta$. By definition, there exists a smooth Hermitian form $\omega_{Y}$ with
$$
T_{\beta} \ge \omega_{Y}.
$$
Let $\omega_{X}$ be a smooth Hermitian form on $X$. Since $Y$ is compact, there exists a constant $C > 0$ such that
$$
f^{*}\omega_{X} \leq C\, \omega_{Y}.
$$
It follows that
$$
f_{*}T_{\beta} \ge f_{*}\omega_{Y} \ge C^{-1} f_{*}f^{*}\omega_{X}
  = C^{-1}\, \omega_{X}.
$$
Thus $f_{*}T_{\beta}$ is again a K\"ahler current. Since by assumption $f_*\beta\in H^{1,1}_{\rm{BC}}(X,\mathbb{R})$, $f_*T_{\beta}$ has local potential. Because $[f_{*}T_{\beta}] = f_{*}\beta$, the class $f_{*}\beta$ is big.
\end{proof}

We say that a bimeromorphic contraction $ \phi : X \dashrightarrow Y $ is a \emph{$(K_X + B +  { {\bbeta}})$-non-positive contraction} if under a common resolution $ p : W \to X $ and $ q : W \to Y $, the divisor $ E $ in
\[
p^*(K_X + B +  { {\bbeta}_X}) = q^*(K_Y + B_Y +  { {\bbeta}}_Y) + E
\]
is an effective $ q $-exceptional divisor (that is, $ q_* E = 0 $).

A $(K_X+B+ \bbeta_X)$-non-positive contraction preserves pseudo-effectivity (resp. bigness) of adjoint classes. 
\begin{proposition}\label{PushPSEF}
Let $(X,B + \bbeta_X)$ be a generalized K\"ahler pair. Given a $(K_X+B+ \bbeta_X)$-non-positive contraction (e.g. a sequence of steps of the $(K_X+B + \bbeta_X)$-MMP) over the relatively compact $S$
\[
\phi : (X, B +  { \boldsymbol{\beta}_X}) 
\dashrightarrow 
(X', B' +  { \boldsymbol{\beta}}_{X'})/S,
\]
the adjoint class $ K_{X} + B +  { {\bbeta}_X} $ is relatively pseudo-effective (resp.\ big) over $ S $ if and only if $ K_{X'} + B' +  { {\bbeta}}_{X'}$ is relatively pseudo-effective (resp.\ big) over $ S $.
\end{proposition}

\begin{proof}
Assume that $K_{X}+ B+ {\bbeta}_X$ is relatively pseudo-effective (resp.\ big) over $S$. 
Let $p:Y \to X$ and $q:Y \to X'$ be a common resolution of $\phi$. 
Since $\phi$ is $(K_{X}+B+{\bbeta}_X)$-non-positive over $S$,  $$p^*(K_{X}+ B+{\bbeta}_X) \equiv_S 
q^*(K_{X'}+B'+ {\bbeta}_{X'}) + E,$$
with $E \ge 0$ effective and $q$-exceptional. 
Since we already showed in Propositions \ref{Prop: PullPSEf} and \ref{Prop: PullBig} that pullback and pushforward under bimeromorphic morphisms preserve pseudo-effectivity (resp.\ bigness), it follows that
\[
q_*\bigl(p^*(K_{X}+B + {\bbeta}_X)\bigr) 
\equiv_S 
q_*\bigl(q^*(K_{X'}+B'+ {\bbeta}_{X'}) + E\bigr) 
= K_{X'}+ B'+ {\bbeta}_{X'}
\]
is relatively pseudo-effective (resp.\ big) over $S$. 

Conversely, if $K_{X'}+B'+ {\bbeta}_{X'}$ is relatively pseudo-effective (resp.\ big) over $S$, then by the effectiveness of $E$ and $p^*(K_{X}+ B+{\bbeta}_X) \equiv_S 
q^*(K_{X'}+B'+ {\bbeta}_{X'}) + E$, it is easy to see that 
$p^*(K_{X}+ B+ {\bbeta}_X)$ is relatively pseudo-effective (resp.\ big) over $S$, and therefore so is 
$K_{X}+ B+ {\bbeta}_X$.
\end{proof}

\begin{remark}
Note that for a bimeromorphic contraction morphism, the pseudo-effectivity of $ f_*  \alpha $  does not necessarily imply that of $  { {\alpha}} $. Consider the blow-up $ f : X \to Y $ of a smooth projective variety $ Y $. Let $ D \ge 0 $ be an effective divisor on $ X $ and $ E \ge 0 $ an exceptional divisor. For $ m \gg 1 $, the divisor $ D - mE $ is not pseudo-effective, but
\[
f_*(D - mE) = f_* D \ge 0
\]
is pseudo-effective.
\end{remark}

\begin{proposition}\label{AbsRelativepseff}
Let $f:X\to S$ be a family from a normal K\"ahler variety $X$ to a complex analytic base $S$, and $\alpha \in H^{1,1}_{\rm{BC}}(X,\mathbb{R})$. Let $W$ be a compact subset of $S$ and $X_W = f^{-1}(W)$. If $\alpha|_{X_W} \in H^{1,1}_{\rm{BC}}(X_W,\mathbb{R})$ is big (resp. pseudo-effective), then it is relatively big (resp. pseudo-effective) over $W$. 
\end{proposition}

\begin{proof}
We first prove that if $\alpha|_{X_W}$ is absolute big then it is also relatively big over $W$. By Lemma \ref{nonKahlerlocus}, the restricted non-K\"ahler locus of $\alpha$ is a proper subvariety of $X_W$. 
In particular, there exists a non-empty Zariski open subset $U\subset W$ such that the fibers $X_s$ for $s\in U$ are not contained in $E_{nK}^{as}(\alpha|_{X_W})$. Since $$E_{nK}^{as}(\alpha|_{X_s}) \subset E_{nK}^{as}(\alpha|_{X_W})\cap X_s,$$ $E_{nK}^{as}(\alpha|_{X_s})$ does not contain the fiber $X_s$ for $s\in U$. Hence, by Lemma \ref{nonKahlerlocus}, the class $\alpha|_{X_s}$ is big for $s\in U$, i.e., $\alpha|_{X_W}$ is $f$-big over $W$. 

Given a pseudo-effective class $\alpha|_{X_W}$ on $X_W$, Proposition \ref{bigandpseff} shows that there exists a K\"ahler class $\omega$ on $X_W$ such that $\alpha|_{X_W}  + \varepsilon \omega$ is (absolute) big for any $\varepsilon >0$. Hence, by what we just proved above, $\alpha|_{X_W} + \varepsilon \omega$ is also $f$-big over $W$ for any $\varepsilon >0$. Thus, $\alpha$ is $f$-pseudo-effective over $W$. 
\end{proof}

\begin{proposition}[{\cite[Lemma 5.1]{DH24}}]\label{KVVanishing}
Let $g: X \rightarrow S$ be a proper surjective morphism from a generalized K\"ahler pair $(X, B+ {\boldsymbol{\beta}})$ with generalized klt singularities onto a relatively compact complex analytic variety $S$. If $D$ is a $\mathbb{Q}$-Cartier $\mathbb{Z}$-divisor on $X$ such that $D-\left(K_X+B+ { {\bbeta}}_X\right)$ is relatively nef and big over $S$, then $R^i g_* \mathcal{O}_X(D)=0$ for all $i>0$.
\end{proposition}
\begin{proof}
The proof is based on \cite[Lemma 5.1]{DH24} with a minor change. Let $\nu:(X',B'+{\boldsymbol{\beta}}_{X'})\to (X,B+{\boldsymbol{\beta}}_X)$ be a log resolution such that ${\boldsymbol{\beta}}_{X'}$ is nef over $S$, with 
$$
K_{X'}+B'+\boldsymbol{\beta}_{X'}=\nu^*(K_X+B+\boldsymbol{\beta}_X).
$$
By Propositions \ref{Prop: PullBig} and \ref{nef}, the $\mathbb{R}$-divisor
$$
 \nu^*(D) - (K_{X'}+B') \equiv_S \boldsymbol{\beta}_{X'} + \nu^*\bigl(D - (K_X+B+\boldsymbol{\beta}_X)\bigr)
$$
is big and nef over $S$. Since $B'$ may have irrational coefficients, possibly no multiple of this divisor is Cartier, so we first perturb the boundary. Write $B'=\sum_i b_i B'_i$ with $B'_i$ the prime components. 
If we choose rational numbers $b_i'$ with $|b_i-b_i'|$ sufficiently small and set $B'_{\mathbb{Q}}:=\sum_i b_i'B'_i$,  then the $\mathbb{Q}$-divisor
$$
\nu^*(D)-(K_{X'}+B'_{\mathbb{Q}}) = \bigl(\nu^*(D)-(K_{X'}+B')\bigr)+\sum_i(b_i-b_i')B'_i
$$
is big over $S$. Since $X'$ is smooth and $D$ is $\mathbb{Q}$-Cartier, there exists $m\in\mathbb{N}_{+}$ such that $m\bigl(\nu^*(D)-(K_{X'}+B'_{\mathbb{Q}})\bigr)$ is Cartier, so $X'$ carries a $g'$-big line bundle. By \cite[Lemma 2.18]{DH24}, there exists a projective modification $\mu:X''\to X'$ such that $g'':X''\to S$ is projective; replacing $X''$ by a further blow-up, we may assume $\psi = \mu \circ \nu$ is a log resolution of $(X,B)$.

Then the proof follows along the same lines as \cite[Lemma~5.1]{DH24}. 
We have the following commutative diagram:
\begin{center}
\begin{tikzcd}[ampersand replacement=\&]
	{X''} \& {X'} \& X \\
	\&\& S
	\arrow["\mu", from=1-1, to=1-2]
	\arrow["{g''}"', from=1-1, to=2-3]
	\arrow["\nu", from=1-2, to=1-3]
	\arrow["{g'}"{description}, from=1-2, to=2-3]
	\arrow["g", from=1-3, to=2-3]
\end{tikzcd}
\end{center}
such that $\boldsymbol{\beta}_{X''} = \mu^*\boldsymbol{\beta}_{X'}$ is nef over $S$, and
\[
K_{X''} + B'' + \boldsymbol{\beta}_{X''}
= \psi^*\bigl(K_X + B + \boldsymbol{\beta}_X\bigr).
\]
We then set 
$$
E := \left\lfloor B'' + \{-\psi^*D\} \right\rfloor, \quad B^* := B'' + \{-\psi^*D\} - E = \{B'' + \{-\psi^*D\}\},
$$ 
where $\{\cdot\},\left\lfloor\cdot\right\rfloor,\left\lceil\cdot\right\rceil $ denote the fractional, round down, and round up part of a divisor, respectively. Hence, 
$$
\left\lceil\psi^*D\right\rceil - E - (K_{X''}+B^*) \equiv \boldsymbol{\beta}_{X''} + \psi^*\left(D - (K_X+B+\boldsymbol{\beta}_X)\right).
$$
By assumption, $D-(K_X +B + \boldsymbol{\beta}_X)$ is nef and big over $S$, hence $\left\lceil\psi^*D\right\rceil - E - (K_{X''}+B^*)$ is also nef and big over $S$. 

Since $g'':X'' \to S$ is projective, we may apply the standard version of the Kawamata--Viehweg vanishing theorem, obtaining
$$
R^p(g'')_* \mathcal{O}_{X''}(\left\lceil\psi^*D\right\rceil - E) = 0, \quad 
R^p \psi_* \mathcal{O}_{X''}(\left\lceil\psi^*D\right\rceil - E) = 0 \quad \text{for any } p > 0.
$$
By the same argument as in \cite[Lemma 5.1]{DH24}, we can prove that $$
\psi_* \mathcal{O}_{X''}(\left\lceil\psi^*D\right\rceil - E) = \mathcal{O}_X(D).
$$ Therefore, a standard Leray spectral sequence argument yields
$$
R^p g_* \mathcal{O}_X(D) 
= R^p g_*(\psi_* \mathcal{O}_{X''}(\left\lceil\psi^*D\right\rceil - E)) 
= R^p g''_* \mathcal{O}_{X''}(\left\lceil\psi^*D\right\rceil - E) 
= 0
$$
for all $p > 0$.
\end{proof}

\subsection{Uniruled varieties}We next introduce some basic results about uniruled varieties.
\begin{definition}[Uniruled varieties]\label{Def: Uniruled}
A complex analytic variety $X$ is called \emph{uniruled} if for every general point $x\in X$ there is a non-constant morphism $f:  \mathbb{P}^1\to X$ such that $f(0) = x$.
\end{definition}

\begin{remark}\label{charaterization-uniruled}
With a slight modification of the proof of \cite[Proposition II.5.4]{MP97}, if $X$ is a compact complex manifold, then the following are equivalent:\\
(1) $X$ is uniruled (in the sense of Definition \ref{Def: Uniruled});\\
(2) There exists a complex analytic variety $Y$ and a dominant meromorphic map $$\Phi: \mathbb{P}^1 \times Y  \dashrightarrow X,$$that is not constant on $\mathbb{P}^1$;\\
(3) There exists a free rational curve $f: \mathbb{P}^1 \to X$;\\
(4) For a general point $x\in X$, there exists a free rational curve $f: \mathbb{P}^1 \to X$ with $f(\infty) = x$. 

Obviously, uniruledness is a bimeromorphically invariant property. If, moreover, $X$ is K\"ahler, then uniruledness is also equivalent to $K_X$ being non pseudo-effective by Theorem~\ref{Theorem Ou25}.
\end{remark}

Any uniruled variety with canonical singularities is covered by a family of $K_X$-negative curves.
Recall that a normal complex variety $M$ is said to have
\emph{(only) canonical singularities} if the canonical divisor $K_M$
is a {$\mathbb{Q}$-Cartier divisor}, i.e., some integer multiple of
$K_M$ is a Cartier divisor, and if the discrepancy divisor
$K_{\tilde{M}}-\mu^*K_M$ is effective for a (or every) resolution of singularities
$\mu:\tilde{M}\rightarrow M$. 
\begin{lemma}\label{Lemma: NegCurve}
Let $X$ be a uniruled variety with canonical singularities. Then $X$ is covered by a family of rational curves $\{{C}_\lambda\}_{\lambda \in \Lambda}$ such that $K_X\cdot {C}_\lambda <0$. 
\end{lemma}
\begin{proof}
Since $X$ is uniruled, by Definition \ref{Def: Uniruled}, $X$ is covered by rational curves $\{{C}_\lambda\}_{\lambda \in \Lambda}$. Since $X$ admits canonical singularities, let $f:X'\to X$ be a resolution then $$ K_{X'}=f^*K_X+E,$$for some effective, $f$-exceptional divisor $E\ge 0$. Let ${C}_\lambda' = f^{-1}_* ({C}_\lambda)$ be the strict transform of a general member in the family of curves $\{{C}_\lambda\}_{\lambda \in \Lambda}$, then $$K_X\cdot {C}_\lambda=(K_{X'}-E)\cdot {C}'_\lambda\leq K_{X'}\cdot {C}'_\lambda$$ and so it suffices to show that $K_{X'}\cdot {C}'_\lambda<0$; in particular, we may assume that $X$ is smooth.

Then the rest of the proof follows from \cite{Kollar96}. Since $X$ is a smooth uniruled variety, there is a family of rational curves $\mathcal C\to T$ and a dominant holomorphic map $g:\mathcal C \to X$. Replacing $T$ by an open subset, we may assume that $T$ is isomorphic to a disc. Cutting by hyperplanes, we may assume that $\dim \mathcal C=\dim X$ and hence $g$ is a generically finite holomorphic map of complex manifolds. We may then write $K_{\mathcal C}=g^*K_X+R$ where $R$ is an effective divisor. We then have \[-2={\rm deg}(K_{\mathbb P ^1})={\rm deg}(K_{\mathcal C}|_{ C _\lambda})=K_X\cdot  C _\lambda+R\cdot  C _\lambda\ge K_X\cdot C _\lambda\] and hence $K_X\cdot  C _\lambda<0$.
\end{proof}



We can easily generalize the remarkable Theorem \ref{Theorem Ou25} of Ou to the singular setting. 
\begin{lemma}\label{BDPPCanonical}
For a normal compact $\mathbb{Q}$-Gorenstein K\"ahler variety $X$, we have:\\
\emph{(1)} If $X$ is not uniruled, then $K_X$ is pseudo-effective;\\
\emph{(2)} If $X$ admits canonical singularities and $K_X$ is pseudo-effective, then $X$ is not uniruled. 
\end{lemma}
\begin{proof}
(1) Let $\pi: Y \to X$ be a resolution. If $X$ is not uniruled, then neither is $Y$. Since $Y$ is smooth and K\"ahler, Theorem \ref{Theorem Ou25} implies that $K_Y$ is pseudo-effective. By Proposition \ref{Prop: PullPSEf}, $K_X$ is pseudo-effective. 

(2) If $X$ has canonical singularities and $\pi: Y \to X$ is a resolution, then $K_Y=\pi^*K_X  + E$ with the exceptional divisor $E \ge 0$. Thus, if $K_X$ is pseudo-effective, then $K_Y$ is also pseudo-effective. By \cite[Theorem 1.1]{Ou25}, $Y$ is not uniruled and thus, $X$ is not uniruled by bimeromorphic invariance of uniruledness, Remark \ref{charaterization-uniruled}.
\end{proof}

\begin{remark}
If $X$ is log terminal, the pseudo-effectivity of $K_X$ does not necessarily imply that $X$ is not uniruled as shown in the following two counterexamples. As the first one, in \cite[Section 5]{k-bmy}, Koll\'ar constructed a hypersurface $$H\left(a_1, \ldots, a_n\right):=\left(x_1^{a_1} x_2+x_2^{a_2} x_3+\cdots+x_{n-1}^{a_{n-1}} x_n+x_n^{a_n} x_1=0\right) \subset \mathbb{P}\left(w_1, \ldots, w_n\right)$$ with quotient singularities (hence log terminal), which is rational and $K_{H(a_1,\ldots ,a_n)}$ is ample for a certain choice of $(a_1, \ldots , a_n)$. The second one was given by Totaro in \cite[Example on p. 245]{Totaro10}, and is a rational Calabi--Yau quotient surface as $Y  = E \times E / (\mathbb{Z}/3)$ where $\mathbb{Z}/3$ acts by $(\zeta ,\zeta )$ on $E\times E $ with $\zeta$ a third root of the unity. Note that $Y$ is rational, since it can be obtained by blowing up 12 points on $\mathbb{P}^2$ in the dual Hesse configuration and then blowing down 9 lines.
\end{remark}

\subsection{Pseudo-effective threshold}
We end this section by introducing the pseudo-effective threshold and several basic facts about it, which will play an important role in the proof of our main Theorem \ref{Deformation invariance of pseudo-effectivity}.

\begin{definition}[Pseudo-effective threshold]
Let $X$ be a normal compact K\"ahler variety, and $\omega \in H^{1,1}_{\rm{BC}}(X, \mathbb{R})$ a K\"ahler class on $X$. We define the \emph{pseudo-effective threshold of} $\alpha \in H^{1,1}_{\rm{BC}}(X,\mathbb{R})$ to be $$\tau(X, \alpha; \omega):=\inf \left\{t \in \mathbb{R}_{\ge 0} \mid \alpha+t \omega \text { is pseudo-effective}\right\}.$$
\end{definition}
Similarly, we can define the relative pseudo-effective threshold. 
\begin{definition}[Relative pseudo-effective threshold]
Let $f:X\to S$ be a proper surjective morphism with connected fibers from a normal K\"ahler variety onto a relatively compact base $S$,  
and $\omega \in H^{1,1}_{\rm{BC}}(X, \mathbb{R})$ a K\"ahler class on $X$. 
We define the \emph{relative pseudo-effective threshold} of $\alpha\in H^{1,1}_{\rm{BC}}(X,\mathbb{R})$ over $S$ to be $$\tau(X/S,\alpha ;\omega) := \inf\{t\in \mathbb{R}_{\ge 0}\mid \alpha + t \omega \text{ is }f\text{-pseudo-effective over }S\}.$$
For a line bundle $L$ on $X$, set $\tau(X/S, L ; \omega) :=  \tau(X/S, c_1(L) ; \omega)$. 

Moreover, we define
\[
\tau'(X/S, \alpha; \omega) := \inf \{ t \in \mathbb{R}_{\ge 0} \mid \alpha + t \omega \text{ is } f\text{-big over } S \}.
\] 
\end{definition}

\begin{lemma}
Let $f:X\to S$ be a proper surjective morphism with connected fibers from a normal K\"ahler variety onto a relatively compact base $S$, and $\omega\in H^{1,1}_{\rm{BC}}(X, \R)$ a K\"ahler class on $X$. Then  $$\tau(X/S, \alpha; \omega)=\tau'(X/S, \alpha; \omega).$$
\end{lemma}
\begin{proof} Set $\tau=\tau(X/S, \alpha; \omega)$ and $\tau'= \tau'(X/S, \alpha; \omega)$. We only need to prove $\tau' \le \tau$. By definition of $\tau$, there exists a non increasing sequence $t_i$ with $\lim_{i \to \infty} t_i =  \tau$ in $\mathbb{R}_{\ge 0}$ and a countable union $Z_i$ of proper Zariski closed subsets of $S$ such that $\alpha_s + t_i \omega_s$ is pseudo-effective for every $s\in V_i : = S \setminus Z_i$. Therefore $\alpha_s + \tau \omega_s$ is pseudo-effective for every $s \in V_\infty = \bigcap_{i =1}^\infty V_i$ (note that $V_\infty$ is also a complement of countable union of proper Zariski closed subsets of $S$). By Proposition \ref{genericpseff}, there exists a dense Zariski open subset $V$ such that $\alpha_s + \tau \omega_s$ is pseudo-effective for $s\in V$. 

Passing to a resolution and further shrinking $V$ we may assume that $f$ is smooth over $V$. Thus, Lemma \ref{vollemma} implies that for any $s\in V$ and any $t> \tau$, $$\vol(\alpha_s + t\omega_s) \ge (t-\tau)^n v>0,$$ with $n = \dim_{\mathbb{C}} (X/S)$ and $v =\int _{X_s} \omega_s^n$. Therefore, by Lemma \ref{BigVol}, for any $t> \tau$, $\alpha + t\omega$ is $f$-big, and thus $\tau' \le  \tau$.
\end{proof}

\begin{lemma}\label{nonpseffcriterion}
Let $f:X\to S$ be a proper surjective morphism with connected fibers from a normal K\"ahler variety onto a relatively compact base $S$, and $\omega \in H^{1,1}_{\rm{BC}}(X,\mathbb{R})$ a K\"ahler class on $X$. Then $\alpha$ is not $f$-pseudo-effective over $S$ if and only if the pseudo-effective threshold $\tau(X/S,\alpha ; \omega)> 0$.
\end{lemma}
\begin{proof}
Assume that $\alpha$ is $f$-pseudo-effective over $S$. Then we have $\tau(X/S, \alpha; \omega) = 0$ by definition. 
Conversely, if $\tau(X/S, \alpha; \omega) = 0$, then for any $\varepsilon > 0$, the class $\alpha + \varepsilon \omega$ is $f$-big. By Proposition \ref{bigandpseff}, $\alpha$ is $f$-pseudo-effective over $S$. 
\end{proof}

\begin{remark}\label{r-pseft} If $S$ is relatively compact, then the relative pseudo-effective threshold
$$\tau(X,\alpha,S;\omega) = \inf\{ t\in \mathbb{R}_{\ge 0}\mid  \alpha + t \omega \text{ is }f\text{-pseudo-effective over some neighborhood } U\subset S \}$$ 
is well defined, i.e.,\ $\tau (X, \alpha ,S ;\omega)<+\infty$. This follows since for $t\gg 0$, 
$\alpha+t\omega$ is K\"ahler (and hence pseudo-effective) over $S$. Thus, $\tau (X,\alpha ,S;\omega)<+\infty$.
\end{remark}

\begin{lemma}\label{Lemma: relativebig}
Let $f:X\to S$ be a proper surjective morphism with positive dimensional fiber from a normal K\"ahler variety onto a relatively compact base $S$. If $\alpha \in H^{1,1}_{\rm{BC}}(X, \R)$ is relatively big over $S$, then $-\alpha$ is not relatively big over $S$. 
\end{lemma}
\begin{proof}
To see this, assume by contradiction that  $-\alpha$ is big over a general fiber $X_s$. 
Then there exist K\"ahler currents
\[
T_+ \ge \varepsilon_+ \omega',\quad T_- \ge \varepsilon_- \omega'
\]
for some K\"ahler metric $\omega'$ on $X_s$ and $\varepsilon_+, \varepsilon_->0$, such that $[T_+]=\alpha|_{X_s},\ [T_-]=-\alpha|_{X_s}$.

Since $[T_+ + T_-] = 0$ and $T_+  +T_- \ge 0$, we have $T_+ + T_-= 0$ (because a plurisubharmonic function on a compact complex variety is constant). This is impossible since then $$0 =T_+ + T_- \ge (\varepsilon_+ + \varepsilon_-) \omega'.$$
\end{proof}

\begin{proposition}\label{NonPSEFFMFS}
Let $f \colon X \to S$ be a proper surjective morphism from a generalized K\"ahler pair $(X, B+\boldsymbol{\beta})$ to a relatively compact base $S$ with positive dimensional fibers.  
If $-(K_X+B+\boldsymbol{\beta}_X)$ is relatively K\"ahler over $S$, then the adjoint class
$K_X+B+\boldsymbol{\beta}_X$ is not relatively pseudo-effective over $S$.  
Moreover, if $S$ is compact, then $K_X+B+\boldsymbol{\beta}_X$ is not (absolutely) pseudo-effective on $X$.
\end{proposition}
\begin{proof}
Let $\omega$ be a K\"ahler class on $X$. Since $S$ is relatively compact, 
$-(K_X+B+{{\bbeta}_X}+\varepsilon\omega)$ is still relatively K\"ahler over $S$ for sufficiently small $0<\varepsilon\ll 1$. By Lemma \ref{Lemma: relativebig}, 
$K_X+B+{{\bbeta}_X}+\varepsilon\omega$ is not relatively big over $S$.

Therefore the pseudo-effective threshold is $\tau(X/S,K_X+B+ { \boldsymbol{\beta}_X};\omega)>0$. 
By Lemma \ref{nonpseffcriterion}, $K_X+B +  { {\bbeta}_X}$ is not relatively pseudo-effective over $S$. When $S$ is compact, Proposition \ref{AbsRelativepseff} implies that $K_X+B+\bbeta_X$ is not absolutely pseudo-effective on $X$ either. 
\end{proof} 
\begin{lemma}\label{Lemma: PSEFandMFS}
Let $X$ be a compact K\"ahler threefold with the K\"ahler class $\omega$, and $(X,B+\bbeta)$ a generalized klt pair. Assume that $K_X+B+\bbeta_X$ is not pseudo-effective and that the $(K_X+B+\bbeta_X)$-MMP with scaling of $\omega$, $$
(X,B+\bbeta_X) \dashrightarrow (X_1,B_1+\bbeta_{X_{1}}) \dashrightarrow  \cdots 
\dashrightarrow (X_N,B_N+\bbeta_{X_{N}}) = (X',B'+\bbeta_{X'})
$$ terminates with a Mori fiber space $f:X' \to Z'$.
Set 
$$
\tau := \inf \bigl\{ t \ge 0 \,\big|\, K_X + B + \bbeta_X + t\omega \text{ is pseudo-effective}\bigr\}(>0).
$$
Then the following hold:\\
(a) The Mori fiber space $X'\to Z'$ is $(K_{X'}+B'+ \bbeta_{X'} + \tau \omega')$-trivial;\\
(b) Conversely, if the Mori fiber space $X'\to Z'$ is $(K_{X'}+B'+ \bbeta_{X'} + \delta \omega')$-trivial for some $\delta>0$, then $\delta = \tau$ (i.e., $\delta$ is the pseudo-effective threshold $\tau$). 
\end{lemma}
\begin{proof}[Proof of (a)]
Running the $(K_X+B+\bbeta_X)$-MMP with scaling of $\omega$, we obtain a sequence
$$
(X,B+\bbeta_X) \dashrightarrow (X_1,B_1+\bbeta_{X_{1}}) \dashrightarrow  \cdots 
\dashrightarrow (X_N,B_N+\bbeta_{X_{N}}) = (X',B'+\bbeta_{X'}).
$$
Here we denote by $\omega_i$ the push forward of $\omega$ on $X_i$ and  
$$
\lambda_i := \inf \left\{ t \ge 0 \,\middle|\, K_{X_i}+B_i+\bbeta_{X_i}+t\omega_i \text{ is nef} \right\},
$$
where $\{\lambda_i\}$ is a non increasing sequence bounded below by $\tau$
$$
\lambda_1 \ge \lambda_2 \ge \cdots \ge \lambda_i \ge \cdots \ge \lambda_N \ge \tau>0.
$$

Next, we prove that $\lambda_N = \tau$. Since $K_{X_N}+B_N+\bbeta _{X_N}+\lambda _N\omega _N$ is nef, then $\lambda _N\geq \tau$, and it suffices to show that $\lambda_N \le \tau$.  
The MMP terminates with a $(K_{X'}+B'+\bbeta_{X'}+\lambda_N\omega')$-trivial Mori fiber space which is $(K_{X'}+B'+\bbeta_{X'})$-negative.  
If $\tau < \lambda_N$, then for the extremal ray $R$ contracted by the Mori fiber space
$X' \to Z'$, we have
$$
\bigl(K_{X'}+B'+\bbeta_{X'}+\tau\omega'\bigr)\cdot R < 0.
$$
This contradicts the pseudo-effectivity of
$K_{X'}+B'+\bbeta_{X'}+\tau\omega'$.  
Hence $\lambda_N = \tau$. 
\end{proof}

\begin{proof}[Proof of (b)]
Conversely, assume that the MMP ends with a $(K_{X'}+ B' + \bbeta_{X'}+ \delta \omega')$-trivial Mori fiber space.

We already showed in (a) that the Mori fiber space is $(K_{X'}+ B' + \bbeta_{X'} + \tau \omega')$-trivial and $(K_{X'}+ B' + \bbeta_{X'})$-negative. Thus
$$
\omega' \cdot R > 0, \quad (\delta - \tau)\,\omega' \cdot R = 0,
$$
and the only possible case is $\delta = \tau$.
\end{proof}
The following proposition was inspired by \cite[Lemma 2.35]{BAB}.

\begin{proposition}\label{pseffthreshold}
Let $X \to S$ be a family from a  K\"ahler variety $X$ onto a smooth, connected, and relatively compact curve $S$ such that $X_t$ admits canonical singularities for $t\in S$, and $\omega$ a K\"ahler class on $X$. 

If $K_X$ is not relatively pseudo-effective over an  open neighborhood $U \subset S$ of $0$, then $(K_X  + \delta \omega)$ is not relatively pseudo-effective over $U$ for some constant $\delta >0$ independent of the choice of $U$.
\end{proposition}

\begin{proof}
Since $K_{X}$ is not relatively pseudo-effective over $U$, there is an uncountable set $B\subset U$, such that $K_{X_t}$ are not pseudo-effective for $t\in B \subset U$. By Lemmata \ref{Lemma: NegCurve} and \ref{BDPPCanonical}, the fibers $X_t$ over $t\in B$ are covered by $K_{X_t}$-negative curves $\{C_{t,s}\}_{s\in S_t}$, where $S_t$ denotes the covering family. 

Since the number of irreducible components of the relative Douady space is countable, there exists a component $D$ of the relative Douady space containing $\{C_{t,s}\}$ for $t\in B' \subset B$ where $B'$ is an uncountable subset of $S$. Therefore there is a covering family of curves $\{C'_{t,s}\}$ (parameterized by the component $D$) that covers the fiber for each $t\in B'$, such that for the K\"ahler class $\omega$ on $X$, $C'_{t,s}\cdot \omega =K>0$ is fixed. Since $X_t$ admits canonical singularities for $t\in S$, $K_X$ is $\Q$-Cartier. Let $R$ be the Cartier index of $K_X$, then $K_{X_t}\cdot C'_{t,s}\leq - 1/R$. So $$(K_{X_t}+\frac{1}{RK+1}\omega_t)\cdot C'_{t,s}<0$$ and thus $K_{X_t}+\frac{1}{RK+1}\omega_t$ is not pseudo-effective for fibers $X_t$ at $t\in B'\subset U$. In particular, $K_X+\frac{1}{RK+1} \omega $ is not relatively pseudo-effective over $U$. 
\end{proof}


\section{Singularities and positivities of line bundles in  families}\label{Section: SingandPositivty}
In this section, we study the deformation behavior of various kinds of singularities and positivities in K\"ahler families. 
\subsection{Deformation of singularities}
We first study the behavior of singularities under deformation, generalizing several results from \cite{FH11}.

\begin{proposition}[{\cite[Proposition 3.5]{FH11}}]\label{DefCanSing}
Let $f: X \to S$ be a flat proper morphism from a generalized K\"ahler pair $(X, B+\boldsymbol{\beta})$ onto a smooth, connected curve $S$. Fix a point $0 \in S$ and assume that the boundary divisor $B$ does not contain the fiber $X_0$, and that the pair $(X_0, B_0+\boldsymbol{\beta}_0)$ has canonical singularities with $\lfloor B_0 \rfloor = 0$, where $(K_X + B +X_0 + \boldsymbol{\beta}_X)|_{X_0} = K_{X_0} + B_0 + \boldsymbol{\beta}_0$. Then, after shrinking $S$, the pair $(X, B+\boldsymbol{\beta})$ has canonical singularities, and $(X_t, B_t+\boldsymbol{\beta}_t)$ has canonical singularities for any $t \in S$.
\end{proposition}
\begin{proof}
By assumption, the pair $(X_0, B_0+  { \boldsymbol{\beta}}_0)$ is canonical and in particular, $\lfloor B_0\rfloor = 0$. By \cite[Theorem 0.1]{HP24}, the pair $(X,B+X_0+  { \boldsymbol{\beta}})$ is generalized plt near $X_0$.  
By the arguments of \cite[Corollary 1.4.3]{BirkarCHM}, it follows that $(X, B+ { \boldsymbol{\beta}})$ has canonical singularities {along $X_0$}, and $\left(X_t, B_t+ { \boldsymbol{\beta}}_t\right)$ has canonical singularities for any $t \in S$. 
\end{proof}

\begin{proposition}\label{DefKLT}
Let $f: X \rightarrow S$ be a flat proper morphism from a normal $\mathbb{Q}$-factorial 
generalized K\"ahler pair $(X, B+ { \boldsymbol{\beta}})$ onto a smooth, connected curve $S$. Fix a point $0\in S$ and assume that the boundary divisor $B$ does not contain the fiber $X_0$, and that the pair $(X_0, B_0+ { \boldsymbol{\beta}}_0)$ is generalized klt. Then, after shrinking $S$, the pair $(X, B+ { \boldsymbol{\beta}})$ is generalized klt and  $(X_t, B_t+\left. { \boldsymbol{\beta}}_t\right)$ has generalized klt singularities for any $t\in S$.
\end{proposition}

\begin{proof}
The proof is similar to the canonical singularities case. Since $(X_0,B_0 +  { \boldsymbol{\beta}}_0)$ is generalized klt, and $(X,B+X_0+ { \boldsymbol{\beta}})$ is a generalized pair. By inversion of adjunction \cite[Theorem 0.1]{HP24}, we know that $(X,B+X_0+  { \boldsymbol{\beta}})$ is generalized plt. Arguing as above, $(X,B+X_t+  { \boldsymbol{\beta}})$ is glc and since $X_t$ is Cartier, the pair $(X_t, B_t +  { \boldsymbol{\beta}}_t)$ is generalized klt for $t\in S$.
\end{proof}

\begin{remark}
By \cite{Kawamata99} it can happen that $(X_0, S_0)$ admits klt singularities but $(X,S)$ does not have klt singularities as $K_X+S$ is not $\mathbb R$-Cartier.  In other words, the $\mathbb R$-Cartier property is not well behaved for families of klt pairs. 
\end{remark}

\subsection{Relative ampleness between families}
Given a proper family $f : X \to S$ and a line bundle $L \in \mathrm{Pic}(X)$, 
if $L_0 = L|_{X_0}$ is ample on $X_0$, then there exists a Zariski open subset of $ S $ 
such that $ L_t = L|_{X_t} $ is ample on $ X_t $ for any $t$ in this subset (cf. \cite[Theorem 1.2.17]{Laza04}). 
This result easily generalizes to the relative setting (Proposition \ref{DefAmple}).
 
\begin{lemma}[{\cite[Lemma 1.5]{Naka86}, \cite[Lemma I.1.10]{Kollar96}}]\label{ResEmbed}
Let $p:X \to S$ and $g:Z \to  Y$ be proper morphisms of complex analytic varieties where $p$ is flat, and let $\Phi: X\to Z$ be a morphism such that the diagram commutes 
\begin{center}
\begin{tikzcd}[ampersand replacement=\&]
	X \&\& Z \\
	\& Y \\
	\& S.
	\arrow["\Phi", from=1-1, to=1-3]
	\arrow["f", from=1-1, to=2-2]
	\arrow["p"', from=1-1, to=3-2]
	\arrow["g"', from=1-3, to=2-2]
	\arrow["q", from=1-3, to=3-2]
	\arrow["r"{description}, from=2-2, to=3-2]
\end{tikzcd}
\end{center}
Fix a point $0\in S$ and assume that the restriction of the diagram $\Phi_0: X_0 \rightarrow Z_0/Y_0$ over $0\in S$ is a closed embedding. Then $\Phi:X\to Z /Y$ is a closed embedding over some open neighborhood $U\subset S$ of $0$. 
\end{lemma}

\begin{proposition}\label{DefAmple}
Let $f: X \to Y/S$ be a proper surjective contraction morphism of complex analytic varieties over a smooth connected curve $S$ with $X\to S$ proper. For each $t \in S$, denote by
\[
f_t : X_t \to Y_t
\]
the induced proper contraction on the fiber over $t$. 

(1) If $L$ is a $\mathbb{Q}$-line bundle on $X$ such that $L|_{X_0}$ is $f_0$-ample, then there exists a Zariski open subset $U \subset S$ around $0$ such that $L|_{X_t}$ is $f_t$-ample for any $t \in U$.

(2) If $\alpha\in H^{1,1}_{\rm{BC}}(X, \R)$ is a class on $X$ such that $\alpha|_{X_0}$ is $f_0$-K\"ahler, then there exists a (Euclidean) open subset $U\subset S$ around $0$ such that $\alpha|_{X_t}$ is $f_t$-K\"ahler for any $t\in U$. 
\end{proposition}
\begin{proof}
Since $L_0 = L|_{X_0}$ is $f_0$-ample,  there exists some $m\ge 0$ such that
$$\rho_m : f^* f_* (L^{\otimes m})\to L^{\otimes m}$$is surjective near $X_0$. To see this, consider the exact sequence 
\[
0 \to L^{\otimes m} \otimes \mathcal{I}_{X_0} \to L^{\otimes m} \to L^{\otimes m}|_{X_0} \to 0.
\]
Since $L|_{X_0}$ is $f_0$-ample, Serre’s vanishing theorem gives 
\[
R^i (f_0)_* \bigl(L^{\otimes m}|_{X_0}\bigr) = 0 \quad \text{for } i>0 \text{ and } m \gg 0.
\]
Therefore,
\[
\psi: R^1 f_* (L^{\otimes m}) \otimes \mathcal{I}_{\{0\}} 
\cong R^1 f_*\bigl(L^{\otimes m} \otimes \mathcal{I}_{X_0}\bigr) 
\longrightarrow R^1 f_* (L^{\otimes m})
\]
is surjective for $m \gg 0$, where the first isomorphism follows from the projection formula together with the fact that $\mathcal{I}_{X_0}=f^*\mathcal{I}_{\{0\}}$ and $\mathcal{I}_{\{0\}}$ are invertible on the smooth connected curve $S$. Taking the stalk at $0$ and tensoring with $\kappa(0)\cong \mathbb{C}$, we obtain a surjective map
\[
\psi(0): \;  R^1 f_* (L^{\otimes m})_0 \otimes \bigl(\mathfrak{m}_0 / \mathfrak{m}_0^2\bigr)
\longrightarrow R^1 f_* (L^{\otimes m})_0 \otimes \kappa(0).
\]
Note that this is the zero map, since the composition in the short exact sequence 
\[
0\to \mathfrak{m}_{0} /\mathfrak{m}_{0}^2 \to \mathcal{O}_{S,0} /\mathfrak{m}_{0}^2 \to \mathcal{O}_{S,0} /\mathfrak{m}_{0} = \kappa (0) \to 0
\] 
is zero. Therefore, $R^1 f_* (L^{\otimes m})_0 \otimes \kappa(0) = 0$ as well. By Nakayama’s lemma, we conclude that 
\[
R^1 f_* (L^{\otimes m})_0 = 0.
\]
Since $R^1 f_* (L^{\otimes m})$ is coherent, its support is a closed analytic subset. Therefore, $R^1 f_* (L^{\otimes m}) = 0$ in a neighborhood of $0$. Consequently,
\[
f_* L^{\otimes m} \longrightarrow (f_0)_* \bigl(L^{\otimes m}|_{X_0}\bigr)
\]
is surjective. Since $L^{\otimes m}|_{X_0}$ is $f_0$-globally generated for $m\gg 0$, i.e., the map
\[
f_0^* (f_0)_* (L^{\otimes m} |_{X_0} ) \to L^{\otimes m}|_{X_0}
\]
is surjective for $m\gg 0$, the morphism
\[
\rho_m: f^* f_* L^{\otimes m} \longrightarrow L^{\otimes m}, \quad \text{for } m\gg 0
\]
is surjective in a neighborhood of $X_0$. 

Therefore, after shrinking $S$, it will induce the commutative diagram  
\begin{center}
\begin{tikzcd}[ampersand replacement=\&]
	X \&\& {\mathbb{P}_Y(f_* L^{\otimes m})} \\
	\& Y \\
	\& S.
	\arrow["\Phi", from=1-1, to=1-3]
	\arrow["f"{description}, from=1-1, to=2-2]
	\arrow["g"', from=1-1, to=3-2]
	\arrow[from=1-3, to=2-2]
	\arrow[from=1-3, to=3-2]
	\arrow["h"{description}, from=2-2, to=3-2]
\end{tikzcd}
\end{center}
Since $L|_{X_0}$ is $f_0$-ample, the restriction to the fiber over $0 \in S$ induces an embedding
$$\Phi_0 = \Phi|_{X_0}: X_0 \hookrightarrow \mathbb{P}_{Y_0}((f_0)_* (L^{\otimes m}|_{X_0}))/Y_0.$$
By Lemma \ref{ResEmbed}, after further shrinking the base $S$, $\Phi:X \hookrightarrow \mathbb{P}_{Y}(f_* L^{\otimes m})/Y$ is still a closed embedding.

The second part of the proposition is clear, since positive form remains positive when restricted to the nearby fibers. 
\end{proof}
\subsection{Relative pseudo-effectivity of line bundles in projective families}
Using a recent result of J. Jiao \cite{Jiao}, we prove that the global stability of pseudo-effectivity of a line bundle over a projective family with irreducible and reduced fibers, if this line bundle is pseudo-effective on a Zariski dense set of fibers. 

\begin{proposition}\label{p-psefrest}
Let $f:X\to S$ be a projective flat morphism of complex analytic spaces with reduced and irreducible fibers. If $L$ is a line bundle on $X$ whose restriction to an analytic Zariski  dense set of fibers is pseudo-effective, then $L|_{X_s}$ is pseudo-effective for every $s \in S$ and $L$ is thus $f$-pseudo-effective over $S$. 
\end{proposition}
\begin{proof}
Let $A$ be a relatively ample line bundle on $X$, set $n=\dim_{\mathbb C}(X/S)$, and let $B\subset S$ be an analytic Zariski dense subset such that $L|_{X_s}$ is pseudo-effective for every $s\in B$. Since $f$ is flat, $v:=(A|_{X_s})^{n}$ is independent of $s$. Put
\[
P:=\{\, s\in S \;:\; L|_{X_s}\ \text{is pseudo-effective} \,\}.
\]
Then $B\subset P$. We show that $P$ is a closed analytic subset of $S$. Since this is a local property, it suffices to check it near a given point $s_1\in S$.

By Lemma~\ref{Hilbertscheme}, after replacing $S$ by a small open neighborhood of $s_1$, the family $f\colon (X,L,A)\to S$ is the pullback, via a morphism $\phi\colon S\to (S')^{\mathrm{an}}$, of the analytification of $f'\colon (X',L',A')\to S'$, where $S'$ is a projective $\mathbb{C}$-scheme and $A'$ is $f'$-ample. For every $s\in S$ the fiber $X_s$ is the analytification of $X'_{\phi(s)}$, with $L|_{X_s}$ and $A|_{X_s}$ the analytifications of $L'|_{X'_{\phi(s)}}$ and $A'|_{X'_{\phi(s)}}$. By GAGA, 
$$P=\phi^{-1}(P')^{\mathrm{an}},$$ 
where $P':=\{\, s'\in S' \;:\; L'|_{X'_{s'}}\ \text{is pseudo-effective} \,\}$.

We claim that $P'$ is Zariski closed in $S'$. Let $Z'$ be the Zariski closure of $P'$ in $S'$, and let $Z'_1,\dots,Z'_r$ be the irreducible components of $Z'$, with generic points $\eta_1,\dots,\eta_r$. So for each $i$, $P'\cap Z'_i$ is Zariski dense in $Z'_i$. Fix $\varepsilon>0$. By Lemma~\ref{vollemma},
\[
\mathrm{vol}\big((L'+\varepsilon A')|_{X'_{s'}}\big)\ \ge\ \varepsilon^{n}v \qquad \text{for every } s'\in P'.
\]
Applying \cite[Theorem 1.1]{Jiao} to the restricted family $X'\times_{S'}Z'_i\to Z'_i$ and the Zariski dense subset $P'\cap Z'_i$, we obtain
\[
\mathrm{vol}\big((L'+\varepsilon A')|_{X'_{\eta_i}}\big)
=\inf_{s'\in P'\cap Z'_i}\mathrm{vol}\big((L'+\varepsilon A')|_{X'_{s'}}\big)\ \ge\ \varepsilon^{n}v>0.
\]
In particular, $L'+\varepsilon A'$ is big on the generic fiber $X'_{\eta_i}$, and hence on $X'_{s'}$ for every $s'$ in a dense Zariski open subset of $Z'_i$. By \cite[Proposition~3.4]{RT2}, $(L'+\varepsilon A')|_{X'_{s'}}$ is big for every $s'\in Z'_i$. Since $\varepsilon>0$ is arbitrary, $L'|_{X'_{s'}}$ is pseudo-effective for every $s'\in Z'_i$. Therefore $Z'\subset P'$, so $P'=Z'$ is Zariski closed in $S'$, and $P=\phi^{-1}(Z'^{\mathrm{an}})$ is a closed analytic subset of $S$.

Thus $P$ is a closed analytic subset of $S$ containing the analytic Zariski dense set $B$, so $P=S$. In other words, $L|_{X_s}$ is pseudo-effective for every $s\in S$, and in particular,  $L$ is $f$-pseudo-effective over $S$ by definition.
\end{proof}

\begin{lemma}\label{Hilbertscheme}
Let $f: X \to S$ be a flat projective morphism of complex analytic
spaces with reduced and irreducible fibers, and let $L$ be a line bundle
on $X$. Then every point $s_0\in S$ has an open neighborhood $S_0\subset S$ for which there exist a projective morphism of schemes $f': X' \to S'$ over $\mathbb{C}$, a line bundle $L'$ on $X'$, and a dominant morphism $\phi\colon S_0 \to (S')^{\mathrm{an}}$, such that $(f, L)|_{S_0}$ is the pullback along $\phi$ of the analytification of $(f', L')$. Here
$(-)^{\mathrm{an}}$ denotes the analytification functor.
\end{lemma}

\begin{proof}
Shrinking $S$ around $s_0$, we may assume that $S$ is Stein and relatively compact in a larger Stein open set over which $f$ factors through a closed embedding $X \hookrightarrow \mathbb{P}^N_S$. Set $A := \mathcal{O}_{\mathbb{P}^N_S}(1)|_X$. For $m \gg 0$, $M := L \otimes A^{\otimes m}$ is $f$-very ample and $R^i f_* M = 0$ for $i>0$, so $f_*M$ is locally free of rank $r>\dim S$ with $f_*M\otimes\kappa(s)\cong H^0(X_s,M|_{X_s})$, the zero locus of a general section $\sigma$ is a relative effective Cartier divisor $D\subset X$, flat over $S$, with $\mathcal{O}_X(D)\cong M$.

The fibers $X_s \subset \mathbb{P}^N$ form a flat family, and the
Hilbert polynomial
\[
\Phi(k) = \chi\!\left(X_s,\, A^{\otimes k}\big|_{X_s}\right)
\]
is locally constant in $s$,
likewise the pair $(D_s \subset X_s)$ has a fixed Hilbert polynomial $\Psi$.

Let $H := \operatorname{Hilb}^{\Phi,\Psi}(\mathbb{P}^N)$ be the flag Hilbert scheme
parametrizing pairs $(W \subset Z \subset \mathbb{P}^N)$ with Hilbert
polynomials $\Phi,\Psi$, and let $(\mathcal{X},\mathcal{D})\to H$ be the universal family. The analytification $H^{\mathrm{an}}$ represents the corresponding flag Douady functor (cf.\ \cite[Theorem 4.5.1]{Sernesi}, note that the cited
result is stated for the Hilbert scheme, but the proof adapts directly to the Douady
space). The family $(X,D)\to S$
therefore induces a classifying morphism of analytic spaces
\[
\phi \colon S \to H^{\mathrm{an}}, \qquad s \mapsto [\,D_s \subset X_s\,],
\]
and $(X, D) \to S$ is the pullback under $\phi$ of the universal
family $((\mathcal{X}, \mathcal{D}) \to H)^{\mathrm{an}}$. Moreover,
$\mathcal{O}_{\mathcal{X}/H}(1)$ pulls back to $A$ and
$\mathcal{O}_{\mathcal{X}}(\mathcal{D})$ pulls back to $\mathcal{O}_X(D) \cong M$.

Let $(S')^{an}$ be the Zariski closure of $\phi(S)$ in $H^{\mathrm{an}}$, which corresponds to a projective subscheme $S' \subset H$.
Restricting the universal family $(\mathcal{X},\mathcal{D})\to H$ to $S'\subset H$ yields a projective flat morphism of
schemes $f' \colon X' \to S'$, and we set
\[
L' := \mathcal{O}_{X'}(\mathcal{D}') \otimes \mathcal{O}_{X'/S'}(-m).
\]
Then $(f' \colon X' \to S',\, L')^{an}$ pulls back under $\phi$ to
$(f \colon X \to S,L)$,
as claimed.
\end{proof}

\subsection{Global generation and semi-ampleness of line bundles in families}
Next, we discuss the global generation of line bundles in families.
\begin{lemma}\label{ResBasechange}
Let $f: X \to S$ be a proper surjective morphism of normal complex analytic spaces, and $L$ a line bundle on $X$. Assume that $S$ is Stein.
Then for any $s\in S$, the base change map $$\varphi^0(s): (f_* L)_s \otimes_{\mathcal{O}_{S,s}} \kappa(s) \;\longrightarrow\; H^0(X_s, L|_{X_s})$$ is surjective if and only if the restriction map $$r: H^0(X,L) \to H^0(X_s, L|_{X_s})$$ is surjective. 

\end{lemma}
\begin{proof}
By Cartan's Theorem A, the direct image sheaf $f_*L$ is globally generated on $S$. Therefore the evaluation map $\alpha_s$ in the diagram is surjective
\[
\begin{tikzcd}[ampersand replacement=\&]
	{H^0(S,f_* L)} \& {(f_*L)(s)} \\
	{H^0(X,L)} \& {H^0(X_s,L|_{X_s})}.
	\arrow["{\alpha_s}", from=1-1, to=1-2]
	\arrow["{=}"', from=1-1, to=2-1]
	\arrow["{\varphi^0(s)}", from=1-2, to=2-2]
	\arrow["r", from=2-1, to=2-2]
\end{tikzcd}
\]
Thus, $\varphi^0(s)$ is surjective if and only if $r$ is surjective.
\end{proof}
\begin{lemma}\label{BsoverStein}
With the same setting as in Lemma \ref{ResBasechange}, we have
\[
\operatorname{coker}\!\bigl(f^* f_* L \to L\bigr) \;=\; \operatorname{coker}\!\bigl(H^0(X,L)\otimes \mathcal{O}_X \to L\bigr).
\]
In particular,
\[
\operatorname{Bs}(L/S) \;=\; \operatorname{Bs}(L) \;=\; \bigcap_{\sigma \in H^0(X,L)} V(\sigma).
\]
Here the \emph{base loci} $\operatorname{Bs}(L/S)$ and $\operatorname{Bs}(L)$ of $L/S$ and $L$ are defined as $$\operatorname{Bs}(L/S) : = \mathrm{supp}\operatorname{coker}\!\bigl(f^* f_* L \to L\bigr)$$ and $$\operatorname{Bs}(L):=\mathrm{supp}\operatorname{coker}\!\bigl(H^0(X,L)\otimes \mathcal{O}_X \to L\bigr),$$ respectively. 
\end{lemma}
\begin{proof}
By Cartan's Theorem~A, the coherent sheaf $f_*L$ is globally generated on $S$. That is, the natural map 
\[
H^0(X,L)\otimes \mathcal{O}_S \;=\; H^0(S,f_*L)\otimes \mathcal{O}_S \;\longrightarrow\; f_*L
\]
is surjective. Since pullback via $f^*$ preserves surjectivity, we obtain the surjective morphism
\[
H^0(X,L)\otimes \mathcal{O}_X \;\longrightarrow\; f^* f_* L.
\]
So $
\operatorname{coker}\!\bigl(f^* f_* L \to L\bigr) \;=\; 
\operatorname{coker}\!\bigl(H^0(X,L)\otimes \mathcal{O}_X \to L\bigr)$. 
\end{proof}

Based on Lemmata \ref{ResBasechange} and \ref{BsoverStein}, we obtain the equivalence between relative global generation and fiberwise global generation of line bundles under the base change assumption.
\begin{proposition}\label{Bslocus}

Let $f: X \to S$ be a proper surjective morphism of normal complex analytic spaces, and $L$ a line bundle on $X$. 
Assume that for every $s \in S$, the base change map
\[
\varphi^0(s): (f_* L)_s \otimes_{\mathcal{O}_{S,s}} \kappa(s) \;\longrightarrow\; H^0(X_s, L|_{X_s})
\]
is an isomorphism, where $X_s = f^{-1}(s)$ denotes the fiber over $s$.

Then the natural morphism
$f^{*} f_{*} L \;\longrightarrow\; L$ is surjective if and only if $L|_{X_s}$ is globally generated for every $s \in S$.
\end{proposition}

\begin{proof}
Since surjectivity of a sheaf homomorphism can be checked on an open cover, we may, without loss of generality, assume that the base $S$ is Stein. Let $x\in X_s$ be a fixed point on the fiber $X_s$ for any $s\in S$. 
Denote by $L(x)$ and $(L|_{X_s})(x)$ the fibers of $L$ and the restriction $L|_{X_s}$ at $x \in X$, respectively. Clearly, both fibers are isomorphic to $\mathbb{C}$. 
Consider the commutative diagram (in the category of $\mathbb{C}$-vector spaces)
\[
\begin{tikzcd}
H^0(X, L)  \arrow[r, "\widetilde{\alpha}_x"] \arrow[d,"r"] & L(x) \cong \mathbb{C} \arrow[d,"\cong"] \\
H^0(X_s, L|_{X_s})  \arrow[r, "{\alpha}_x"] & {(L|_{X_s})}(x) \cong \mathbb{C},
\end{tikzcd}
\]
where $\widetilde{\alpha}_x$ (resp. ${\alpha}_x$) denotes the evaluation map 
\[
\widetilde{\alpha}_x (\widetilde{\sigma}) = \widetilde{\sigma}(x), \qquad \text{(resp. ${\alpha}_x (\sigma) = \sigma(x)$)},
\]
for $\widetilde{\sigma} \in H^0(X,L)$ (resp. $\sigma\in H^0(X_s,L|_{X_s})$), while $r$ is the natural restriction. 
By assumption, the base-change map 
\[
\varphi^0(s): (f_*L)\otimes \kappa(s) \longrightarrow H^0(X_s, L|_{X_s})
\]
is an isomorphism for every $s \in S$. Thus, by Lemma~\ref{ResBasechange}, the restriction map $r$ is surjective, and hence the cokernels of $\widetilde{\alpha}_x$ and ${\alpha}_x$ coincide.  

Note that $\widetilde{\alpha}_x$ is surjective precisely when there exists a section $\widetilde{\sigma} \in H^0(X,L)$ such that $\widetilde{\alpha}_x(\widetilde{\sigma}) \neq 0$. Similarly, ${\alpha}_x$ is surjective if and only if there exists a section $\sigma \in H^0(X_s, L|_{X_s})$ with ${\alpha}_x(\sigma) \neq 0$.
Moreover, 
$\operatorname{Bs}(L|_{X_s}) \;=\; \bigcap_{\sigma\in H^0(X_s, L|_{X_s})} V(\sigma)$ and by Lemma~\ref{BsoverStein},  
$$\operatorname{Bs}(L/S) \;=\; \bigcap_{\widetilde{\sigma}\in H^0(X,L)} V(\widetilde{\sigma}).
$$
Therefore, ${\alpha}_x$ is surjective if and only if $x \notin \operatorname{Bs}(L|_{X_s})$, and $\widetilde{\alpha}_x$ is surjective if and only if $x \notin \operatorname{Bs}(L/S)|_{X_s}$. It follows that the restriction $L|_{X_s}$ is globally generated at $x \in X_s$ if and only if the map 
\[
f^*f_*L \longrightarrow L
\]
is surjective at $x \in X_s$.
\end{proof}
\begin{remark}\label{SteinBslocus}
Let $f: X \to S$ be a proper surjective morphism of normal complex analytic spaces with Stein $S$ and $L$ a line bundle on $X$. If we assume that at some $s \in S$, the base change map $$\varphi^0(s): (f_* L)_s \otimes_{\mathcal{O}_{S,s}} \kappa(s) \;\longrightarrow\; H^0(X_s, L|_{X_s})$$ is surjective, then
Proposition \ref{Bslocus} shows that for any $x\in X_s$, the following are equivalent:\\
(a) The natural morphism $(f^* f_* L)_x \to L_x$ is surjective;\\
(b) The natural morphism $H^0(X_s,L|_{X_s}) \otimes \mathcal{O}_{X_s ,x} \to (L|_{X_s})_x$ is surjective;\\
(c) The evaluation map $H^0(X_s, L|_{X_s}) \to (L|_{X_s})(x)$ is surjective. \\
Moreover, we have $x \in \mathrm{Bs}(L|_{X_s})$ if and only if $x \in \mathrm{Bs}(L/S)\cap X_s$.

\end{remark}
\begin{remark}\label{Grauertbasechange}
Let $f:X\to S$ be a flat proper morphism of normal complex analytic spaces, and $L$ a line bundle on $X$. Then \cite[Corollary~III.3.7]{BS} gives the equivalent conditions for the base change maps to be surjective: \\
(1) The line bundle $L$ satisfies the \emph{base change property in degree $0$}, i.e., $$ \varphi^0(s):  f_*(L) \otimes_{\mathcal{O}_{S,s}} \kappa(s) \rightarrow H^0\left(X_s, L|_{X_s}\right)$$is surjective for any $s\in S$.\\
(2) The $0$-cohomological dimension function $$s \mapsto h^0(X_s, L|_{X_s})$$ is locally constant over $S$.\\
(3) The line bundle $L$ is \emph{cohomologically flat in dimension 0}, i.e., the functor $$F: \text{Coh}_S \to \text{Coh}_S,\quad  \mathcal{M} \mapsto f_*(L \otimes_{\mathcal{O}_X} f^*\mathcal{M})$$is exact. In what follows, we will use the equivalent characterization  implicitly.
\end{remark}

As direct corollaries of Proposition \ref{Bslocus}, we first give several equivalent characterizations of relative global generation of line bundles. 
\begin{corollary}\label{equi-rgg}
Let $f:X\to \Delta$ be a flat proper morphism from a normal complex analytic variety onto the unit disk in $\mathbb{C}$, and $L$ a line bundle on $X$. Then the following are equivalent: \\
\emph{(a)} The natural morphism $f^* f_* L \to L$ is surjective over $f^{-1}(U)$, for some non-empty Zariski open subset $U \subset \Delta$;\\
\emph{(b)} $L|_{X_t}$ is globally generated for any $t \in V$, where $V \subset \Delta$ is a non-empty Zariski open subset;\\
\emph{(c)} There exists a subset $E\subset \Delta$ with an accumulation point contained in $\Delta$, such that $L|_{X_t}$ is globally generated for $t\in E$;\\
\emph{(d)} There exists a point $s_0\in \Delta$, such that the base change map $$\varphi^0(s_0): (f_* L)(s_0)\to H^0(X_{s_0}, L|_{X_{s_0}})$$is surjective, and $L|_{X_{s_0}}$ is globally generated. 
\end{corollary}

\begin{proof}
$(a) \Rightarrow (b)$. Since Grauert's base change theorem holds over some Zariski open subset $W\subset \Delta$, $W\cap U = V $ is a non-empty Zariski open subset of $\Delta$ over which $f^*f_* L \to L$ is surjective and the base change property holds. Thus, by Proposition \ref{Bslocus}, $L|_{X_t}$ is globally generated for any $t\in V$. \\
$(b) \Rightarrow (c)$ is clear, since given any point $t_0\in V$ we can always find a countable sequence of points $E = \{t_i\}\subset V$ that converge to $t_0$. \\
$(c) \Rightarrow (d)$. There exists some non-empty Zariski open subset $W\subset \Delta$, such that Grauert's base change holds. Since $E\cap W \ne \emptyset$ (by the identity principle), there exists a point $s\in E\cap W$ such that the base change map $\varphi^0 (s)$ is surjective and $L|_{X_s}$ is globally generated.\\
$(d) \Rightarrow (a)$. By Grauert base change theorem \cite[Corollary III.4.8]{BS},  
$$W = \{s\in \Delta \mid \varphi^0(s): (f_* L)(s)\to H^0(X_s, L|_{X_s}) \text{ is surjective}\}$$ 
is a non-empty Zariski open subset of $\Delta$, which is Stein. By assumption (d), at the point $s_0\in W\subset \Delta$, $\varphi^0(s_0)$ is surjective and $L|_{X_{s_0}}$ is globally generated. Therefore, by Remark \ref{SteinBslocus} 
$$\mathrm{Bs}(L/\Delta)|_{X_{s_0}} = \mathrm{Bs}(L|_{X_{s_0}}) = \emptyset.$$

In particular, $U = W \setminus f(\mathrm{Bs}(L/\Delta))$ is a non-empty Zariski open subset of $\Delta$, such that $f^*f_* L \to L$ is surjective on $U$.
\end{proof}

Furthermore, as a more special corollary of Proposition \ref{Bslocus}, one establishes the Zariski openness of global generation for line bundles under the assumption that the dimension of their global sections is deformation invariant.
\begin{corollary}\label{opensemiample}
Let $ f: X \to S $ be a flat proper morphism from a normal complex analytic variety $ X $ onto a smooth, connected, and relatively compact curve $S$, and  $ L $ a line bundle on $ X $. Assume that $ L|_{X_0} $ is globally generated and that the dimension $ h^0(X_s, L|_{X_s}) $ is constant for every $ s \in S $. Then there exists a dense Zariski open subset $ U \subset S $ such that $ L|_{X_s} $ is globally generated for any $ s \in U $.
\end{corollary} 
\begin{proof} Although the proof is implicit in that of Corollary \ref{equi-rgg}, we still include a direct proof here. 

{Since $\mathrm{Bs}(L/S)$ is an analytic closed subset of $X$ and $f$ is proper, $W := f(\mathrm{Bs}(L/S)) \subset S$ is a proper analytic closed subset not containing $0$. For every $s\in U:= S\setminus W$, we may therefore shrink $S$ to a Stein open neighborhood $U'$ of $s$ disjoint from $W$. By Remark~\ref{SteinBslocus},
$$\mathrm{Bs}(L/S)|_{X_s} = \mathrm{Bs}(L|_{X_s}) = \emptyset, \quad \text{for every } s \in U'.$$
In particular, $L|_{X_s}$ is globally generated for every $s\in U$.}
\end{proof}
It is worth noting that the global generation criterion for line bundles developed above (Proposition~\ref{Bslocus}) also applies to relatively flat coherent analytic sheaves of rank $1$.
At the end of this subsection, we apply Proposition~\ref{Bslocus} to the semi-ampleness of canonical sheaves.
\begin{definition}[Semi-ample coherent sheaves, semi-ample (1,1)-classes]\label{DefinitionSemiample}
Let $X$ be a normal complex analytic variety. Then:\\
\emph{(1)} A coherent analytic sheaf $\mathcal{S}$ on $X$ is \emph{semi-ample} if there exists some sufficiently large $m \gg 0$ such that $\mathcal{S}^{\otimes m}$ is globally generated. \\
\emph{(2)} Let $f: X \rightarrow S$ be a proper surjective morphism of normal complex analytic spaces, and $\mathcal{S}$ a coherent analytic sheaf on $X$. We say that $\mathcal{S}$ is \emph{$f$-semi-ample} (or \emph{relatively semi-ample}) if there exists an integer $m>0$ such that the natural evaluation map
$$
f^* f_*\left(\mathcal{S}^{\otimes m}\right) \longrightarrow \mathcal{S}^{\otimes m}
$$
is surjective as a morphism of sheaves on $X$.\\
\emph{(3)} A class $\alpha \in H^{1,1}_{\mathrm{BC}}(X,\mathbb{R})$ is a \emph{semi-ample $(1,1)$-class} if there exists a contraction morphism  $f : X \to Z$ 
together with a K\"ahler form $\omega_Z$ on the K\"ahler variety  $Z$ such that  
\[
f^*([\omega_Z]) = \alpha.
\] 
\end{definition}

Now let us strengthen \cite[Corollaries 1.2, 1.12]{LRW25} via a different approach. 
\begin{proposition}\label{defsemiample}
Let $f: X\to S$ be a proper flat morphism onto a connected complex
manifold $S$. Assume that $f$ is either a smooth family of projective
manifolds, or a smooth K\"ahler family of threefolds, or a family whose
fibers admit canonical singularities and uncountably many of whose
fibers are of general type. If $K_{X_0}$ is semi-ample for some point
$0\in S$, then $K_{X_t}$ is semi-ample for general $t\in S$.

Moreover, if $f$ is a smooth family of projective manifolds or a smooth
K\"ahler family of threefolds, then $K_X$ is $f$-semi-ample; in particular,
$K_{X_t}$ is semi-ample for every $t\in S$.
\end{proposition}
\begin{proof}
By \cite[Theorem 1.2]{RT2} and \cite[Corollary 1.12]{LRW25} (or Theorem \ref{DefVol3fold} here for the K\"ahler case), the deformation invariance of plurigenera holds for each family in the assumption. Then Corollary \ref{opensemiample} implies that the canonical sheaf of a general fiber of $f$ is semi-ample (and hence nef).

In particular, when $f$ is a smooth projective family or a smooth K\"ahler family (\cite[Theorem 1.2]{LLR26}, and \cite[Theorems 1.1, 1.10]{LRW25}) gives the deformation closedness of nefness. Thus, the canonical divisors of all fibers of $f$ are nef. Then Kawamata's relative base point free theorem \cite{Kawa85} (also \cite[Theorem 5.8]{Naka86} and \cite[Theorem 1]{FujinoKawa}) yields that $K_X$ is $f$-semi-ample. By Proposition \ref{Bslocus}, this is equivalent to saying that the canonical divisor of any fiber of $f$ is semi-ample.  
\end{proof}

\section{Extension of the analytic MMP and stability of nefness and bigness}\label{Section: ExtMMP}
The relative minimal model program is a powerful tool in the study of deformation problems. The purpose of this section is to study the behavior of the minimal model program under deformations. The definitions below follow \cite{DHY23}. 
\begin{definition}[Divisorial contractions]\label{DefinitionDivisorial}
Let $(X,B+\boldsymbol{\beta})$ be a generalized pair.
A $(K_X+B+\boldsymbol{\beta}_X)$-\emph{divisorial contraction} $f:X\to Z$ is a bimeromorphic morphism with the relative Picard number $\rho(X/Z) = 1$ and exceptional locus being an (irreducible) divisor, such that $-(K_X+B+ \boldsymbol{\beta}_X)$ is K\"ahler over $Z$.
\end{definition}

\begin{definition}[Small bimeromorphic maps]
Let $X$ and $Y$ be complex analytic varieties. A bimeromorphic map $\phi:X \dashrightarrow Y$ is a \emph{bimeromorphic contraction} if there exists a resolution $p: W \longrightarrow X$ and $q: W \longrightarrow Y$ of $\phi$ such that $p$ and $q$ are contraction morphisms and every $p$-exceptional divisor is $q$-exceptional. We say that $\phi$ is a \emph{small bimeromorphic map} if both $\phi$ and $\phi^{-1}$ are bimeromorphic contractions. 
\end{definition}

\begin{definition}[Flipping contractions, flips]\label{DefinitionFlip} Let $(X,B+\boldsymbol{\beta})$ be a generalized pair.
A $(K_X+B+\boldsymbol{\beta}_X)$-\emph{flipping contraction} $f: X \rightarrow Z$ is a small bimeromorphic morphism such that $\rho(X / Z)=1$ and $-\left(K_X+B+ { \boldsymbol{\beta}}_X\right)$ is K\"ahler over $Z$. The corresponding \emph{flip} (if it exists) is a small bimeromorphic morphism $f^{+}: X^{+} \rightarrow$ $Z$ such that $\rho\left(X^{+} / Z\right)=1$ and $K_{X^{+}}+B^{+}+ \boldsymbol{\beta}_{X^{+}}$ is  K\"ahler over $Z$, where $B^{+}$ is the strict transform of $B$.
\end{definition}
\begin{remark}
By \cite[Remark 5.13]{DH24}, the flipping contraction is a locally projective morphism. When the transcendental class satisfies $ { \boldsymbol{\beta}}_X = 0 $ in Definitions \ref{DefinitionDivisorial} and \ref{DefinitionFlip}, the divisorial and flipping contractions are projective morphisms.
\end{remark}
\begin{definition}[{Mori fiber spaces of a generalized K\"ahler pair}]\label{MFSKahler}
If $(X / S, B+ { \boldsymbol{\beta}})$ is a generalized dlt pair over a relatively compact analytic variety $S$, then we say that $f:(X/S,B + { \boldsymbol{\beta}})\to Z/S$ is a \emph{Mori fiber space of $(X,B+  { \boldsymbol{\beta}})$} over $S$ if $f$ is a contraction morphism associated to a $(K_X+B+ { {\bbeta}_X})$-negative extremal face such that $-(K_X+B+ { {\bbeta}_X})$ is $f$-K\"ahler and $\dim Z < \dim X$.
\end{definition}

\begin{remark}
Some references require the additional assumption that relative Picard number satisfies $ \rho(X / S) = \rho(Z / S) + 1 $ (e.g., \cite{DHY23}). For simplicity and without causing any problem, we omit this condition in our discussion.
\end{remark}
\subsection{Extension of the analytic MMP to nearby fibers} By \cite{KM92}, we have the crucial extension result: 
\begin{theorem}[{Extension of a contraction morphism, \cite[Proposition (11.4)]{KM92}}]\label{ExtendCont}
Let $\phi_0: X_0 \to X_{0}'$ be a proper surjective morphism of normal compact complex analytic spaces. Assume that $(\phi_0)_* \mathcal{O}_{X_0} = \mathcal{O}_{X_0'}$ and $R^1(\phi_0)_* \mathcal{O}_{X_0} = 0$. If $f:X\to S$ is a small deformation of $X_0$, then $\phi_0$ can be extended to a contraction morphism $\phi: X\to X'$ over $S$:
\[\begin{tikzcd}
	{X_0} && X \\
	& {X'_0} & {} & {X'} \\
	{\{0\}} && {S.}
	\arrow[hook, from=1-1, to=1-3]
	\arrow["{\phi_0}", from=1-1, to=2-2]
	\arrow[from=1-1, to=3-1]
	\arrow["\phi", from=1-3, to=2-4]
	\arrow[from=1-3, to=3-3]
	\arrow[hook, no head, from=2-2, to=2-3]
	\arrow[from=2-2, to=3-1]
	\arrow[from=2-3, to=2-4]
	\arrow[from=2-4, to=3-3]
	\arrow[hook, from=3-1, to=3-3]
\end{tikzcd}\]
\end{theorem}
\begin{remark}
The result does not hold without the condition  `$R^1 (\phi_0)_* \mathcal{O}_{X_0} = 0$'. Consider the product of abelian varieties $ \pi_2 : A_1 \times A_2 \to A_2 $; note that
\[
R^1 (\pi_2)_* \mathcal{O}_{A_1 \times A_2} \neq 0,
\]
since $\dim H^1(A_1,\mathcal{O}_{A_1}) = \dim A_1  >0$. In this case, a general deformation of $ A_1 \times A_2 $ is a simple complex torus $ \mathcal{A} $, and there cannot exist any non-trivial morphism to a lower-dimensional abelian variety $ \mathcal{A}\to \mathcal{A}' $.
\end{remark}


We now apply Theorem \ref{ExtendCont} to the study of the minimal model program under deformations. Our first result concerns one of its building blocks: Mori fiber spaces under deformations.

\begin{proposition}[Extension of Mori fiber spaces]\label{extFano}
Let $f:X\to S$ be a flat proper morphism from a generalized pair $(X,B+ { \boldsymbol{\beta}})$  onto a smooth, connected, and relatively compact curve $S$. Fix a point $0\in S$ and assume that the boundary divisor $B$ does not contain the fiber $X_0$, and that $(X_0,B_0 + { \boldsymbol{\beta}}_0)$ is a generalized klt pair, where $(K_X+X_0+B+\bbeta _X)|_{X_0}=K_{X_0}+B_0+\bbeta _0$. 

If there is a $(K_{X_0}+B_0+ \bbeta_0)$-Mori fiber space $\phi_0:X_0 \to Z_0$ (cf. Definition \ref{MFSKahler}), then the family admits a relative Mori fiber space $\phi: X\to Z /U$ over a neighborhood $U \subset S$ of $0$.
\end{proposition}

\begin{proof}
Since $\phi_0 : X_0\to Z_0$ is a $(K_{X_0} +B_0+ \bbeta_0)$-Mori fiber space, $-(K_{X_0}+ B_0 + { {\bbeta}}_0)$ is $\phi_0$-K\"ahler. 
Therefore, by the Kawamata--Viehweg vanishing theorem (Proposition \ref{KVVanishing}), $R^1(\phi_0)_*(\mathcal{O}_{X_0}) = 0$. 
Then, Theorem \ref{ExtendCont} implies that the contraction morphism extends to $\phi:X\to Z/U$ for some $U \subset S$. 
We claim that,  for $t$ near 0, $\phi_t:X_t \to Z_t$ are also  contractions of fiber type. Indeed, by Chevalley's upper semi-continuity theorem,  $\dim Z_t \le \dim Z_0$. 
Moreover, since $f:X\to S$ is flat, $\dim X_0 = \dim X_t$. Therefore, $$ \dim X_t  - \dim Z_t \ge \dim X_0 - \dim Z_0 >0.$$
By Proposition \ref{DefAmple}, $-(K_{X_t}+B_t +  { {\bbeta}}_t )$ is $\phi_t$-K\"ahler for $t$ near $0$. 
\end{proof}

\begin{remark}
Note that in Proposition \ref{extFano}, the existence of the generalized pair $(X,B+\bbeta )$ is crucial; without it, the deformation stability of Fano type structures remains an open problem (cf. \cite[Question 1.1]{CLZ25}).
\end{remark}

The next result concerns two more building blocks of the MMP: divisorial and flipping contractions under deformations.

\begin{proposition}[Extension of divisorial and flipping contractions]\label{ContInMMP}
Let $f:X\to S$ be a flat proper morphism from a generalized pair $(X,B+ { \boldsymbol{\beta}})$ onto a smooth, connected, and relatively compact curve $S$. Fix a point $0\in S$ and assume that the boundary divisor $B$ does not contain the fiber $X_0$, and that $(X_0,B_0 + { \boldsymbol{\beta}}_0)$ is a generalized klt pair, where $(K_X+X_0+B+\bbeta _X)|_{X_0}=K_{X_0}+B_0+\bbeta _0$. 

If there exists a divisorial/flipping contraction $\phi_0:X_0 \to Z_0$ (cf. Definitions \ref{DefinitionDivisorial} and \ref{DefinitionFlip}), then the divisorial/flipping contraction can be extended to a bimeromorphic contraction $\phi:X\to Z/U$ over a neighborhood $U \subset S$ of $0$. If $\phi _0$ is a flipping contraction, then $\phi$ is small.
\end{proposition}

\begin{proof}
Since $\phi_0 :X_0 \to Z_0$ is a divisorial or flipping contraction, 
$- (K_{X_0} + B_0 +  {{\bbeta}}_0)$ is $\phi_0$-K\"ahler. It follows from Kawamata--Viehweg vanishing (Proposition \ref{KVVanishing}) that 
$R^1 (\phi_0)_*(\mathcal{O}_{X_0}) = 0$. By Theorem \ref{ExtendCont}, after shrinking $S$ to some open neighborhood $U\subset S$, the flipping or divisorial contraction extends to a contraction 
$\phi : X \to Z$ over $U$.

Next, we show that if the central fiber admits a flipping contraction $\phi_0:X_0 \to Z_0$, then it extends to small bimeromorphic morphisms for the nearby fibers. Set
\[
\mathrm{Ex}(\phi_t) := \{ x \in X_t \mid \dim \phi_t^{-1}(\phi_t(x)) > 0 \}.
\]
Since $\phi_0$ is an isomorphism in codimension $1$, 
$\mathrm{codim}_{X_0}(\mathrm{Ex}(\phi_0)) \ge 2$. Consider the restriction of the morphism 
$\mathrm{Ex}(\phi) \to S$. By Chevalley’s upper semi-continuity theorem,
\[
\dim \mathrm{Ex}(\phi)|_{X_t} \le \dim \mathrm{Ex}(\phi)|_{X_0}.
\]
Thus,
\[
\dim \mathrm{Ex}(\phi_t) \le \dim \mathrm{Ex}(\phi_0) \le \dim X_0 - 2 = \dim X_t - 2.
\]
Therefore on the fibers near $0$, $\phi_t : X_t \to Z_t$ are small morphisms and $-(K_{X_t}+B_t+\boldsymbol{\beta} _t)$ is relatively K\"ahler. 

By similar arguments, we can show that if the central fiber is a divisorial contraction, then the nearby fibers are also bimeromorphic contractions (which need not be divisorial contractions, as the example below indicates).
\end{proof}
\begin{example}
In \cite[Remark 4.5]{FH11}, de Fernex--Hacon construct a family 
$ f: X \to \mathbb{A}^1 $ of quadric surfaces degenerating to the Hirzebruch surface 
$ X_0 = \mathbb{F}_2 $. A line in one of the two rulings on the general fiber sweeps out a divisor 
$ R $ on $X$, which restricts on the central fiber as  
\[
R|_{X_0} = E + F,
\]
where $ F $ is a fiber of $ \mathbb{F}_2 \to \mathbb{P}^1 $ and $ E $ is the $(-2)$-curve on $ \mathbb{F}_2 $. They then construct a $(K_X + \varepsilon R)$-negative contraction $ X \to Z $, which restricts to a divisorial contraction $ X_0 \to Z_0 $ but is an isomorphism 
$ X_t \to Z_t $ for $ t \neq 0 $. This example shows that, in some situations, the deformation of a divisorial contraction need not remain a divisorial contraction on nearby fibers. However, under some reasonable assumptions, \cite[Theorem 5.7]{FH11} shows that a divisorial contraction deforms to a family of divisorial contractions. 
\end{example}

The divisors contracted by a negative contraction are contained in the support of the negative part of the Boucksom--Zariski decomposition.
\begin{lemma}[{\cite[Theorem A.11]{DHY23}}]\label{ZariskiCont}
Let $\phi: X \dashrightarrow X^{\prime}$ be a bimeromorphic contraction of normal compact K\"ahler varieties. Let $\left(X, B+\boldsymbol{\beta}\right)$ and $\left(X^{\prime}, B^{\prime}+\boldsymbol{\beta}'\right)$ be generalized dlt pairs such that $K_X+B+\boldsymbol{\beta}$ is pseudo-effective and $B^{\prime}+\boldsymbol{\beta}'=\phi_*\left(B+\boldsymbol{\beta}\right)$.

If $\phi$ is $(K_X+B+\boldsymbol{\beta})$-negative (which holds for steps of the MMP), then the divisors contracted by $\phi$ are contained in the support of the negative part of the Boucksom--Zariski decomposition $N\left(K_X+B+\boldsymbol{\beta}\right)$, that is $\mathrm{Ex}(\phi)\subset \operatorname{Supp} N(K_X+B+ \boldsymbol{\beta})$, and moreover, $$\phi_*(N(K_X+B+ \bbeta)) = N(K_{X'}+ B' + \bbeta ').$$
\end{lemma}
Before proceeding, we recall that if $D=\sum_i d_iD_i$ and $D'=\sum_i d'_iD_i$, for distinct prime divisors $D_i$, then 
$$D \wedge D^{\prime}:=\sum_i \min \left\{d_i, d_i^{\prime}\right\} D_i.$$
\begin{lemma}\label{singular}
In the same setting as Lemma~\ref{ZariskiCont}, if $(X,B+\bbeta)$ is a generalized pair with canonical singularities and 
$B \wedge N(K_X+B+\bbeta)=0$, then $(X',B'+\bbeta')$ also admits canonical singularities and $B' \wedge N(K_{X'}+B'+\bbeta')=0$. 
\end{lemma}

\begin{proof}
Take a common resolution $p:Y \to X$ and $q:Y \to X'$
of the birational map $\phi: X \dashrightarrow X'$. 
Since the contraction $\phi$ is $(K_X+B+\bbeta)$-negative, \[
p^*(K_X+B+\bbeta) - q^*(K_{X'}+B'+\bbeta') = E \ge 0
\]
for some effective $q$-exceptional divisor $E$. Since $\phi$ only contracts divisors and does not extract any divisor, every $p$-exceptional divisor is also $q$-exceptional. 

Therefore, by $E \ge 0$, for any  $p$-exceptional divisor $E'$ we obtain  
\[
a(X',B'+\bbeta',E') \;\ge\; a(X,B+\bbeta,E') \;\ge\; 0.
\]
Moreover, since $B \wedge N(K_X+B+\bbeta)=0$, 
Lemma~\ref{ZariskiCont} implies that $\phi$ does not contract any component of $B$. 

For the remaining $q$-exceptional divisors $E''$, we have that $p_*(E'') \neq 0$ is a divisor on $X$. Since $p_*(E'')$ is not contained in $\mathrm{supp}(B)$, it follows that  
\[
a(X',B'+\bbeta',E'') > a(X,B+\bbeta,E'') = 0.
\]
Thus $(X',B'+\bbeta')$ is generalized canonical.  Finally, since  
\[
\phi_* N(K_X+B+\bbeta) = N(K_{X'}+B'+\bbeta'),
\]
and $B' = \phi_*B$, it follows immediately that $B' \wedge N(K_{X'}+B'+\bbeta')=0$. 
\end{proof}

As for the minimal model program in the complex analytic setting, the following Property $(P)$ is useful.
\begin{definition}[{\cite{FujinoBCHM}}]
Let $\pi: X \rightarrow Y$ be a projective morphism of complex analytic spaces and $W$ a compact subset of $Y$ with the properties:\\
\emph{(P1)} $X$ is a normal complex variety,\\
\emph{(P2)} $Y$ is a Stein space,\\
\emph{(P3)} $W$ is a Stein compact subset of $Y$, and\\
\emph{(P4)} $W \cap Z$ has only finitely many connected components for any analytic subset $Z$ which is defined over an open neighborhood of $W$.\\
We say that $\pi : X \to Y$ and $W$ satisfy \emph{Property $(P)$} if the four conditions above hold.
\end{definition}

\begin{proposition}[{Extension of MMP to nearby fibers, \cite[Theorem 2]{DefGeneralType}}]\label{ExtendMMP}
Let $g: X \rightarrow S$ be a flat proper morphism with connected fibers from a generalized pair $(X,B+ \bbeta)$ to a smooth, connected, and relatively compact curve. Fix a point $0\in S$ and assume that the support of the boundary divisor $B$ does not contain the fiber $X_0$. Suppose that $(X_0,B_0 + { \boldsymbol{\beta}}_0)$ is projective, with canonical singularities, and $\lfloor B_0 \rfloor = 0$, where $$(K_X+X_0+B+\bbeta _X)|_{X_0}=K_{X_0}+B_0+\bbeta _0,$$ and the negative part of Boucksom--Zariski decomposition satisfies the relation $$N(K_{X_0}+ B_0 + \bbeta_0) \wedge B_0 = 0.$$ Then every sequence of transcendental MMP-steps \[(X_0, B_0  +\bbeta_0) \dashrightarrow (X_0^{(1)}, B_0^{(1) }+  \bbeta_0^{(1)}) \dashrightarrow (X_0^{(2)},B_0^{(2)}+ \bbeta_0^{(2)}) \dashrightarrow \cdots\]  extends to a sequence of $(K_X+B+\bbeta _X)$-negative proper bimeromorphic maps 
$$ (X, B+ \bbeta)/U \dashrightarrow (X^{(1)},B^{(1)}+ \bbeta^{(1)})/U \dashrightarrow (X^{(2)},B^{(2)}+ \bbeta^{(2)}) /U\dashrightarrow \cdots,$$over some open neighborhood $U\subset S$ of $0$.
\end{proposition}
\begin{proof}
The proof is based on \cite[Theorem 2]{DefGeneralType} and \cite[Theorem 5.12]{DH24}. Note that by Proposition \ref{DefCanSing}, after shrinking $S$, the generalized pair $(X,B+\bbeta)$ admits canonical singularities. Assume that $\phi_0:X_0 \to Z_0$ is a $(K_{X_0}+ B_0 + \bbeta_0)$-negative extremal contraction. Then, by Kawamata--Viehweg vanishing (Proposition \ref{KVVanishing}), $R^1 (\phi_0)_* \mathcal{O}_{X_0} = 0$. By Theorem \ref{ExtendCont}, $\phi_0$ extends to a contraction $\phi: X \to Z$. Note that by \cite{DH24MMP}, there exists a klt pair $(X_0,\Delta _0)$ such that $K_{X_0}+\Delta  _0\equiv K_{X_0}+B_0+\bbeta _0$ and an ample divisor $H_0$ such that $K_{X_0}+\Delta  _0+H_0$ is semi-ample and defines the contraction $X_0\to Z_0$. It follows that $$K_{Z_0}+\Delta  _{Z_0}+H_{Z_0}:=\phi _{0,*}(K_{X_0}+\Delta  _0+H_0)$$ is klt and hence $Z_0$ has rational singularities. But then $Z$ also has rational singularities on a neighborhood of $Z_0$ by \cite{Elkikj}. 

Suppose that $\phi _0$ is of flipping type. 
We will construct  $X^+\to Z$ locally over $Z$. Pick a point $z\in Z_0$ and fix a neighborhood $z\in W\subset Z$ such that $W$ is relatively compact, Stein and satisfies Property (P). 
Let $\nu :X'_W\to X_W$ be a resolution, projective over $W$ such that $X'_z:=X'_W\times _W\{z\}$ is a divisor with simple normal crossings and $\bbeta _{X'_W}$ is nef over $W$. Arguing as in the proof of \cite[Theorem 5.12]{DH24} we may assume that there is a divisor $D'_W\equiv_W \bbeta _{X'_W}$ and hence $D'_W$ is nef and big over $W$. We may then assume that $(X'_W,B'_W+D'_W)$ is sub-klt and hence $(X_W, B_W+D_W\equiv \nu _*( B'_W+D'_W))$ is klt and in fact canonical.
By \cite[Theorem C]{FujinoBCHM},  $({X_W},B_W+D_W)$ has a log canonical model over $W$, say $X^+_W$. Note that these log canonical models are defined locally and uniquely determined by the class of the $\R$-divisor $$K_{X_W}+B_W+D_W\equiv _W K_{X_W}+B_W+\bbeta _{X_W}.$$ It follows that the $X^+_W$ glue together to give $\phi ^+:X^+\to Z$.
    
Finally, we observe that $\phi ^+_0:X^+_0\to Z_0$ is the flip of $\phi _0:X_0\to Z_0$. This can be checked locally over $Z_0$. Let $W_0=Z_0\cap W$, then \[K_{X_{W_0}}+B_{X_{W_0}}+\bbeta |_{X_{W_0}}\equiv K_{X_{W_0}}+B_{X_{W_0}}+D_{X_{W_0}}\equiv  (K_{X_{W}}+B_{X_{W}}+X_{W}+D_{X_W})|_{X_{W_0}}.\] 
Note that since $({X_{W_0}},B_{X_{W_0}}+D_{X_{W_0}})$ is canonical, then $X_{W_0}\dasharrow X^+_{W_0}$ is a small birational map since otherwise there is a divisor $F$ on $X_{W_0}^+$ that is exceptional over $X_{W_0}$. But then, letting $$K_{X_{W_0}^+}+B_{X^+_{W_0}}+D_{X^+_{W_0}}=(K_{X_{W}^+}+B_{X^+_{W}}+D_{X^+_{W}})|_{X^+_{W_0}}$$ we have the following inequalities between discrepancies \[0\ge a(X_{W_0}^+,B_{X^+_{W_0}}+D_{X^+_{W_0}};F)>a(X_{W_0},B_{X_{W_0}}+D_{X_{W_0}};F)\ge 0\] which is impossible. 
Note that the first inequality is an equality unless $F$ is contained in the support of $B_{X^+_{W_0}}+D_{X^+_{W_0}}$, the second strict inequality follows by the monotonicity lemma as in \cite[Lemma 3.38]{KM98}, and the last inequality is immediate as $(X_{W_0},B_{X_{W_0}}+D_{X_{W_0}})$ is canonical. Thus $X_{W_0}\dasharrow X_{W_0}^+$ is a small birational map and $$K_{X_{W_0}^+}+B_{X^+_{W_0}}+D_{X^+_{W_0}}=(K_{X_{W}^+}+B_{X^+_{W}}+X^+_{W_0}+D_{X^+_{W}})|_{X^+_{W_0}}$$ is ample over $W_0$ so that $X_{W_0}\dasharrow X_{W_0}^+$ is the $(K_{X_{W_0}}+B_{X_{W_0}}+D_{X_{W_0}})$-flip.
    
One easily checks that $(X_0^+, B_0^+ + \bbeta _{X^+_0})$ is projective with  canonical singularities, and the negative part of the divisorial Zariski decomposition satisfies the relation $$N (K_{X_0^+} + B_0^+ + \bbeta _{X^+_0}) \wedge B_0^+  = 0.$$ 

Suppose now that $\phi _0$ is a divisorial contraction and so $K_{Z_0}+B_{Z_0}+\bbeta _{Z_0}$ is canonical and $K_{Z_0}+B_{Z_0}$ is $\R$-Cartier. As observed above $Z_0$ has rational singularities and hence we may assume that $Z$ has rational singularities (after shrinking in a neighborhood of $Z_0$). Note that it also follows that $(Z_0,B_{Z_0})$ has canonical singularities and hence $(Z,B_Z)$ has canonical singularities \cite[Proposition 3.5]{FH11}.
Arguing as above we may construct $X\dasharrow Z^c$ the log canonical model for $(X,B+\bbeta _X)$ over $Z$.

It follows that $$K_{Z_0^c}+B_{Z^c_0}+\bbeta _{Z^c_0}=(K_{Z^c}+B_{Z^c}+Z^c_0+\bbeta _{Z^c})|_{Z^c_0}$$ is K\"ahler over $Z$ and hence $Z^c_0\to Z_0$ is an isomorphism. Therefore, $Z^c\to Z$ is also an isomorphism in a neighborhood of $Z_0$, and so $(Z,B_Z+\bbeta)$ is a generalized canonical pair and $-(K_X+B+\bbeta _X)$ is K\"ahler over $Z$. 

If we let $X^+:=Z$, then one easily checks that $(X_0^+, B_0^+ + \bbeta _{X^+_0})$ is projective with 
canonical singularities, and the negative part of divisorial Zariski
decomposition satisfies the relation $N (K_{X_0^+} + B_0^+ + \bbeta _{X^+_0}) \wedge B_0^+  = 0$.  

Replacing $(X,B+\bbeta)$ by $(X^+,B^++\bbeta^+)$ and proceeding by induction, we conclude the proof. 
\end{proof}

\begin{remark}\label{Remark:ExtMMP}
Thanks to the contraction theorem for K\"ahler threefolds (cf. \cite[Theorem 1.2]{DH24}), the same argument shows that Proposition~\ref{ExtendMMP} also holds for families of K\"ahler threefolds, without assuming that the central fiber is projective.
\end{remark}


For a projective generalized klt pair $(X, B + \bbeta)$, \cite{DH24MMP} proved the following transcendental version of \cite{BirkarCHM}, which will be repeatedly used in what follows. 
\begin{theorem}[{\cite[Theorem 1.2]{DH24MMP}}]\label{Theorem: TransBCHM}
Let $(X, B+\boldsymbol{\beta})$ be a compact generalized klt pair, and $f:X\to S$ a morphism of projective varieties such that $B+\boldsymbol{\beta}_X$ is big over $S$. 
Then the following hold:\\
(1) if $K_X+B+\boldsymbol{\beta}_X$ is pseudo-effective over $S$, then $(X, B+\boldsymbol{\beta})$ has a good minimal model over $S$, and\\
(2) if $K_X+B+\boldsymbol{\beta}_X$ is not pseudo-effective over $S$, then $(X, B+\boldsymbol{\beta})$ has a Mori fiber space over $S$.
\end{theorem}
\begin{proof}
The above result holds when $\dim S=0$ and $X$ is $\Q$-factorial by \cite[Theorem 1.2]{DH24MMP}. Assume now that $X$ is $\Q$-factorial and $\dim S>0$. Let $H$ be an ample divisor on $S$ and $A\in |(2\dim X+1)H|$ a general member. Replacing $B$ by $B+f^*A$, one sees that every $(K_X+B+\bbeta _X)$-negative extremal ray is a $(K_X+B+\bbeta _X)$-negative extremal ray over $S$ (cf. \cite[Section 3]{HP24}). It follows that $K_X+B+\bbeta _X$ is nef if and only if it is nef over $S$ and every step of the $(K_X+B+\bbeta _X)$-MMP is a step of the $(K_X+B+\bbeta _X)$-MMP over $S$.
Thus the claim holds in this case.
    
 If $X$ is not $\Q$-factorial, then the $\Q$-factorialization $\nu:X'\to X$ is a small birational morphism such that $X'$ is $\Q$-factorial.  (The existence of the $\Q$-factorialization $\nu$ follows by standard arguments using the $\Q$-factorial case and following the proof of \cite[Theorem 1.5]{HP24}).  By {\cite[Theorem 1.2]{DH24MMP}}, (1) and (2) hold for $K_{X'}+B'+\bbeta _{X'}:=\nu ^*(K_X+B+\bbeta _X)$ and hence (1) and (2) also hold for $K_X+B+\bbeta _X$.
\end{proof}

\subsection{Stability of nefness and bigness of adjoint classes}
Das--Hacon (cf. \cite[Conjecture 1.6]{DH24} and also \cite{TosatiKAWA}) made the following transcendental base point free conjecture. 
\begin{conjecture}[{Transcendental base point free conjecture}]\label{Transbpf}
Let $(X, B+\bbeta)$ be a generalized klt K\"ahler pair, where $X$ is a normal compact K\"ahler variety, and $\alpha \in H_{\mathrm{BC}}^{1,1}(X,\mathbb{R})$ a nef class on $X$. If $\alpha-\left(K_X+B+ \bbeta_X\right)$ is nef and big, then there exists a proper surjective morphism with connected fibers $f: X \rightarrow Y$ to a normal compact K\"ahler variety $Y$ with rational singularities and a K\"ahler class $\omega_Y \in H_{\mathrm{BC}}^{1,1}(Y,\mathbb{R})$ such that $\alpha=f^* \omega_Y$. 
\end{conjecture}

Thanks to \cite{DH24MMP}, the transcendental base point free conjecture holds when $X$ is projective.

\begin{theorem}[{\cite[Corollary~3.7]{DH24MMP}}]\label{DHTransbpf}
Let $(X, B+ \bbeta)$ be a projective generalized klt pair and $\alpha \in H_{\mathrm{BC}}^{1,1}(X,\mathbb{R})$ a nef class. If $\alpha-\left(K_X+B+ \bbeta_X\right)$ is nef and big, then there is a proper surjective morphism $f: X \rightarrow Z$ with connected fibers to a normal projective variety $Z$ with rational singularities such that $\alpha=f^* \omega_Z$ for some K\"ahler class $\omega_Z$ on $Z$.
\end{theorem}

Theorem \ref{DHTransbpf} allows us to prove the deformation stability of bigness-and-nefness of adjoint classes in K\"ahler families with projective central fiber (Theorem \ref{Extbigandnef}). Before that, let us present a technical proposition.

\begin{lemma}\label{c-very-general}
Let $f\colon U\to V$ be a proper surjective Moishezon morphism of analytic varieties and $(U,B_U+\bbeta)$ a generalized klt pair such that $(K_U+B_U+\bbeta_U)|_{f^{-1}(v)}\equiv 0$ for some $v\in V$. After replacing $V$ by an open neighborhood of $v$, there is a subset $W\subset V$, the complement of countably many nowhere dense closed analytic subsets of $V$, such that for every $w\in W$, we have 
\[
(K_U+B_U+\bbeta_U)|_{U_w}\equiv 0 .
\]
\end{lemma}

\begin{proof}
The statement is local at $v$, so we may assume that $V$ is Stein, relatively compact.

By Hironaka's Chow lemma, we obtain a complex manifold $X$ and a projective bimeromorphic morphism $\mu\colon X\to U$ such that $g:=f\circ\mu$ is projective. Put $\alpha:=\mu^*(K_U+B_U+\bbeta_U)$, represented by a smooth closed form. For every $w$ we have $\mu^{-1}(U_w)=X_w$, and $\mu$ maps $X_w$ onto $U_w$. By the projection formula $\alpha\cdot C=(K_U+B_U+\bbeta_U)\cdot\mu_*C$ for curves $C\subset X_w$, so $\alpha|_{X_v}\equiv0$. Conversely, if $\alpha|_{X_w}\equiv0$, then $(K_U+B_U+\bbeta_U)|_{U_w}\equiv0$, since every curve $D\subset U_w$ is the image of a curve $C\subset X_w$ (cut a component of $\mu^{-1}(D)$ dominating $D$ by general hyperplanes), and then $\mu_*C=dD$ with $d\ge1$.

Since $X\to V$ is projective, there is an embedding $X\subset\mathbb P^N\times V$ over $V$, and let $\mathcal C_d(X/V)$ be the Barlet relative cycle space of effective $1$-cycles of degree $d$, with projection $\pi: \mathcal{C}_d(X/V) \to V$. By \cite[Chapter~IV]{Barlet}, $\pi\colon\mathcal C_d(X/V)\to V$ is proper. Hence $\bigcup_d\mathcal C_d(X/V)$ has countably many irreducible components $Z_i$, and each $\pi(Z_i)$ is closed analytic. Since $\alpha$ is represented by a closed form, by \cite[Corollary~4.2.45]{Barlet}, $c\mapsto \alpha \cdot c$ is equal to a constant $a_i$ for $c\in Z_i$. If $v\in\pi(Z_i)$, then $Z_i$ contains a cycle $\sum m_jC_j$ with $C_j\subset X_v$, so $a_i=0$. Let $W_1$ be the complement of the union of the sets $\pi(Z_i)$ with $a_i\neq0$; none of them contains $v$, hence proper closed analytic so each is nowhere dense. For $w\in W_1$, every irreducible curve $C\subset X_w$ is a point of some $Z_i$ with $w\in\pi(Z_i)$, so $\alpha\cdot C=a_i=0$. Therefore $(K_U+B_U+\bbeta_U)|_{U_w}\equiv0$ for all $w\in W_1$.
\end{proof}

\begin{proposition}\label{l-bir}
Let $f:U\to V$ be a proper Moishezon contraction morphism of normal analytic varieties and $(U,B_U+\bbeta)$ a generalized klt pair such that $B_U+\bbeta_U$ is modified big over $V$ and $(K_U+B_U+\bbeta_U)|_{f^{-1}(v)}\equiv 0$ for some $v\in V$. Then replacing $V$ by a smaller open neighborhood $V'\subset V$ of $v$ and $U$ by $f^{-1}(V')$ we have that $K_U+B_U+\bbeta _U\equiv 0$. Moreover, if $f$ is bimeromorphic, then
$K_V+B_V+\bbeta _V:=f_*(K_U+B_U+\bbeta _U)$ is generalized klt and $K_U+B_U+\bbeta _U=f^*(K_V+B_V+\bbeta _V)$.     
\end{proposition}

\begin{proof}
Let $\nu :U'\to U$ be a log resolution such that $\bbeta$ descends to $U'$. We may assume that $V$ is Stein and that $f':U'\to V$ satisfies property (P). We write $K_{U'}+B_{U'}+\bbeta _{U'}={\nu}^*(K_U+B_U+\bbeta _U)$. Since $B_U+\bbeta _U$ is modified big over $V$, for $0<\varepsilon \ll 1$ the divisor $B_{U'}^{\ge 0}+\bbeta _{U'}+\varepsilon {\rm Ex}(\nu)$ is big over $V$ by \cite[Lemma 2.5]{HLX26} and $(U', B_{U'}^{\ge 0}+\bbeta_{U'}+ \varepsilon {\rm Ex}(\nu))$ is generalized klt. It follows easily that \[B_{U'}^{\ge 0}+\varepsilon {\rm Ex}(\nu )+\bbeta _{U'}\equiv \Delta _{U'},\] where, for any $0< \varepsilon\ll 1$, the pair $(U',\Delta _{U'})$ is klt and $K_{U'}+\Delta _{U'}\equiv _U B_{U'}^{\leq 0}+\varepsilon {\rm Ex}(\nu )=:F$. By the MMP for projective morphisms as in \cite{FujinoBCHM} and \cite{DHP24}, $K_{U'}+\Delta _{U'}$ has a good minimal model $\phi:U'\dasharrow U''$ over $U$, as in the following commutative diagram:
\begin{center}
\begin{tikzcd}
	{U'} && {U''} \\
	& U \\
	& V.
	\arrow["\phi", dashed, from=1-1, to=1-3]
	\arrow["\nu", from=1-1, to=2-2]
	\arrow["{f'}"', from=1-1, to=3-2]
	\arrow["\mu"', from=1-3, to=2-2]
	\arrow["{f''}", from=1-3, to=3-2]
	\arrow["f"{description}, from=2-2, to=3-2]
\end{tikzcd}
\end{center}
Since $K_{U'}+\Delta  _{U'}\equiv _{U}F$, where $F$ is an effective $\mathbb R$-divisor whose support is equal to ${\rm Ex}(\nu)$, the negativity lemma gives $\phi _*F=0$, and so $\mu:U''\to U$ is a small birational morphism. Thus $(U'',\Delta _{U''})$ is a klt pair. Moreover, since $\phi_*F=0$, we have $K_{U''}+\Delta_{U''}\equiv\mu^*(K_U+B_U+\bbeta_U)$, and hence $(K_{U''}+\Delta_{U''})|_{U''_v}\equiv 0$, where $U''_v=\mu^{-1}(f^{-1}(v))$.

Since $\Delta _{U''}$ is big over $V$, we may run a $(K_{U''}+\Delta_{U''})$-MMP with scaling over $V$; it terminates with a good minimal model and, after shrinking $V$ around $v$, it consists of finitely many steps. We claim that each step is an isomorphism over a neighborhood of $v$. Suppose that the first step is given by the contraction $c:U''\to Z$ over $V$ of an extremal ray $R$ with $(K_{U''}+\Delta_{U''})\cdot R<0$. Its exceptional locus ${\rm Ex}(c)$ is covered by curves $C$ whose classes lie in $R$, so $(K_{U''}+\Delta_{U''})\cdot C<0$ for these curves. On the other hand, every curve $C\subset U''_v$ satisfies $(K_{U''}+\Delta_{U''})\cdot C=0$. Hence ${\rm Ex}(c)\cap U''_v=\emptyset$. Since $f''$ is proper, $f''({\rm Ex}(c))$ is a closed subset of $V$ that does not contain $v$, and $c$ is an isomorphism over its complement; therefore the divisorial contraction or flip given by this step is an isomorphism over a neighborhood of $v$. In particular, the fiber over $v$ and the restriction of $K_{U''}+\Delta_{U''}$ to it are unchanged, so $(K_{U''}+\Delta_{U''})|_{U''_v}\equiv 0$ still holds after this step. Repeating this argument for each of the finitely many steps, we may replace $U''$ by its good minimal model over $V$, and $(K_{U''}+\Delta_{U''})|_{U''_v}\equiv 0$ still holds. Let \[U''\xrightarrow{g} V''\xrightarrow{h} V\] be the induced morphisms, so that $K_{U''}+\Delta_{U''}\equiv g^*A$, where $A$ is ample over $V$. For any curve $C\subset U''_v$ we get $A\cdot g_*C=0$, so $g$ contracts $C$. By the rigidity lemma \cite[Lemma 4.1.13]{BS95}, we may therefore assume that $V=V''$, and hence $K_{U''}+\Delta _{U''}\equiv_V 0$. As $V$ is Stein, this also means that $K_{U''}+\Delta_{U''} \equiv 0$. Since $K_{U''}+\Delta_{U''}\equiv\mu^*(K_U+B_U+\bbeta_U)$ and $\mu$ is surjective, it follows that $K_U+B_U+\bbeta_U\equiv 0$.


Assume now that $f$ is birational. By the base point free theorem, $K_{U''}+\Delta _{U''}$ is semi-ample over $V$, and hence $K_{U''}+\Delta _{U''}=(f'')^*(K_V+\Delta _V)$, where $(V,\Delta _V)$ is klt; in particular, $V$ has rational singularities. By \cite[Lemma 3.3]{HP16} or \cite[Lemma 2.6]{DH24MMP}, $K_U+B_U+\bbeta_U =f^*\alpha _V$ for some locally exact $(1,1)$-form $\alpha _V$ on $V$. It follows that $K_V+B_V+\bbeta _V:=f_*(K_U+B_U+\bbeta _U)$ is generalized klt.
\end{proof}

Now we can prove:
\begin{theorem}\label{Extbigandnef}
Let $f:X\to S$ be a proper surjective morphism from a 
generalized K\"ahler pair $(X,B +  { \boldsymbol{\beta}})$ onto a smooth, connected, relatively compact curve $S$. Fix a point $0\in S$ and assume that the support of the boundary divisor $B$ does not contain the fiber $X_0$. 

Assume that the restriction to the central fiber $(X_0,B_0 +  { \boldsymbol{\beta}}_0)$ is a projective  generalized klt pair, such that $K_{X_0}+B_0 +  { {\bbeta}}_0$ is nef and big. Then $K_{X}  + B +  { {\bbeta}_X}$ is nef and big over some Euclidean open neighborhood $U\subset S$ of $0$. 
\end{theorem}
\begin{proof}
Since $(X_0, B_0 + \bbeta_0)$ is a generalized klt pair, by Proposition \ref{DefKLT}, after shrinking the base, $(X,B + \bbeta)$ also has generalized klt singularities. 
By the transcendental base point free Theorem \ref{DHTransbpf} and semiample fibration, we know that $K_{X_0} + B_0 +  { {\bbeta}}_0$ is a semi-ample class and there exists a bimeromorphic contraction morphism  $$g_0: X_0 \to Y_0,$$ with a generalized klt pair $(Y_0, B_{Y_0} +  { \boldsymbol{\beta}}_0)$, such that $$g_0^* (\omega_0) = K_{X_0 } +B_0+  { {\bbeta}}_0,$$ where $\omega_0$ is a K\"ahler class on $Y_0$. 

By Kawamata--Viehweg vanishing (Proposition \ref{KVVanishing}), $R^i (g_0)_* \mathcal{O}_{X_0} = 0$, for $i>0$. Then, by Theorem \ref{ExtendCont}, we know that the contraction extends to a relative contraction $$g:X\to Y/S.$$
Set $B_Y = g_* B$ and $ { {\bbeta}}_Y = g_*  { {\bbeta}_X}$. By a similar upper semi-continuity argument as in Proposition \ref{ContInMMP}, we know that the extension $g:X\to Y/S$ is a bimeromorphic contraction, which is fiberwise bimeromorphic by \cite[Lemma 2.22]{CRC25}. 
Since $g_0^* \omega_0 = K_{X_0}+ B_0 +  { {\bbeta}}_0$, $g_0$ is a $(K_{X_0}+ B_0+  { {\bbeta}}_0)$-trivial contraction morphism. Since $g:X\to Y/S$ is bimeromorphic and $(K_X+B+\bbeta_X)|_{g^{-1}(y)}=(g_0^*\omega_0)|_{g_0^{-1}(y)}\equiv 0$ for every $y\in Y_0$, Proposition \ref{l-bir} gives an open neighborhood $V\subset Y$ of $Y_0$ over which $g$ is $(K_X+B+\bbeta_X)$-trivial. Since $Y\to S$ is proper, $V$ contains the preimage of a neighborhood of $0\in S$; so after shrinking $S$, $g$ is a $(K_X+ B+ { {\bbeta}})$-trivial contraction morphism. By \cite[Lemma 2.6]{DH24MMP}, it follows that there exists a class $\omega\in H^{1,1}_{\rm{BC}}(Y,\R)$ such that  $$(K_X+B + { {\bbeta}_X} ) = g^* \omega.$$
Since $\omega_{0}=\omega|_{Y_0}$ is K\"ahler, $\omega$ is K\"ahler near $Y_0$. So $K_X+B+ \bbeta_X$ is nef and big over $U$ for some Euclidean open neighborhood $U\subset S$ of $0$. 
\end{proof}

We propose the following more general statement. 
\begin{conjecture}\label{conj: bignef}
Let $f:X\to S$ be a proper surjective morphism from 
generalized K\"ahler pair $(X,B +  { \boldsymbol{\beta}})$ onto a smooth, connected, relatively compact curve $S$. Fix a point $0\in S$ and assume that the support of the boundary divisor $B$ does not contain the fiber $X_0$. 

Assume that the restriction to the central fiber $(X_0,B_0 +  { \boldsymbol{\beta}}_0)$ is a generalized klt pair, such that $K_{X_0}+B_0 +  { {\bbeta}}_0$ is nef and big. Then $K_{X}  + B +  { {\bbeta}_X}$ is nef and big over $U$ for some open neighborhood $U\subset S$ of $0$. 
\end{conjecture}

\begin{remark}\label{Extbigandnef3fold}
Following the same proof as in Theorem  \ref{Extbigandnef}, one can show that the conjecture above follows from the transcendental base point free Conjecture \ref{Transbpf}. 

Since the transcendental base point free conjecture has been proved for K\"ahler threefolds (cf.~\cite{DH24}) and for cases where the numerical Kodaira dimension $\mathrm{nd}(K_X + \alpha) \le 3$ (cf.~\cite{JunSheng}), Conjecture \ref{conj: bignef} holds in these two cases.
\end{remark}

\section{Pseudo-effectivity of canonical divisors under deformations}\label{Section: Deformation pseff}
The goal of this section is to study the global stability of the pseudo-effectivity of canonical divisors in K\"ahler families, under the assumption that the central fiber is projective with canonical singularities. Our first result involves the deformation closedness of pseudo-effectivity of adjoint classes.


\begin{proposition}\label{Def-Close-Pseff}
Let $f: X \to S$ be a family from a generalized K\"ahler pair $(X, B+  { \boldsymbol{\beta}})$ onto a smooth, connected, and relatively compact curve $S$. Fix a point $0\in S$ and assume that the boundary divisor $B$ does not contain the central fiber $X_0$. 
Suppose that $(X_0,B_0 + { \boldsymbol{\beta}}_0)$ is projective, with canonical singularities, where $(K_X+X_0+B+\bbeta _X)|_{X_0}=K_{X_0}+B_0+\bbeta _0$, and $\lfloor B_0 \rfloor = 0$ with $N(K_{X_0 }+B_0 + \bbeta_0)\wedge B_0  = 0$. 

If there exists a sequence $\{t_i\}_{i\in \mathbb N}\subset S$ converging to $0$ such that $K_{X_{t_i}}+B_{t_i}+\bbeta _{t_i}$  is pseudo-effective for each $i$, then $K_{X_0} + B_0+  { \bbeta}_0$ is pseudo-effective.
\end{proposition}

\begin{proof}
We will argue by contradiction. Assume that $K_{X_0}+B_0+ { \bbeta}_0$ is not pseudo-effective. Theorem \ref{Theorem: TransBCHM} shows that the $(K_{X_0}+ B_0 + { \bbeta}_0)$-MMP with scaling will terminate with some Mori fiber space $X_0' \to Z_0$. By Proposition \ref{ExtendMMP}, there exists a neighborhood $U \subset S$ of $0$ such that the steps of the $(K_{X_0}+ B_0 + \bbeta_0)$-MMP extend to a sequence of $(K_X+B+\bbeta)$-negative proper meromorphic maps over $U$
$$(X,B+ \bbeta)/U \dashrightarrow (X^{(1)}, B^{(1)} + \bbeta^{(1)})/U \dashrightarrow \cdots \dashrightarrow (X', B' + \bbeta')/U \dashrightarrow \cdots.$$
By Proposition \ref{extFano}, after shrinking $U$ the Mori fiber space $X'_0 \to Z_0$ extends to a relative Mori fiber space $X'\to Z/U$ (and thus the relative $(K_{X}+ B + \bbeta)$-MMP over $U$ terminates at $X'$). 
As $X'_{t_i} \to Z_{t_i}$ ($t_i \to 0$) is a Mori fiber space, Proposition \ref{NonPSEFFMFS} yields that the divisor
$K_{X'_{t_i}} + B'_{t_i} + \bbeta'_{t_i}$
is not pseudo-effective for $i \gg 1$. 
Since $K_{X_{t_i}}+B_{t_i} + {\bbeta}_{t_i}$ is pseudo-effective for a sequence of $t_i \to 0$ by assumption, Proposition \ref{PushPSEF} implies that $K_{X'_{t_i}}+B'_{t_i} +  { \bbeta}_{t_i}'$ is pseudo-effective as well, which is a contradiction.
\end{proof}

We next prove the deformation openness of the pseudo-effectivity of canonical divisors. At present, we can establish this result only for canonical divisors, rather than for adjoint classes, due to the absence of an analogue of Proposition~\ref{pseffthreshold} for adjoint classes.

\begin{theorem}\label{Deformation invariance of pseudo-effectivity}
Let $f: X \to S$ be a family from a  K\"ahler variety $X$ onto a smooth, connected, and relatively compact curve $S$. Fix a point $0\in S$ and assume that $X_0$ is projective with canonical singularities. Then the canonical divisor $K_{X_0}$ of $X_0$ is pseudo-effective if and only if there exists a Euclidean open neighborhood $U \subset S$ of $0$ such that $K_{X}$ is relatively pseudo-effective over $U$.
\end{theorem}
\begin{proof}By Proposition \ref{DefCanSing}, we may assume that ${X_t}$ and $X$ admit canonical singularities. Suppose that $K_{X_0}$ is pseudo-effective, then we must prove that $K_X$ is relatively pseudo-effective over some open neighborhood $U\subset S$. By contradiction, suppose that $K_X $ is not $f$-pseudo-effective over any (sufficiently small) open neighborhood $U\subset S$ of $0$. For a K\"ahler class $\omega$ on $X$, define the pseudo-effective threshold  $$\tau':= \inf\{t>0 \mid K_X+ t \omega \text{ is }f\text{-pseudo-effective over some neighborhood }U\subset S\}.$$ 
Since $K_X$ is not $f$-pseudo-effective over any neighborhood $U \subset S$, $\tau '>0$ by Proposition \ref{pseffthreshold}. 

Let $0< \tau < \tau '$, and consider the generalized K\"ahler pair $(X_0,  \tau \boldsymbol{\omega}_{X_0})$ where $\boldsymbol{\omega}_{X}=\overline \omega$. By assumption $K_{X_0} $ is pseudo-effective, so that $K_{X_0} + \tau {{\omega}}_{X_0}$ is big. Since $X_0$ is projective, by Theorem \ref{Theorem: TransBCHM}, we can run the transcendental MMP on the central fiber. This MMP will terminate in $(X_0^{(n)} ,   \tau \boldsymbol{\omega}_{X_0^{(n)}})$ with nef $K_{X_0^{(n)} }+  \tau {\omega}_{X^{(n)}_0}$ 
$$(X_0,   \tau \boldsymbol{\omega}_{X_0}) \dashrightarrow (X^{(1)}_0 ,  \tau \boldsymbol{\omega}_{X_0^{(1)}}) \dashrightarrow \cdots \dashrightarrow (X^{(n)}_0,  \tau \boldsymbol{\omega}_{X_0^{(n)}}).$$ 
By Proposition \ref{ExtendMMP}, there exists a neighborhood $0\in U\subset S$, such that the $(K_{X_0}+ \tau {\omega}_{X_0})$-MMP extends to a proper $(K_X  +\tau{{\omega}}_X)$-negative map over $U$
$$(X, \tau \boldsymbol{\omega}_X) /U\dashrightarrow (X^{(1)},  \tau\boldsymbol{ \omega}_{X^{(1)}}) /U\dashrightarrow \cdots \dashrightarrow (X^{(n)}, \tau \boldsymbol{\omega}_{X^{(n)}})/U\dashrightarrow \cdots.$$
By Proposition \ref{PushPSEF}, $K_{X_0^{(n)}}  + \tau {\omega}_{X_0^{(n)}}$ is big and nef. By Theorem \ref{Extbigandnef}, we know that after further shrinking $U$, $K_{X^{(n)}}  + \tau {\omega}_{X^{(n)}}$ is nef (and hence also relatively pseudo-effective) over $U$. By Proposition \ref{PushPSEF}, $K_{X} + \tau {\omega}_X$ is relatively pseudo-effective over $U$, and hence $\tau' \le \tau < \tau'$, which is a contradiction. 

Finally, suppose that $K_X$ is relatively pseudo-effective over some open neighborhood $U\subset S$ of $0$, then it follows from Proposition \ref{Def-Close-Pseff} that $K_{X_0}$ is pseudo-effective.
\end{proof}

\begin{proposition}\label{pseflocalglobal}
Let $f:X\to S$ be a family from a  K\"ahler variety onto a smooth, connected, and relatively compact curve $S$. Assume that $X_t$ admits canonical singularities for any $t\in S$ and there exists a Euclidean open subset $U\subset S $ such that $K_X$ is relatively pseudo-effective over $U$. Then $K_X$ is relatively pseudo-effective over $S$.
\end{proposition}

\begin{proof}
Since all the fibers admit canonical singularities, the restriction of the canonical divisor satisfies $K_X|_{X_{t}} \cong K_{X_{t}}$ for any $t\in S$. 

By contradiction, suppose that $K_X$ is not relatively pseudo-effective over $S$, then there is an uncountable set $T\subset  S$ such that $K_{X_t}$ is not pseudo-effective for any $t\in T$. By Lemma \ref{BDPPCanonical}, $X_t$ is uniruled for any $t\in T$. Therefore, by Lemma \ref{Lemma: NegCurve}, there is a family of $K_{X_t}$-negative rational curves dominating $X_t$ for each $t\in T$. Denote by $\mathcal{C}_t\subset \mathcal{D}(X/S)$ the corresponding representative of this family in the relative Douady space $\mathcal{D}(X/S)$. {By \cite{Douadycount} (and \cite[Theorem 4.2.28]{Barlet} for Barlet cycle spaces)}, the relative Douady space has only countably many components, so there is a component $\mathcal{D}\subset \mathcal{D}(X/S)$ and an uncountable subset $T' \subset T$ of $S$, such 
that $$\mathcal{C}_t\subset \mathcal{D}\ \text{for any } t\in T'.$$
Since $S$ is a relatively compact curve, the identity principle implies $\overline{T'} = S$ (where $\overline{T'}$ is the Zariski closure of $T'$). Therefore, $\mathcal{D} \to S$ is proper and surjective. 

Let $
\mathcal{C} = \bigcup_{s \in T'} \mathcal{C}_s \subset \mathcal{D}$, 
and $\overline{\mathcal{C}}$ the Zariski closure of $\mathcal{C}$ in $\mathcal{D}$. Since $T'$ is Zariski dense in $S$, $\overline{\mathcal{C}} \to S$ is proper surjective. Let $\mathcal{X}\to \overline{\mathcal{C}}$ be the restriction of the universal family over $\mathcal{D} \to S$ onto $\overline{\mathcal{C}}\to S$. We then have the commutative diagram
\begin{center}
\begin{tikzcd}[ampersand replacement=\&]
	{\mathcal{X}} \& X \\
	{\overline{\mathcal{C}}} \& {S.}
	\arrow["\Phi", from=1-1, to=1-2]
	\arrow[from=1-1, to=2-1]
	\arrow["f", from=1-2, to=2-2]
	\arrow["\pi"', from=2-1, to=2-2]
\end{tikzcd}
\end{center}

Since
$\Phi \colon \mathcal{X} \to X$ is proper, $Y := \Phi(\mathcal{X}) \subset X$ is closed. For any $t \in T'$, $\mathcal{C}_t$ gives a covering family of the fiber $X_t$, and therefore
$$
\bigcup_{t \in T'} X_t \subset Y.
$$
Since $T'$ is Zariski dense in $S$,
it follows that $\Phi \colon \mathcal{X} \to X$ is surjective, i.e.,  the image $Y = X$.

Since $K_X$ is relatively pseudo-effective over $U$, there exists some general $s \in U$, such that $K_{X_s}$ is pseudo-effective. Then $K_{X_s} + \varepsilon \omega_s$ is big for any $\varepsilon > 0$. Let $[C] \in \overline{\mathcal{C}}$ be a general curve that lies in $X_s$, whose existence follows from the surjectivity of $\Phi :\mathcal{X} \to X$. Then we have
\[
0 < (K_{X_s} + \varepsilon \omega_s) \cdot C.
\]
On the other hand, for $t \in T'$
\[
K_{X_t} \cdot C' < 0, \quad \text{for } [C'] \in \mathcal{C}_t.
\]
Therefore, for $0 < \varepsilon \ll 1$, we deduce that
\[
0 < (K_{X_s} + \varepsilon \omega_s) \cdot C = (K_X + \varepsilon \omega) \cdot C = (K_X + \varepsilon \omega) \cdot C' = (K_{X_t} + \varepsilon \omega_t) \cdot C' < 0,
\]
where the second equality holds, since $[C], [C'] \in \overline{\mathcal{C}} \subset \mathcal{D}$. This is a contradiction. 
\end{proof}
Combining Theorem \ref{Deformation invariance of pseudo-effectivity} and Proposition \ref{pseflocalglobal}, we obtain the global deformation invariance of the pseudo-effectivity of canonical divisors. 
\begin{theorem}\label{GlobalPseff}
Let $f: X \to S$ be a family from a K\"ahler variety $X$ onto a smooth, connected, and relatively compact curve $S$. Assume that $X_0$ is projective, and $X_t$ have canonical singularities for all $t\in S$. Then $K_{X_0}$ is pseudo-effective if and only if $K_{X} $ is pseudo-effective over $S$.
\end{theorem}
\begin{proof}
The `if part' follows from Proposition \ref{Def-Close-Pseff}, while the converse follows from Theorem \ref{Deformation invariance of pseudo-effectivity} and Proposition \ref{pseflocalglobal}.
\end{proof}

Based on the results above, we reprove the global stability of uniruled structures.

\begin{corollary}[\cite{DefGeneralType,Kollar96}]\label{DefUniruled}
Let $f: X \to S$ be a family from a K\"ahler variety $X$ onto a smooth, connected, and relatively compact curve $S$. Assume that $X_0$ is projective, and $X_t$ have canonical singularities for all $t\in S$. Then $X_0$ is uniruled if and only if $X_t$ are uniruled for all $t \in S \setminus \{0\}$.
\end{corollary}

\begin{proof}
Suppose that $X_0$ is uniruled, then $K_{X_0}$ is not pseudo-effective by Lemma \ref{BDPPCanonical}  and hence the $K_{X_0}$-MMP ends with a Mori fiber space. By Proposition \ref{ExtendMMP}  this MMP can be extended over a neighborhood $U\subset S$ of $0$. Therefore the $X_t$ are uniruled for all $t\in U$.

We next show that $X_t$ are uniruled for all $t\in S$. The proof is essentially the same as Proposition \ref{pseflocalglobal}, for the sake of completeness let us include it here.   By contradiction, suppose that some fiber $X_{t_0}$ $(t_0\in S)$ is not uniruled.
By Lemma \ref{BDPPCanonical}, $K_{X_{t_0}}$ is then pseudo-effective. Assume that $\omega$ is a K\"ahler class on $X$ such that $\omega_{t_0} = \omega|_{X_{t_0}}$ is K\"ahler, so that $K_{X_{t_0}}+\varepsilon\omega_{t_0}$ is big for every $\varepsilon>0$.

Since $U$ is a nonempty Euclidean open subset of the curve $S$, it is uncountable,
and $X_t$ is uniruled for every $t\in U$. By Lemma \ref{Lemma: NegCurve}, for each
$t\in U$ there is a family of $K_{X_t}$-negative rational curves dominating $X_t$.
Denote by $\mathcal{C}_t\subset\mathcal{D}(X/S)$ the corresponding representative in
the relative Douady space. By \cite{Douadycount} (and \cite[Theorem 4.2.28]{Barlet}
for Barlet cycle spaces), the relative Douady space has only countably many
components, so there is a component $\mathcal{D}\subset\mathcal{D}(X/S)$ and an
uncountable subset $B\subset U$ with
$$\mathcal{C}_t\subset\mathcal{D}\quad\text{for any }t\in B.$$

Let $\mathcal{C}=\bigcup_{s\in B}\mathcal{C}_s\subset\mathcal{D}$ and 
$\overline{\mathcal{C}}$ an irreducible component of its Zariski closure in
$\mathcal{D}$. The identity principle gives
$\overline{B}^{Zar}=S$ for the Zariski closure $\overline{B}^{Zar}$ of $B$. Since $X$ is K\"ahler,  $\overline{\mathcal{C}}\to S$ is proper and surjective. Let
$\mathcal{X}\to\overline{\mathcal{C}}$ be the restriction of the universal family,
giving the commutative diagram
\begin{center}
\begin{tikzcd}[ampersand replacement=\&]
	{\mathcal{X}} \& X \\
	{\overline{\mathcal{C}}} \& {S.}
	\arrow["\Phi", from=1-1, to=1-2]
	\arrow[from=1-1, to=2-1]
	\arrow["f", from=1-2, to=2-2]
	\arrow["\pi"', from=2-1, to=2-2]
\end{tikzcd}
\end{center}

Since $\Phi\colon\mathcal{X}\to X$ is proper, $Y:=\Phi(\mathcal{X})\subset X$ is
closed. For any $t\in B$, $\mathcal{C}_t$ covers $X_t$, hence
$\bigcup_{t\in B}X_t\subset Y$. As $B$ is Zariski dense in $S$ we get $Y=X$, i.e., 
$\Phi$ is surjective. Hence the members of
$\overline{\mathcal{C}}$ lying over $t_0$ form a covering family of $X_{t_0}$. Let
$[C]\in\overline{\mathcal{C}}$ be a general member $C$ on $X_{t_0}$ such that
\[
0<(K_{X_{t_0}}+\varepsilon\omega_{t_0})\cdot C,
\qquad\text{for every }\varepsilon>0,
\]
since $K_{X_{t_0}}+\varepsilon\omega_{t_0}$ is big. Fix any $t_1\in B$ and $[C']\in\mathcal{C}_{t_1}$, then
$K_{X_{t_1}}\cdot C'<0$, so for $0<\varepsilon\ll1$ we have
$(K_{X_{t_1}}+\varepsilon\omega_{t_1})\cdot C'<0$. Since $[C],[C']$ lie in $\overline{\mathcal{C}}\subset\mathcal{D}$, for
$0<\varepsilon\ll1$,
\[
0<(K_{X_{t_0}}+\varepsilon\omega_{t_0})\cdot C
=(K_X+\varepsilon\omega)\cdot C
=(K_X+\varepsilon\omega)\cdot C'
=(K_{X_{t_1}}+\varepsilon\omega_{t_1})\cdot C'<0,
\]
which leads to a contradiction. Therefore every fiber of $f$ is uniruled.

The converse follows directly by Lemma \ref{BDPPCanonical} and Theorem  \ref{GlobalPseff}.
\end{proof}

\begin{remark}
Alternatively, we may argue by using \cite[Theorem IV.1.8]{Kollar96}, which shows that the uniruled locus 
$$\mathcal{U}:=\{t\in S \mid X_t \text{ is uniruled}\}$$
is a Zariski closed subset in $S$. Since $\mathcal{U}$ contains the non-empty Euclidean open subset $U$
and $S$ is an irreducible curve,  it follows that $\mathcal{U}=S$, hence every fiber is uniruled. Note that
\cite[Theorem IV.1.8]{Kollar96} is stated for a proper, equidimensional family of algebraic
varieties, whereas our setting concerns a family from a K\"ahler
variety. In the proof, one replaces the relative Chow variety by the relative Barlet
cycle space $\mathcal{C}(X/S)$, the properness of each irreducible component of
$\mathcal{C}(X/S)$ still holds, so the argument carries over.
\end{remark}
\begin{remark}\label{Remark: Psef3fold}
Note that \cite[Theorem 2]{DefGeneralType} already implies that if $f:X\to S$ is a flat, proper morphism of reduced, complex analytic
spaces such that the special projective fiber $X_0$ has canonical singularities and is uniruled, then $X_t$ are uniruled over some open neighborhood $0\in U\subset S$. It would be interesting to see to what extent Corollary \ref{DefUniruled} holds without the K\"ahler assumption.

Note that, if the MMP and the transcendental base point free conjecture hold on the central fiber, the same argument 
shows that Theorem \ref{GlobalPseff} and Corollary \ref{DefUniruled} also hold.

Moreover, the MMP exists and the transcendental base point free theorem holds for K\"ahler threefolds (cf. \cite{DH24}). Therefore, Theorem \ref{GlobalPseff} and Corollary \ref{DefUniruled} also hold for K\"ahler families of threefolds, without assuming that $X_0$ is projective.
\end{remark}

\section{Invariance of the volumes of adjoint classes for K\"ahler families}\label{Section: DeformationVol}
In the final section, we study the deformation behavior of volume and plurigenera in K\"ahler families.

\subsection{Invariance of volume when the central fiber has a big adjoint class}
We first prove one of the main results of the paper: the deformation invariance of the volume when the central fiber is projective and the adjoint class is big. 
\begin{theorem}\label{InvarVol}
Let $f: X \to S$ be a smooth family from a K\"ahler manifold $X$ onto a smooth, connected, and relatively compact curve $S$. For a point $0\in S$, assume that $(X,B + \bbeta)$ is a generalized K\"ahler pair such that the boundary divisor $B$ does not contain the fiber $X_0$ and let $\omega$ be a K\"ahler class on $X$ with $\omega_{X_t}:=\omega|_{X_t}$ its restriction to the fiber $X_t$. 

Assume that there exists $\delta>0$ such that $K_{X_0} + B_0 +  \boldsymbol{\beta}_{X_0}+\delta {\omega}_{X_0}$ is big and $(X_0, B_0+  { \boldsymbol{\beta}}_{X_0}+ \delta\boldsymbol{\omega}_{X_0})$ is projective, with canonical singularities, $\lfloor{ B_0\rfloor} = 0$ and $N(K_{X_0 }+B_0 + \bbeta_{X_0})\wedge B_0  = 0$. Then the volume  function 
$$t \longmapsto \mathrm{vol}(K_{X_t}+B_t +  \boldsymbol{\beta}_{X_t} +\delta{\omega}_{X_t})$$ 
is constant on a Euclidean open neighborhood $U \subset S$ of $0$.
\end{theorem}

\begin{proof}
By Theorem \ref{Theorem: TransBCHM}, we can run the generalized $(K_{X_0}+B_0 + { \bbeta}_{X_0} +  \delta {\omega}_{X_0})$-MMP with scaling of $\omega_{X_0}$ on the central fiber $X_0$, and it will terminate at some minimal model
$$X_0 \dashrightarrow X_0^{(1)} \dashrightarrow X_0^{(2)} \dashrightarrow \cdots \dashrightarrow X_0^{(n)} \dashrightarrow X_0'.$$
Since $\omega _{X_0}$ is K\"ahler, $N(K_{X_0} + B_0 + \bbeta_{X_0}+ \delta{\omega}_{X_0}) \le  N(K_{X_0} + B_0 + \bbeta_{X_0})$. Therefore $N(K_{X_0} + B_0 + \bbeta_{X_0}+ \delta{\omega}_{X_0})\wedge B_0 = 0$. By Proposition \ref{ExtendMMP}, after shrinking the base to some (connected) open neighborhood $U \subset S$ of $0$, we can extend the steps of the MMP on $X_0$ to a sequence of $(K_X+B+\bbeta_X+ \delta{\omega}_X)$-negative proper bimeromorphic maps over $U$, $$X /U \dashrightarrow X^{(1)}/U \dashrightarrow \cdots \dashrightarrow X'/U\dashrightarrow \cdots.$$
Since $K_{X_0'}  + B_0' +  { \bbeta}_{X_0'}+ \delta {\omega}_{X_0'}$ is big and nef, by Theorem \ref{Extbigandnef}, after shrinking $U$, $K_{X'} +B' +  { \bbeta}_{X'} + \delta {\omega}_{X'}$ is relatively nef over $U$. By Lemma \ref{BigVol}, $$\mathrm{vol}(K_{X'_t}+ B_t' +  { \bbeta}_{X_t'}+\delta {\omega}_{X_t'}) = (K_{X'_t}+ B_t' +  { \bbeta}_{X_t'} + \delta {\omega}_{X_t'})^m,$$where  $m = \dim_{\mathbb{C}}(X/S)$. 
We claim that $t \longmapsto(K_{X'_t} + B_t' +  { \bbeta}_{X_t'}+ \delta {\omega}_{X_t'})^m$ is constant for $t\in U$. Pick an arbitrary point $s\in U$. There exists a real smooth connected curve $\gamma: [0,1] \to U$ with  $\gamma(0) = 0,\gamma(1) = s$, so that the inverse images $V = (f')^{-1}(\gamma(0,1))$ of $f':X'\to U$ and  $\overline{V} = (f')^{-1}(\gamma([0,1]))$ are analytic subvarieties, with topological boundary $\partial V = X_0 '\cup X_s'$. Set $$Z = \{x\in X' \mid f' \text{ is not a submersion at } x\in X' \},$$ and $M_V : = {V} - {V} \cap Z$ with $\partial:=  \partial \overline{V} - \partial \overline{V} \cap Z$. By the constant rank theorem, the inverse image of the smooth connected curve under a submersion is a smooth manifold. Hence $M_V = (f'|_{X'- Z})^{-1}(\gamma (0,1))$
is smooth, such that $(M_V ,\partial )$ is a manifold with boundary. Since $X_0'$ and $X_s'$ are compact, we know that the volume of $\partial \overline{V}$ (and hence the volume of $\partial$) is finite. 
    
Since $f:X\to S$ is a smooth surjective morphism, $f$ is submersive at every point $x\in X$. Since $X'\dasharrow X$ is an isomorphism outside of a (complex) codimension 2 subset, it follows that
$\mathrm{codim}_{X'}Z\ge 2 $. By the singular Stokes' formula in Proposition \ref{stokes} below, $$(K_{X_s'}+ B'_s  + \bbeta_{X_s'}  + \delta{\omega}_{X_s'})^m-(K_{X_0'}+ B'_0 + \bbeta_{X_0'} + \delta{\omega}_{X_0'})^m  = \int_{M_V}d (K_{X'}+B'+\bbeta_{X'}+ \delta {\omega}_{X'})^m =  0.$$By the proof of Proposition \ref{ContInMMP} and \cite[Lemma 2.22]{CRC25}, 
$$\vol(K_{X_t}+ B_t + \bbeta_{X_t}+ \delta {\omega}_{X_t}) = \vol (K_{X_t'}+ B_t' + \bbeta_{X_t'} + \delta {\omega}_{X_t'}),\quad \text{ for }t\in U.$$ And therefore
$$t \longmapsto \mathrm{vol}(K_{X_t}+B_t +  \boldsymbol{\beta}_{X_t} +\delta{\omega}_{X_t})$$ is constant in the neighborhood $0\in U$. 
\end{proof}

\begin{proposition}[{\cite[p.10]{Stolzenberg}}]\label{stokes}
Let $M$ be a smooth, relatively compact, $n$-dimensional real manifold with finite $n$-volume, and  $\overline{M}$ its compactification. 
Assume that $\overline{M} \setminus M$ can be decomposed as $\partial \cup S$, where\\
\emph{(1)} $(M, \partial)$ is a manifold with boundary,\\
\emph{(2)} $\partial$ has finite $(n-1)$-volume,\\
\emph{(3)} $S$ is compact with $\operatorname{codim}_{\mathbb{R}}(S, M) \ge 2$.\\
Then for any smooth $(n-1)$-form $\alpha$,
\[
\int_{\partial} \alpha = \int_M d\alpha.
\]
\end{proposition}

Let us now turn to the second main result of this section. We first introduce some key lemmata. The following lemma was first studied by \cite{BFJ09} and \cite{ELMNP09} for the case when $\alpha$ is an integral big divisor on a projective manifold, and was recently generalized to the case where $\alpha$ is a big Bott--Chern cohomology class on a compact K\"ahler manifold by \cite{Witt24,DV25}. We need a weaker version:  
\begin{lemma}\label{l-diffvol}
Let $X$ be a compact normal K\"ahler variety of dimension $n$. For every nef and big class $\alpha \in H^{1,1}_{\rm{BC}}(X, \mathbb{R})$ and every $\mathbb{R}$-Cartier divisor $D$ with $\gamma:=\{D\}$ on $X$, we have
$$
\left.\frac{d}{d t}\right|_{t=0} \operatorname{vol}(\alpha+t \gamma)=n \alpha^{n-1}\cdot D  .  
$$
\end{lemma}
\begin{proof}
Note that when $X$ is a compact K\"ahler manifold, this follows immediately from \cite[Theorem 1.1]{DV25}. In the singular case, there exists a resolution $\mu: X' \to X$ such that $$\text{vol}(\alpha + t \gamma) := \text{vol}(\mu^*\alpha  + t \mu^* \gamma).$$
Since $\mu^* \alpha$ is still nef and big and $\mu^* \gamma$ is the class of the $\mathbb{R}$-divisor $\mu^* D$, by \cite[Theorem 1.1]{DV25}, we have $$\frac{d}{d t}|_{t=0} \operatorname{vol}(\alpha+t \gamma) = \frac{d}{d t}|_{t=0} \operatorname{vol}(\mu^*\alpha+t \mu^*\gamma) = n (\mu^*\alpha)^{n-1}\cdot \mu^* D = n \alpha^{n-1}\cdot D.$$
\end{proof}
\begin{lemma}\label{volder}
Let $X$ be a normal compact K\"ahler space of dimension $n$, $P\in H^{1,1}_{\mathrm{BC}}(X)$ a nef and big class, and $E\ge0$ an effective $\mathbb R$-Cartier divisor with $P^{n-1}\cdot E=0$ on $X$. Then $\operatorname{vol}(P+sE)=\operatorname{vol}(P)$ for every $s\ge0$.
\end{lemma}
\begin{proof}
Put $g(s)=\operatorname{vol}(P+sE)^{1/n}$ on $[0,\infty)$, where $P+sE$ are big for all $s\ge 0$ as $P$ is big. By the log concavity of the volume function (see Proposition \ref{r-logc}), the function $g$ satisfies the  concavity
\begin{equation}\label{eq-concave}
g((1-\lambda)s_1+\lambda s_2)\ge(1-\lambda)g(s_1)+\lambda g(s_2).    
\end{equation}
As $E\ge 0$, the function $g(s)$ is monotone increasing. Since $P$ is nef, Lemma \ref{l-diffvol} gives $$g'(0)=\operatorname{vol}(P)^{\frac1n-1}\big(P^{n-1}\cdot E\big)=0.$$
By \eqref{eq-concave} this implies $g(s) = g(0)$ for $s \ge 0$. 
\end{proof}

Let $f : X \to S$ be a proper surjective morphism from a normal complex variety $X$ to a relatively compact complex variety $S$. We say that a log pair $(X,B)$ is \emph{log smooth over $S$} if $f : X \to S$ is a smooth morphism and, for each stratum $B_I$ of $B$, the restriction $B_I \to S$ is also a smooth morphism.

We can now state the second main result of this section, a variant of Theorem \ref{InvarVol} which holds without the hypothesis $N(K_{X_0} + B_0 + \bbeta_{X_0}) \wedge B_0 = 0$.
 
\begin{theorem}\label{InvarVol2}
Let $f:X\to S$ be a smooth family from a K\"ahler manifold $X$ onto a smooth, connected, relatively compact curve $S$. Assume that $(X,B+\bbeta)$ is a generalized klt pair such that $(X,B)$ is log smooth over $S$ and $\bbeta =\overline \beta $ where $\beta$ is nef over $S$, and $X_0$ is projective. If $K_{X_0}+B_0+\bbeta _{X_0}$ is big, then the volume function
\[t\longmapsto {\rm vol}(K_{X_t}+B_t+\bbeta _{X_t})\]
is constant on some Euclidean neighborhood $U\subset S$ of $0$.
\end{theorem}
\begin{proof}
Set $n=\dim_{\mathbb C}(X/S)$ and $v_0=\vol(K_{X_0}+B_0+\bbeta_{X_0})>0$. After an \'etale base change we may assume that $f|_Z$ has connected fibers for any
stratum $Z$ of the support of $B$. After blowing up strata of $B$, we may assume that the components of $B$ are disjoint and
so $(X,B)$ is terminal. As $(X,B)$ is log smooth over $S$, there is a unique
$\R$-divisor $0\le B^\sharp\le B$ such that
\[
  B^\sharp_0=B_0-B_0\wedge N(K_{X_0}+B_0+\bbeta_{X_0}),
\]
and we set $B^\varepsilon=B^\sharp+\varepsilon(B-B^\sharp)$ for $\varepsilon\in[0,1]$. By
Proposition  \ref{VolZariski},
\begin{align*}
\vol(K_{X_0}+B^\varepsilon_0+\bbeta_{X_0})
  &= \vol\bigl(K_{X_0}+B_0+\bbeta_{X_0}-(1-\varepsilon)B_0\wedge N(K_{X_0}+B_0+\bbeta_{X_0})\bigr) \\
  &= \vol(K_{X_0}+B_0+\bbeta_{X_0}) = v_0 > 0
     \qquad\text{for any } \varepsilon\in[0,1].
\end{align*}
Since 
$N(K_{X_0}+B^\sharp_0+\bbeta_{X_0})\wedge B^\sharp_0=0$ and $K_{X_0}+B^\sharp_0+\bbeta_{X_0}$ is big, Theorem \ref{Theorem: TransBCHM} gives rise to a
$(K_{X_0}+B^\sharp_0+\bbeta_{X_0})$-minimal model $X_0\dasharrow X'_0$. By Proposition
\ref{ExtendMMP}, it extends to a $(K_X+B^\sharp+\bbeta_X)$-negative 
bimeromorphic contraction $\varphi:X\dasharrow X'/U$ over a neighborhood $0\in U\subset S$. Moreover, $\varphi$ is fiberwise bimeromorphic over $U$ by \cite[Lemma 2.22]{CRC25}. 
Let $$B' = \varphi_* B,\ (B^{\sharp})' = \varphi_* B^{\sharp},\ (B^{\varepsilon})'  = \varphi_* B^{\varepsilon}, $$
and write 
$$P'=K_{X'}+ (B^{\sharp})' +\bbeta_{X'},\ E'=B'-(B^{\sharp})'\ge0.$$ 
By Theorem \ref{Extbigandnef}, after one shrinks $U$, the class $P'$ is relatively nef over $U$ and $\vol(P'_t)=(P'_t)^n $ is
independent of $t$ by the singular Stokes' formula in Proposition \ref{stokes}. Since $\varphi_0$ is $(K_{X_0}+B^\sharp_0+\bbeta_{X_0})$-negative,
$$\vol(P'_0)=\vol(K_{X_0}+B^\sharp_0+\bbeta_{X_0})=v_0.$$
Thus, for $t
\in U$
$$
\vol(K_{X_t}+B^\sharp_t+\bbeta_{X_t})=\vol(P'_t)=\vol(P'_0)=v_0>0
$$
and in particular each $P'_t$ is nef and big. 

Since the above MMP consists of finitely many steps, each contracting a single
extremal ray, $\varphi$ is $(K_X+B^\varepsilon+\bbeta)$-negative for
$0\le\varepsilon\ll1$ as well. Applying this on the central fiber and then for any $0\le\varepsilon\ll1 $ we have
\[ \vol(P'_0)\leq
  \vol(P'_0+\varepsilon E'_0)\leq \vol(K_{X'_0}+(B^\varepsilon_0)'+\bbeta_{X'_0})
  =\vol(K_{X_0}+B^\varepsilon_0+\bbeta_{X_0})=v_0=\vol(P'_0) .
\]
Then $\varepsilon\mapsto\vol(P'_0+\varepsilon E'_0)$ is constant for $0\le \varepsilon \ll 1$ and
its derivative vanishes at $0$. By Lemma \ref{l-diffvol}, this derivative equals
$n\langle(P'_0)^{n-1}\rangle\cdot E'_0=n(P'_0)^{n-1}\cdot E'_0$, and so $(P'_0)^{n-1}\cdot E'_0=0$.  By the singular Stokes' formula in Proposition \ref{stokes}, $t\mapsto(P'_t)^{n-1}\cdot E'_t$ is constant on $U$,
and thus $(P'_t)^{n-1}\cdot E'_t=0$ for $t\in U$. So by Lemma \ref{volder}, we have $ \vol(P'_t+ \varepsilon E'_t)=\vol(P'_t)$ and then for $0\le \varepsilon \ll 1$, 
$$\vol(K_{X_t}+B^\varepsilon_t +\bbeta_{X_t})=\vol(K_{X'_t}+(B_t^{\varepsilon})'+\bbeta_{X'_t})=\vol(P'_t)=v_0,\qquad t\in U.
$$
Thus, $$\vol(K_{X_t}+B_t+ \bbeta_{X_t}) \le v_0$$
since Proposition \ref{r-logc} gives $$\varepsilon \vol(K_{X_t}+B _t+\bbeta_{X_t})^{1/n}+(1-\varepsilon)\vol(K_{X_t}+B^\sharp _t+\bbeta_{X_t})^{1/n}\leq \vol(K_{X_t}+B^\varepsilon _t+\bbeta_{X_t})^{1/n}.$$ Therefore,  $$v_0=\vol(K_{X_t}+B _t^\sharp +\bbeta_{X_t})\leq \vol(K_{X_t}+B_t+\bbeta_{X_t})\le v_0,$$ which implies that $\vol(K_{X_t}+B _t+\bbeta_{X_t})=v_0$ is constant for $t\in U$. 
\end{proof}

\begin{remark}
In \cite[Remark 4.10]{FH11}, de Fernex--Hacon showed that the deformation invariance of log plurigenera or volume does not hold when the family is not log smooth. Consider the family of smooth quadrics degenerating to the Hirzebruch surface $\mathbb{F}_2$. A line in one of the two rulings on the general fiber $X_t$ sweeps out under deformation over the base $S$ a divisor $R$ on $X$ with the following properties\\
(a) $\left.R\right|_{X_0}=E+F$, where $F$ is a fiber of the projection $\mathbb{F}_2 \rightarrow \mathbb{P}^1$, and $E$ is the tautological section on $\mathbb{F}_2$ whose intersection number is $E^2 = -2$. \\
(b) $R|_{X_t} = F_t$ for $t
 \ne 0$, where $F_t$ is one of the rulings of the quadrics.

For any $b \ge 1$, consider a divisor $L$ on $X$ such that $\left.L\right|_{X_0} \sim(b+1) E+2 b F$ and on the general fiber $L|_{X_t} = (b+1)F_0 + (b-1)F_1$ (where $F_0$ and $F_1$ are the rulings of the quadrics). Moreover, if we define $D:=\frac{1}{m}(H+2 R)$, for a general $H \in\left|-\frac{b-1+2 m}{2} K_X\right|$ with $m\gg 1$, then $$L \sim m\left(K_X+D\right).$$
The pair $(X,D)$ is klt but not log smooth since the divisor $R$ splits on the central fiber. In this case, the log plurigenera are not constant since on the general fiber $$ P_m (K_{X_t}+D_t)=h^0\left(\mathcal{O}_{X_t}(L)\right)=h^0\left(\mathcal{O}_{\mathbb{P}^1 \times \mathbb{P}^1}(b+1, b-1)\right)=b^2+2 b,$$while on the central fiber $|L|_{X_0}|=| b(E+2 F)+E|=| b(E+2 F) |+E$ and therefore $$P_m (K_{X_0} + D_0) = h^0\left(\mathcal{O}_{X_0}(L)\right)=\binom{b+3}{3}-\binom{b+1}{3}=b^2+2 b+1.$$

To compare the volumes, fix the pair $(X,D)$ (and hence $m$) and consider
$h^0(\mathcal{O}_{X_t}(qL))$ as $q\to\infty$. On the central fiber, the
fixed part of $|qL|_{X_0}| = |qb(E+2F)| + qE$, so
\[
h^0(\mathcal{O}_{X_0}(qL))
=h^0\!\left(\mathcal{O}_{\mathbb{F}_2}(qb(E+2F))\right)
=(qb+1)^2.
\]
On a general fiber,
\[
h^0(\mathcal{O}_{X_t}(qL))
=(q(b+1)+1)(q(b-1)+1).
\]
Since the fibers are surfaces and $L\sim m(K_X+D)$, it follows that
\[
\operatorname{vol}(K_{X_0}+D_0)=\frac{2b^2}{m^2}
\quad\text{and}\quad
\operatorname{vol}(K_{X_t}+D_t)=\frac{2(b^2-1)}{m^2}\quad(t\ne0).
\]
Thus both the log plurigenera and volume jump on the central fiber.
\end{remark}

\subsection{Invariance of volume for families of K\"ahler threefolds}
In this subsection, we prove the deformation invariance of the plurigenera and the volume of an adjoint class for smooth families of K\"ahler threefolds.

\begin{theorem}\label{DefVol3fold}
Let $f: X \to S$ be a smooth K\"ahler family of threefolds over a smooth, connected, and relatively compact curve $S$. Then for any integer $m\ge 1$, the $m$-genus $P_m(X_t)$ is constant for any $t\in S$. Furthermore, assume that there exists a K\"ahler class $\omega_X$ on $X$ such that $\omega_{X_t} :  =  \omega_X |_{X_t}$. Then for every $\delta\ge 0$, the function $$t \longmapsto \operatorname{vol}(K_{X_t}+\delta\,\omega_{X_t})$$is constant for any $t\in S$.
\end{theorem}
\begin{proof}
We first consider the deformation invariance of plurigenera. Using a similar extension technique (cf. Proposition \ref{ExtendMMP}), \cite[Theorem 1.2]{DHP24} proved the existence of relative MMPs for smooth families of K\"ahler threefolds. Then, the invariance of plurigenera follows from \cite[Theorem 11]{Naka85}, as the abundance conjecture is completely known for K\"ahler threefolds by \cite{CHP16, CHP23Erratum, DO24, DO24II}. 

We next consider the deformation invariance of the volume. It is easy to see that the case $\delta=0$ follows from the deformation invariance of plurigenera established above. Fix a point $t_0 \in S$, and let $
\tau := \inf\{ t \ge 0 \mid K_{X_{t_0}} + t\,\omega_{X_{t_0}} \text{ is pseudo-effective} \}$. We now assume that $\delta>0$ and divide the argument into three cases according to the relation between $\delta$ and $\tau$.

Case 1. If $\delta >  \tau$, then $K_{X_{t_0}}+ \delta{\omega}_{X_{t_0}}$ is big. Following a similar line of argument as in Theorem \ref{InvarVol}, and using \cite[Theorem 1.2]{DHY23}, we can run the generalized $(K_{X_{t_0}} + \delta {\omega}_{X_{t_0}})$-MMP with scaling of $\omega_{X_{t_0}}$ on the central fiber, and it will terminate at some minimal model $$(X_{t_0},\delta \boldsymbol{\omega}_{X_{t_0}}) \dashrightarrow (X_{t_0}^{(1)} , \delta \boldsymbol{\omega}_{X_{t_0}^{(1)}})\dashrightarrow  \cdots \dashrightarrow (X_{t_0}^{(n)}, \delta \boldsymbol{\omega}_{X_{t_0}^{(n)}}) \dashrightarrow (X_{t_0}', \delta \boldsymbol{\omega}_{X_{t_0}'}).$$

By Remark \ref{Remark:ExtMMP} of Proposition \ref{ExtendMMP}, after shrinking the base to some (connected) open neighborhood $U \subset S$ of ${t_0}$, we can extend the steps of the MMP on $X_{t_0}$ to a sequence of $(K_X+ \delta{\omega}_X)$-negative proper bimeromorphic contractions over $U$, $$(X , \delta \boldsymbol{\omega}_X)/U \dashrightarrow (X^{(1)},\delta \boldsymbol{\omega}_{X^{(1)}})/U \dashrightarrow \cdots \dashrightarrow (X', \delta \boldsymbol{\omega}_{X'})/U\dashrightarrow \cdots.$$
Since $K_{X_{t_0}'} + \delta {\omega}_{X_{t_0}'}$ is big and nef, Remark \ref{Extbigandnef3fold} implies that after shrinking the base $K_{X'} + \delta {\omega}_{X'}$ is relatively nef over $U$. Then, by the singular Stokes' formula in Proposition \ref{stokes}, the volume function $$t \longmapsto \operatorname{vol}\!\left(K_{X_t}+\delta {\omega}_{X_t}\right) = (K_{X_t'}+ \delta{\omega}_{X_t'})^3$$
is constant for $t\in U$.

Case 2. If $0 < \delta < \tau$, then $K_{X_{t_0}}+\delta \omega_{X_{t_0}}$ is not pseudo-effective. We can run the generalized $(K_{X_{t_0}} + \delta {\omega}_{X_{t_0}})$-MMP with scaling of $\omega_{X_{t_0}}$ on the central fiber, and by \cite[Theorem 1.5]{DHY23}, it will terminate at some Mori fiber space $X_{t_0}' \to Z_{t_0}'$ (with $-(K_{X_{t_0}'}+\delta \omega_{X'_{t_0}})$ relative K\"ahler over $Z_{t_0}'$). As in Case 1, after shrinking the base to some (connected) open neighborhood $U \subset S$ of ${t_0}$, we can extend the steps of the MMP on $X_{t_0}$ to a sequence of $(K_X+ \delta{\omega}_X)$-negative proper bimeromorphic contractions over $U$, $$(X , \delta \boldsymbol{\omega}_X)/U \dashrightarrow (X^{(1)},\delta \boldsymbol{\omega}_{X^{(1)}})/U \dashrightarrow \cdots \dashrightarrow (X', \delta \boldsymbol{\omega}_{X'})/U\dashrightarrow \cdots.$$Then, by Proposition~\ref{extFano}, after shrinking $U$ again, one can extend the Mori fiber space $X'_{t_0} \to Z'_{t_0}$ to a relative Mori fiber space $X' \to Z'/U$ over $U$. It follows that $K_{X_t} + \delta \omega_{X_t}$ is not pseudo-effective for $t \in U$. Therefore
$$
\vol(K_{X_t} + \delta \omega_{X_t}) = \vol(K_{X_{t_0}} + \delta \omega_{X_{t_0}}) = 0,
$$
for $t \in U$.

Case 3. Suppose that $0<\delta = \tau$. Then $K_{X_{t_0}}$ is not pseudo-effective and thus, we can run a $K_{X_{t_0}}$-MMP with scaling of $\omega_{X_{t_0}}$ on $X_{t_0}$, which terminates with a $(K_{X_{t_0}'} + \tau \omega_{X_{t_0}'})$-trivial Mori fiber space $X_{t_0}' \to Z_{t_0}'$ (with $-K_{X_{t_0}'}$ relative ample over $Z_{t_0}'$). By Proposition~\ref{extFano}, this Mori fiber space on $X_{t_0}'$ extends to a relative Mori fiber space $X' \to Z'/U$ over some neighborhood $U$ of $t_0$. Since this Mori fiber space is projective, by Proposition~\ref{l-bir}, the induced Mori fiber spaces on the nearby fibers are $(K_{X_t'}+\tau \omega_{X_t'})$-trivial (and $K_{X_t'}$-negative). By Lemma \ref{Lemma: PSEFandMFS}, it follows that the pseudo-effective threshold is locally constant and equal to $\tau$ in a neighborhood of $t_0$. Consequently, the volume function $t \longmapsto \vol(K_{X_t} + \tau \omega_t)$ is equal to $0$ in a neighborhood of $t_0$ and locally constant.

In all three cases, the volume function $t\longmapsto \vol(K_{X_t} + \delta {\omega}_{X_t})$ is locally constant around any point $t_0 \in S$. Since $S$ is connected, the volume function is constant for any $t\in S$. 
\end{proof}

From the above results, it is natural to propose:
\begin{conjecture}
Let $f:X\to S$ be a smooth family from a K\"ahler manifold $X$ onto a smooth, connected, relatively compact curve $S$. Assume that $(X,B+\bbeta)$ is a generalized klt pair such that $(X,B)$ is log smooth over $S$ and $\bbeta =\overline \beta $ where $\beta$ is nef over $S$. Then the volume function
\[t\longmapsto {\rm vol}(K_{X_t}+B_t+\bbeta _{X_t})\]
is constant over $S$.
\end{conjecture}



\bibliographystyle{amsalpha}
\bibliography{mybib.bib}

@article{Boucksom02,
  key = {Bo02},
  title={On the volume of a line bundle},
  author={Boucksom, S.},
  journal={International Journal of Mathematics},
  volume={13},
  number={10},
  pages={1043--1063},
  year={2002},
  publisher={World Scientific}
}

@misc{DHY23,
    author = {O. Das and C. D. Hacon and J.-I. Yáñez},
    title = {M{M}{P} for Generalized Pairs on {K}\"ahler 3-folds},
    journal = {Ann. Sc. Norm. Super. Pisa Cl. Sci. (5)},
    fjournal = {Annali della Scuola Normale Superiore di Pisa. Classe di Scienze. Serie V},
    howpublished = {Ann. Sc. Norm. Super. Pisa Cl. Sci. (5)},
    year = {2026},
    note = {Published online; forthcoming in print},
}

@article{DefGeneralType,
  key = {Kl21},
  title={Deformations of varieties of general type},
  author={Koll{\'a}r, J.},
  journal={Milan Journal of Mathematics},
  volume={89},
  pages={345--354},
  year={2021},
  publisher={Springer}
}

@article {KM92,
    AUTHOR = {Koll{\'a}r, J. and Mori, S.},
     TITLE = {Classification of three-dimensional flips},
   JOURNAL = {J. Amer. Math. Soc.},
  FJOURNAL = {Journal of the American Mathematical Society},
    VOLUME = {5},
      YEAR = {1992},
    NUMBER = {3},
     PAGES = {533--703}
}

@book {KM98,
    AUTHOR = {Koll{\'a}r, J. and Mori, S.},
     TITLE = {Birational geometry of algebraic varieties},
    SERIES = {Cambridge Tracts in Mathematics},
    VOLUME = {134},
      NOTE = {With the collaboration of C. H. Clemens and A. Corti,
              Translated from the 1998 Japanese original},
 PUBLISHER = {Cambridge University Press, Cambridge},
      YEAR = {1998},
     PAGES = {viii+254},
}

@article{BirkarCHM,
  title={Existence of minimal models for varieties of log general type},
  author={C. Birkar and Cascini, P. and Hacon, C. D. and J. M\textsuperscript{c}Kernan},
  journal={Journal of the American Mathematical Society},
  volume={23},
  number={2},
  pages={405--468},
  year={2010}
}

@article {BZ16,
    AUTHOR = {C. Birkar and D.-Q. Zhang},
     TITLE = {Effectivity of {I}itaka fibrations and pluricanonical systems
              of polarized pairs},
   JOURNAL = {Publ. Math. Inst. Hautes \'{E}tudes Sci.},
  FJOURNAL = {Publications Math\'{e}matiques. Institut de Hautes \'{E}tudes
              Scientifiques},
    VOLUME = {123},
      YEAR = {2016},
     PAGES = {283--331},
}

@misc{DH24,
      title={On the Minimal Model Program for {K}\"ahler 3-folds}, 
      author={O. Das and C. D. Hacon},
      year={2024},
      eprint={2306.11708},
      archivePrefix={arXiv},
      primaryClass={math.AG},
    note= {\href{https://arxiv.org/abs/2306.11708}{arxiv:2306.11708}},
}

@book {Kollar96,
    key = {Kl96},
    AUTHOR = {Koll{\'a}r, J.},
     TITLE = {Rational curves on algebraic varieties},
    VOLUME = {32},
 PUBLISHER = {Springer-Verlag, Berlin},
      YEAR = {1996},
     PAGES = {viii+320},
}

@article {BDPP,
    AUTHOR = {S. Boucksom and J.-P. Demailly and
              M. P\u{a}un and T. Peternell},
     TITLE = {The pseudo-effective cone of a compact {K}\"{a}hler manifold and varieties of negative {K}odaira dimension},
   JOURNAL = {J. Algebraic Geom.},
  FJOURNAL = {Journal of Algebraic Geometry},
    VOLUME = {22},
      YEAR = {2013},
    NUMBER = {2},
     PAGES = {201--248},
}

@article {FH11,
    AUTHOR = {de Fernex, T. and Hacon, C. D.},
     TITLE = {Deformations of canonical pairs and {F}ano varieties},
   JOURNAL = {J. Reine Angew. Math.},
  FJOURNAL = {Journal f\"{u}r die Reine und Angewandte Mathematik. [Crelle's
              Journal]},
    VOLUME = {651},
      YEAR = {2011},
     PAGES = {97--126},
}

@article {DH20,
    AUTHOR = {O. Das and C. D. Hacon},
     TITLE = {The log minimal model program for {K}\"ahler 3-folds},
   JOURNAL = {J. Differential Geom.},
  FJOURNAL = {Journal of Differential Geometry},
    VOLUME = {130},
      YEAR = {2025},
    NUMBER = {1},
     PAGES = {151--207},
}

@article{RT1,
  title={Deformation limit and bimeromorphic embedding of {M}oishezon manifolds},
  author={Rao, Sheng and I-Hsun Tsai},
  journal={Communications in Contemporary Mathematics},
  volume={23},
  number={08},
  pages={2050087},
  year={2021},
  publisher={World Scientific}
}

@article {RT2,
    AUTHOR = {Rao, Sheng and I-Hsun Tsai},
     TITLE = {Invariance of plurigenera and {C}how-type lemma},
   JOURNAL = {Asian J. Math.},
  FJOURNAL = {Asian Journal of Mathematics},
    VOLUME = {26},
      YEAR = {2022},
    NUMBER = {4},
     PAGES = {507--554},
}

@article {Levine,
    key = {Lv81},
    AUTHOR = {Levine, M.},
     TITLE = {Deformations of uniruled varieties},
   JOURNAL = {Duke Math. J.},
  FJOURNAL = {Duke Mathematical Journal},
    VOLUME = {48},
      YEAR = {1981},
    NUMBER = {2},
     PAGES = {467--473},

}

@article {Fujiki,
key = {Fj81},
    AUTHOR = {Fujiki, A.},
     TITLE = {Deformation of uniruled manifolds},
   JOURNAL = {Publ. Res. Inst. Math. Sci.},
  FJOURNAL = {Kyoto University. Research Institute for Mathematical
              Sciences. Publications},
    VOLUME = {17},
      YEAR = {1981},
    NUMBER = {2},
     PAGES = {687--702},
}

@article{Siu98,
key = {Si98},
  title={Invariance of plurigenera},
  author={Siu, Y.-T.},
  journal={Invent. Math.},
  year={1998},
  pages= {661-673},
}

@article {Taka06,
key = {Tk07},
    AUTHOR = {Takayama, S.},
     TITLE = {On the invariance and  semi-continuity of plurigenera of algebraic varieties},
   JOURNAL = {J. Algebraic Geom.},
  FJOURNAL = {Journal of Algebraic Geometry},
    VOLUME = {16},
      YEAR = {2007},
    NUMBER = {1},
     PAGES = {1--18},
}

@article {BFJ09,
    AUTHOR = {Boucksom, S. and Favre, C. and Jonsson, M.},
     TITLE = {Differentiability of volumes of divisors and a problem of
              {T}eissier},
   JOURNAL = {J. Algebraic Geom.},
  FJOURNAL = {Journal of Algebraic Geometry},
    VOLUME = {18},
      YEAR = {2009},
    NUMBER = {2},
     PAGES = {279--308},

}

@misc{HP24,
      title={On the Canonical Bundle Formula and Adjunction for Generalized {K}aehler Pairs}, 
      author={C. D. Hacon and M. P\u{a}un},
      year={2024},
      eprint={2404.12007},
      archivePrefix={arXiv},
      primaryClass={math.AG},
    note= {\href{https://arxiv.org/abs/2404.12007}{arxiv:2404.12007}},
}

@article {ELMNP09,
    AUTHOR = {Ein, L. and Lazarsfeld, R. and Musta\c{t}\u{a},
              M. and Nakamaye, M. and Popa, M.},
     TITLE = {Restricted volumes and base loci of linear series},
   JOURNAL = {Amer. J. Math.},
  FJOURNAL = {American Journal of Mathematics},
    VOLUME = {131},
      YEAR = {2009},
    NUMBER = {3},
     PAGES = {607--651},
}

@book {Demailly12,
    AUTHOR = {Demailly, J.-P.},
    key ={De12}, 
     TITLE = {Analytic methods in algebraic geometry},
    SERIES = {Surveys of Modern Mathematics},
    VOLUME = {1},
 PUBLISHER = {Higher Education Press, Beijing},
      YEAR = {2012},
     PAGES = {viii+231},
}

@article {HP16,
    AUTHOR = {H\"{o}ring, A. and Peternell, T.},
     TITLE = {Minimal models for {K}\"{a}hler threefolds},
   JOURNAL = {Invent. Math.},
  FJOURNAL = {Inventiones Mathematicae},
    VOLUME = {203},
      YEAR = {2016},
    NUMBER = {1},
     PAGES = {217--264},
}

@article {Naka85,
key = {Ny86},
    AUTHOR = {Nakayama, N.},
     TITLE = {Invariance of the plurigenera of algebraic varieties under
              minimal model conjectures},
   JOURNAL = {Topology},
  FJOURNAL = {Topology. An International Journal of Mathematics},
    VOLUME = {25},
      YEAR = {1986},
    NUMBER = {2},
     PAGES = {237--251},
}

@article {Kawamata99,
    key = {Kw99a},
    AUTHOR = {Kawamata, Y.},
     TITLE = {Deformations of canonical singularities},
   JOURNAL = {J. Amer. Math. Soc.},
  FJOURNAL = {Journal of the American Mathematical Society},
    VOLUME = {12},
      YEAR = {1999},
    NUMBER = {1},
     PAGES = {85--92},
}

@misc{FujinoBCHM,
key ={Fu22},
title={Minimal model program for projective morphisms between complex analytic spaces}, 
author={O. Fujino},
year={2022},
eprint={2201.11315},
note= {\href{https://arxiv.org/abs/2201.11315}{arxiv:2201.11315}},
}

@book {Naka04,
key = {Ny04},
    AUTHOR = {Nakayama, N.},
     TITLE = {Zariski-decomposition and Abundance},
    SERIES = {MSJ Memoirs},
    VOLUME = {14},
 PUBLISHER = {Mathematical Society of Japan, Tokyo},
      YEAR = {2004},
     PAGES = {xiv+277},
}

@article {CHP16,
    AUTHOR = {Campana, F. and H\"{o}ring, A. and
              Peternell, T.},
     TITLE = {Abundance for {K}\"{a}hler threefolds},
   JOURNAL = {Ann. Sci. \'{E}c. Norm. Sup\'{e}r. (4)},
  FJOURNAL = {Annales Scientifiques de l'\'{E}cole Normale Sup\'{e}rieure.
              Quatri\`eme S\'{e}rie},
    VOLUME = {49},
      YEAR = {2016},
    NUMBER = {4},
     PAGES = {971--1025},
}

@article {DO24,
    AUTHOR = {O. Das and Wenhao. Ou},
     TITLE = {On the log Abundance for compact {K}\"{a}hler threefolds},
   JOURNAL = {Manuscripta Math},
  FJOURNAL = {Manuscripta Mathematica},
    VOLUME = {173},
      YEAR = {2024},
    NUMBER = {1-2},
     PAGES = {341--404},
}

@incollection {Siu02,
key = {Si02},
    AUTHOR = {Siu, Y.-T.},
     TITLE = {Extension of twisted pluricanonical sections with
              plurisubharmonic weight and invariance of semipositively
              twisted plurigenera for manifolds not necessarily of general
              type},
 BOOKTITLE = {Complex geometry ({G}\"{o}ttingen, 2000)},
     PAGES = {223--277},
 PUBLISHER = {Springer, Berlin},
      YEAR = {2002},
}

@incollection {Kawa99,
    AUTHOR = {Kawamata, Y.},
    key = {Kw99b},
     TITLE = {On the extension problem of pluricanonical forms},
 BOOKTITLE = {Algebraic geometry: {H}irzebruch 70 ({W}arsaw, 1998)},
    SERIES = {Contemp. Math.},
    VOLUME = {241},
     PAGES = {193--207},
 PUBLISHER = {Amer. Math. Soc., Providence, RI},
      YEAR = {1999},
}

@article {DHP24,
    AUTHOR = {O. Das and C. D. Hacon and P\u{a}un, M.},
     TITLE = {On the 4-dimensional minimal model program for {K}\"{a}hler
              varieties},
   JOURNAL = {Adv. Math.},
  FJOURNAL = {Advances in Mathematics},
    VOLUME = {443},
      YEAR = {2024},
     PAGES = {Paper No. 109615, 68},

}

@article {Tsuji02,
 key = {Ts02},
    AUTHOR = {Tsuji, H.},
     TITLE = {Deformation invariance of plurigenera},
   JOURNAL = {Nagoya Math. J.},
  FJOURNAL = {Nagoya Mathematical Journal},
    VOLUME = {166},
      YEAR = {2002},
     PAGES = {117--134},
}

@article {CH20,
    AUTHOR = {Cao, Junyan and H\"{o}ring, A.},
     TITLE = {Rational curves on compact {K}\"{a}hler manifolds},
   JOURNAL = {J. Differential Geom.},
  FJOURNAL = {Journal of Differential Geometry},
    VOLUME = {114},
      YEAR = {2020},
    NUMBER = {1},
     PAGES = {1--39},
}

@misc{DH24MMP,
    author = {O. Das and C. D. Hacon},
    title = {Transcendental minimal model program for projective varieties},
    journal = {Algebr. Geom.},
    fjournal = {Algebraic Geometry},
    howpublished = {Algebr. Geom.},
    year = {2024},
    note = {To appear; \href{https://arxiv.org/abs/2412.07650}{arXiv:2412.07650}},
    eprint = {2412.07650},
    archivePrefix = {arXiv},
    primaryClass = {math.AG},
}

@article{DP04,
  title={Numerical characterization of the K{\"a}hler cone of a compact K{\"a}hler manifold},
  author={Demailly, J.-P. and P\u{a}un, M.},
  journal={Annals of mathematics},
  pages={1247--1274},
  year={2004},
  publisher={JSTOR}
}

@misc{CRC25,
      title={{C}haracterization of fiberwise bimeromorphism and specialization of bimeromorphic types {I}: locally {M}oishezon case},
      key = {CRT25},
      author={Jian Chen and Sheng Rao and  I-Hsun  Tsai},
      year={2025},
    note= {\href{https://arxiv.org/abs/2506.12670}{arxiv:2506.12670}}
}

@article {LM09,
    AUTHOR = {Lazarsfeld, R. and Mustat\u{a}, M.},
     TITLE = {Convex bodies associated to linear series},
   JOURNAL = {Ann. Sci. \'Ec. Norm. Sup\'er. (4)},
  FJOURNAL = {Annales Scientifiques de l'\'Ecole Normale Sup\'erieure.
              Quatri\`eme S\'erie},
    VOLUME = {42},
      YEAR = {2009},
    NUMBER = {5},
     PAGES = {783--835},
      ISSN = {0012-9593,1873-2151},
}

@book {Stolzenberg,
    AUTHOR = {Stolzenberg, G.},
    key = {St66},
     TITLE = {Volumes, limits, and extensions of analytic varieties},
    SERIES = {Lecture Notes in Mathematics},
    VOLUME = {No. 19},
 PUBLISHER = {Springer-Verlag, Berlin-New York},
      YEAR = {1966},
     PAGES = {ii+45},
}

@incollection {Naka86,
    AUTHOR = {Nakayama, N.},
    key = {Ny87},
     TITLE = {The lower semicontinuity of the plurigenera of complex
              varieties},
 BOOKTITLE = {Algebraic geometry, {S}endai, 1985},
    SERIES = {Adv. Stud. Pure Math.},
    VOLUME = {10},
     PAGES = {551--590},
 PUBLISHER = {North-Holland, Amsterdam},
      YEAR = {1987},
}

@misc {LRW25,
    AUTHOR = {Mu-Lin Li and Sheng Rao and Kai Wang },
     TITLE = {Deformation of nef adjoint canonical line bundles},
    YEAR = {2025},
    note = {\href{https://arxiv.org/abs/2510.23967}{arxiv:2510.23967}},
}

@incollection {Fujinokawa,
    key ={Fu11},
    AUTHOR = {O. Fujino},
     TITLE = {On {K}awamata's theorem},
 BOOKTITLE = {Classification of algebraic varieties},
    SERIES = {EMS Ser. Congr. Rep.},
     PAGES = {305--315},
 PUBLISHER = {Eur. Math. Soc., Z\"urich},
      YEAR = {2011},
}

@article {Paunnef,
    key = {Pa98},
    AUTHOR = {M. P\u{a}un},
     TITLE = {Sur l'effectivit\'e{} num\'erique des images inverses de
              fibr\'es en droites},
   JOURNAL = {Math. Ann.},
  FJOURNAL = {Mathematische Annalen},
    VOLUME = {310},
      YEAR = {1998},
    NUMBER = {3},
     PAGES = {411--421},
}

@misc{Ou25,
      title={A characterization of uniruled compact {K}\"ahler manifolds}, 
      author={Wenhao Ou},
      year={2025},
      eprint={2501.18088},
      archivePrefix={arXiv},
      primaryClass={math.AG},
      note={\href{https://arxiv.org/abs/2501.18088}{arxiv:2501.18088}}, 
}

@misc{CP25,
      title={Remarks on Relative Canonical Bundles and Algebraicity Criteria for Foliations in {K}\"ahler context}, 
      author={Junyan Cao and M. P\u{a}un},
      year={2025},
      eprint={2502.02183},
      archivePrefix={arXiv},
      primaryClass={math.CV},
      note= {\href{https://arxiv.org/abs/2502.02183}{arxiv:2502.02183}}, 
}

@misc{CLZ25,
    author = {Sung Rak Choi and Zhan Li and Chuyu Zhou},
    title = {Variation of cones of divisors in a family of varieties--{F}ano type case},
    journal = {Épijournal Géom. Algébrique},
    fjournal = {Épijournal de Géométrie Algébrique},
    howpublished = {Épijournal Géom. Algébrique},
    year = {2025},
    note = {To appear; \href{https://arxiv.org/abs/2504.04109}{arXiv:2504.04109}},
    eprint = {2504.04109},
    archivePrefix = {arXiv},
    primaryClass = {math.AG},
}

@article{BAB,
  title={Anti-pluricanonical systems on {F}ano varieties},
  author={Birkar, C.},
    key = {Bi19},
  journal={Annals of Mathematics},
  volume={190},
  number={2},
  pages={345--463},
  year={2019},
  publisher={Department of Mathematics, Princeton University Princeton, New Jersey, USA}
}

@book {BS,
    key = {BS76},
    AUTHOR = {C. B\u{a}nic\u{a} and O. St\u{a}n\u{a}sil\u{a}},
     TITLE = {Algebraic methods in the global theory of complex spaces},
      NOTE = {Translated from the Romanian},
 PUBLISHER = {Editura Academiei, Bucharest; John Wiley \& Sons, London-New
              York-Sydney},
      YEAR = {1976},
     PAGES = {296},

}

@article {Kawa85,
    AUTHOR = {Kawamata, Y.},
     TITLE = {Pluricanonical systems on minimal algebraic varieties},
    key = {Kw85},
   JOURNAL = {Invent. Math.},
  FJOURNAL = {Inventiones Mathematicae},
    VOLUME = {79},
      YEAR = {1985},
    NUMBER = {3},
     PAGES = {567--588},

}

@book {Barlet,
    AUTHOR = {Barlet, D. and Magn\'usson, J.},
     TITLE = {Complex analytic cycles. {I}---basic results on complex
              geometry and foundations for the study of cycles},
    VOLUME = {356},

 PUBLISHER = {Springer, Cham; Soci\'et\'e{} Math\'ematique de France, Paris},
      YEAR = {2019},
     PAGES = {xi+533},

}

@article{Jiao,
    author = {Junpeng Jiao},
    key = {Ji26},
    title = {Volumes of divisors in a family and variation of singularities of linear systems},
    journal = {Int. Math. Res. Not. IMRN},
    fjournal = {International Mathematics Research Notices. IMRN},
    volume = {2026},
    year = {2026},
    number = {15},
    pages = {rnag169},
}

@book {Laza04,
    AUTHOR = {Lazarsfeld, R.},
    key = {Lz04},
     TITLE = {Positivity in algebraic geometry. {I}: {L}ine bundles and linear series},
    VOLUME = {48},
 PUBLISHER = {Springer-Verlag, Berlin},
      YEAR = {2004},
     PAGES = {xviii+387},
}

@article{DO24II,
    author = {O. Das and Wenhao Ou},
    title = {On the log abundance for compact {K}\"ahler threefolds {II}},
    journal = {Proc. Lond. Math. Soc. (3)},
    fjournal = {Proceedings of the London Mathematical Society},
    volume = {132},
    year = {2026},
    number = {3},
    pages = {e70141},
}

@article {k-bmy,
 AUTHOR = {Koll{\'a}r, J.},
 TITLE = {Is there a topological {B}ogomolov-{M}iyaoka-{Y}au
inequality?},
key = {Kl08},
 JOURNAL = {Pure Appl. Math. Q.},
  FJOURNAL = {Pure and Applied Mathematics Quarterly},
  VOLUME = {4},
  YEAR = {2008},
  NUMBER = {2, part 1},
 PAGES = {203--236},
}

@article {Dem93,
    key = {De93},
    AUTHOR = {Demailly, J.-P.},
     TITLE = {A numerical criterion for very ample line bundles},
   JOURNAL = {J. Differential Geom.},
  FJOURNAL = {Journal of Differential Geometry},
    VOLUME = {37},
      YEAR = {1993},
    NUMBER = {2},
     PAGES = {323--374},

}

@article {Totaro10,
    AUTHOR = {Totaro, B.},
    key = {Tt10},
     TITLE = {The cone conjecture for {C}alabi-{Y}au pairs in dimension 2},
   JOURNAL = {Duke Math. J.},
  FJOURNAL = {Duke Mathematical Journal},
    VOLUME = {154},
      YEAR = {2010},
    NUMBER = {2},
     PAGES = {241--263},
}

@inproceedings{Boucksom04,
  title={Divisorial {Z}ariski decompositions on compact complex manifolds},
  author={Boucksom, S.},
    key = {Bo04},
  booktitle={Annales scientifiques de l'Ecole normale sup{\'e}rieure},
  volume={37},
  number={1},
  pages={45--76},
  year={2004}
}

@article{Paun07,
key = {Pa07},
  title={Siu’s invariance of plurigenera: a one-tower proof},
  author={M. P\u{a}un},
  journal={Journal of Differential Geometry},
  volume={76},
  number={3},
  pages={485--493},
  year={2007},
  publisher={Lehigh University}
}

@misc{CP23Pluri,
      title={Infinitesimal extension of pluricanonical forms}, 
      author={Junyan Cao and M. P\u{a}un},
      year={2023},
      eprint={2012.05063},
      archivePrefix={arXiv},
      primaryClass={math.AG},
      note={\href{https://arxiv.org/abs/2012.05063}{arxiv:2012.05063}}, 
}

@article {TosatiKAWA,
    AUTHOR = {Tosatti, V.},
    key = {Ts18},
     TITLE = {K{AWA} lecture notes on the {K}\"ahler-{R}icci flow},
   JOURNAL = {Ann. Fac. Sci. Toulouse Math. (6)},
  FJOURNAL = {Annales de la Facult\'e{} des Sciences de Toulouse.
              Math\'ematiques. S\'erie 6},
    VOLUME = {27},
      YEAR = {2018},
    NUMBER = {2},
     PAGES = {285--376},
}

@book {Sernesi,
    AUTHOR = {Sernesi, E.},
     TITLE = {Deformations of algebraic schemes},
    key = {Sn06},
    SERIES = {Grundlehren der mathematischen Wissenschaften},
    VOLUME = {334},
 PUBLISHER = {Springer-Verlag, Berlin},
      YEAR = {2006},
     PAGES = {xii+339},
}

@article {Elkikj,
    AUTHOR = {Elkik, R.},
    key = {Ek78},
     TITLE = {Singularit\'es rationnelles et d\'eformations},
   JOURNAL = {Invent. Math.},
  FJOURNAL = {Inventiones Mathematicae},
    VOLUME = {47},
      YEAR = {1978},
    NUMBER = {2},
     PAGES = {139--147},
}

@misc{CHP23Erratum,
    author = {Campana, F. and H{\"o}ring, A. and Peternell, T.},
    title = {Abundance on {K}\"ahler threefolds},
    journal = {Comm. Pure Appl. Math.},
    fjournal = {Communications on Pure and Applied Mathematics},
    volume = {78},
    howpublished = {Comm. Pure Appl. Math. \textbf{78}, no. 11},
    year = {2025},
    number = {11},
    note = {Appendix to H. Guenancia and M. P\u{a}un, \emph{Bogomolov--Gieseker inequality for log terminal {K}\"ahler threefolds}, pp. 2206--2244},
}

@book {MP97,
    AUTHOR = {Miyaoka, Y. and Peternell, T.},
     TITLE = {Geometry of higher-dimensional algebraic varieties},
    SERIES = {DMV Seminar},
    VOLUME = {26},
 PUBLISHER = {Birkh\"auser Verlag, Basel},
      YEAR = {1997},
     PAGES = {vi+217},
}

@PHDTHESIS{Meo96,
url = "http://www.theses.fr/1996GRE10015",
key = {Me96},
title = "Transformations intégrales pour les courants positifs fermés et théorie de l'intersection",
author = "Meo, M.",
year = "1996",
pages = "1 volume (56 P.)",
note = "Thèse de doctorat dirigée par Demailly, Jean-Pierre Mathématiques Université Joseph Fourier (Grenoble, Isère, France ; 1971-2015) 1996",
note = "1996GRE10015",
}

@article{Demailly82,
     author = {Demailly, J.-P.},
    key = {De82},
     title = {Sur les nombres de {Lelong} associ\'es \`a l'image directe d'un courant positif ferm\'e},
     journal = {Annales de l'Institut Fourier},
     pages = {37--66},
     year = {1982},
     publisher = {Institut Fourier},
     address = {Grenoble},
     volume = {32},
     number = {2},
}

@misc{JunSheng,
      title={{W}eak transcendental base-point freeness and diameter lower bounds for the {K}\"ahler-{R}icci flow}, 
    key = {Zh25},
      author={Junsheng Zhang},
      year={2025},
      eprint={2511.14735},
      archivePrefix={arXiv},
      primaryClass={math.DG},
      note ={\href{https://arxiv.org/abs/2511.14735}{arXiv:2511.14735}}, 
}

@article {HMX13,
    AUTHOR = {Hacon, C. D. and M\textsuperscript{c}Kernan, J. and Xu, Chenyang},
     TITLE = {On the birational automorphisms of varieties of general type},
   JOURNAL = {Ann. of Math. (2)},
  FJOURNAL = {Annals of Mathematics. Second Series},
    VOLUME = {177},
      YEAR = {2013},
    NUMBER = {3},
     PAGES = {1077--1111},
}

@article {LW26,
    AUTHOR = {Li, Zhan and Wang, Zhiwei},
     TITLE = {Invariance of plurigenera for generalized polarized pairs with
              abundant nef parts},
   JOURNAL = {Complex Anal. Synerg.},
  FJOURNAL = {Complex Analysis and its Synergies},
    VOLUME = {12},
      YEAR = {2026},
    NUMBER = {1},
     PAGES = {Paper No. 2},
}

@article {BP12,
    AUTHOR = {Berndtsson, B. and P\u{a}un, M.},
     TITLE = {Quantitative extensions of pluricanonical forms and closed
              positive currents},
   JOURNAL = {Nagoya Math. J.},
  FJOURNAL = {Nagoya Mathematical Journal},
    VOLUME = {205},
      YEAR = {2012},
     PAGES = {25--65},
}

@article {FS20,
    key = {FS20},
    AUTHOR = {Filipazzi, S. and Svaldi, R.},
     TITLE = {Invariance of plurigenera and boundedness for generalized
              pairs},
   JOURNAL = {Mat. Contemp.},
  FJOURNAL = {Matem\'{a}tica Contempor\^{a}nea},
    VOLUME = {47},
      YEAR = {2020},
     PAGES = {114--150},
}

@article{Douadycount,
  title={Countability of the Douady Space of a Complex Space},
  key = {Fj79},
  author={A. Fujiki},
  journal={Japanese journal of mathematics. New series},
  volume={5},
  number={2},
  pages={431-447},
  year={1979}
}

@article {DV25,
    AUTHOR = {Vu, Duc-Viet},
     TITLE = {Derivative of volumes of big cohomology classes},
   JOURNAL = {Indiana Univ. Math. J.},
  FJOURNAL = {Indiana University Mathematics Journal},
    VOLUME = {74},
      YEAR = {2025},
    NUMBER = {6},
     PAGES = {1531--1560},
}

@article {Witt24,
    AUTHOR = {Nystr\"om, David Witt},
     TITLE = {Deformations of {K}\"ahler manifolds to normal bundles and
              restricted volumes of big classes},
   JOURNAL = {J. Differential Geom.},
  FJOURNAL = {Journal of Differential Geometry},
    VOLUME = {128},
      YEAR = {2024},
    NUMBER = {3},
     PAGES = {1177--1223},
}

@misc{LLR26,
      title={Deformation invariance of canonical nefness in smooth Kahler morphisms}, 
      author={Mu-Lin Li and Xiao-Lei Liu and Sheng Rao},
      year={2026},
      eprint={2609.14435},
      archivePrefix={arXiv},
      primaryClass={math.AG},
      url={https://arxiv.org/abs/2609.14435}, 
}

@book {BS95,
    AUTHOR = {Beltrametti, Mauro C. and Sommese, Andrew J.},
     TITLE = {The adjunction theory of complex projective varieties},
    SERIES = {De Gruyter Expositions in Mathematics},
    VOLUME = {16},
 PUBLISHER = {Walter de Gruyter \& Co., Berlin},
      YEAR = {1995},
     PAGES = {xxii+398},
      ISBN = {3-11-014355-0},
}

@misc{HLX26,
      title={Fujiki Class $\mathcal C$ Varieties and a K\"ahler Criterion}, 
      author={Christopher Hacon and Yi Li and Lingyao Xie},
      year={2026},
      eprint={2608.20588},
      archivePrefix={arXiv},
      primaryClass={math.AG},
      url={https://arxiv.org/abs/2608.20588}, 
}
\end{document}